\documentclass[11pt,reqno]{amsart}

\usepackage[utf8]{inputenc}
\usepackage[T1]{fontenc}
\usepackage[english]{babel}
\usepackage{csquotes}

\usepackage[margin=3cm]{geometry}

\usepackage{xcolor}
\usepackage[colorlinks=true, linkcolor=black, citecolor=black, urlcolor=black]{hyperref}
\usepackage[nameinlink]{cleveref}

\usepackage[onehalfspacing]{setspace}

\numberwithin{equation}{section} 

\usepackage{paralist}
\usepackage{enumitem}
\setlist[enumerate,1]{label = (\alph*),leftmargin=*, topsep=1mm, itemsep=1mm}
\setlist[enumerate,2]{label = (\roman*),leftmargin=*, topsep=1mm, itemsep=1mm}
\setlist[itemize,1]{leftmargin=*, topsep=1mm, itemsep=1mm}

\usepackage{times}
\usepackage{txfonts}
\usepackage{dsfont}
\renewcommand{\mathbb}{\mathds}
\usepackage{yhmath}

\usepackage{amsmath,amssymb}
\usepackage{mathdots}
\usepackage{extpfeil}
\usepackage{extarrows}
\usepackage{oubraces}
\usepackage{array}

\usepackage{tikz}
\usetikzlibrary{cd,arrows,decorations.markings}

\tikzset{
	marrow/.style={decoration={markings,mark=at position 0.75 with {\arrow{#1}}}, postaction=decorate}
}

\tikzcdset{
	arrow style=tikz,
	diagrams={>={Straight Barb[scale=0.9, width=5pt]}} }

\tikzset{epi/.code={\pgfsetarrowsend{Computer Modern Rightarrow[width=5pt, length=3pt] Computer Modern Rightarrow[width=5pt, length=3pt]}}}

\tikzset{epi_mini/.code={\pgfsetarrowsend{Computer Modern Rightarrow[width=5pt, length=3pt, scale=0.85] Computer Modern Rightarrow[width=5pt, length=3pt, scale=0.85]}}}

\usepackage[backend=biber,style=alphabetic,maxbibnames=4,doi=false,isbn=false,url=false]{biblatex}
\bibliography{stabmMor.bib}

\theoremstyle{plain}
\newtheorem{thm}{Theorem}[section]
\newtheorem{prp}[thm]{Proposition}
\newtheorem{cor}[thm]{Corollary}
\newtheorem{lem}[thm]{Lemma}

\newtheorem{thmA}{Theorem}

\theoremstyle{definition}
\newtheorem{dfn}[thm]{Definition}
\newtheorem{ntn}[thm]{Notation}
\newtheorem{con}[thm]{Construction}

\newtheorem{cnv}[thm]{Convention}

\theoremstyle{remark}
\newtheorem{rmk}[thm]{Remark}

\Crefname{subsection}{Subsection}{Subsections}
\Crefname{thm}{Theorem}{Theorems}
\Crefname{thmA}{Theorem}{Theorems}
\Crefname{prp}{Proposition}{Propositions}
\Crefname{cor}{Corollary}{Corollaries}
\Crefname{lem}{Lemma}{Lemmas}
\Crefname{cnj}{Conjecture}{Conjectures}
\Crefname{dfn}{Definition}{Definitions}
\Crefname{ntn}{Notation}{Notations}
\Crefname{con}{Construction}{Constructions}
\Crefname{asn}{Assumption}{Assumptions}
\Crefname{cnv}{Convention}{Conventions}
\Crefname{rmk}{Remark}{Remarks}
\Crefname{exa}{Example}{Examples}

\newcommand{\ZZ}{\mathbb{Z}}
\newcommand{\NN}{\mathbb{N}}

\newcommand{\A}{\mathcal{A}}
\newcommand{\B}{\mathcal{B}}
\newcommand{\C}{\mathcal{C}}
\newcommand{\D}{\mathcal{D}}
\newcommand{\E}{\mathcal{E}}
\newcommand{\F}{\mathcal{F}}
\newcommand{\M}{\mathcal{M}}
\newcommand{\T}{\mathcal{T}}
\newcommand{\U}{\mathcal{U}}
\newcommand{\V}{\mathcal{V}}
\newcommand{\W}{\mathcal{W}}

\newcommand{\mMor}{\mm\Mor}
\newcommand{\smMor}{\sm \Mor}

\newcommand{\ul}{\underline}
\newcommand{\ol}{\overline}
\newcommand{\stab}[1]{\ul{\smash{#1}}}

\newcommand{\myvdots}{\raisebox{0pt}[0.6\height][0.0\height]{$ \vdots $}} 

\newcommand{\sm}[1]{{#1}^\textup{sm}}
\newcommand{\mm}[1]{{#1}^\textup{m}}

\newcommand{\Vect}{\mathrm{Vect}}
\newcommand{\lo}[1]{{}^\bot#1}
\newcommand{\ro}[1]{#1^\bot}

\newcommand{\epic}[2]{\begin{tikzcd}[cramped, sep=small, ampersand replacement=\&] #1 \ar[r, two heads] \& #2 \end{tikzcd}}
\newcommand{\monic}[2]{\begin{tikzcd}[cramped, sep=small, ampersand replacement=\&] #1 \ar[r, tail] \& #2 \end{tikzcd}}
\newcommand{\inj}[2]{\begin{tikzcd}[cramped, sep=small, ampersand replacement=\&] #1 \ar[r, hook] \& #2 \end{tikzcd}}

\DeclareMathOperator{\id}{id}
\DeclareMathOperator{\Hom}{Hom}
\DeclareMathOperator{\Mor}{Mor}
\DeclareMathOperator{\Proj}{Proj}
\DeclareMathOperator{\Inj}{Inj}

\title[Structure of stable monomorphism categories]{On the triangulated structure\\of stable monomorphism categories}

\author[J.~Frank]{Jonas Frank}
\address{\linebreak
	Jonas Frank\\
	Department of Mathematics, RPTU University Kaiserslautern-Landau\\ 
	67663 Kaiserslautern\\
	Germany
}
\email{\href{mailto:jfrank@rptu.de}{jfrank@rptu.de}}

\author[M.~Schulze]{Mathias Schulze}
\address{\linebreak
	Mathias Schulze\\
	Department of Mathematics, RPTU University Kaiserslautern-Landau\\ 
	67663 Kaiserslautern\\
	Germany
}
\email{\href{mailto:mschulze@rptu.de}{mschulze@rptu.de}} 

\subjclass[2020]{Primary 18G80; Secondary 18G65, 16G20}

\keywords{Monomorphism category, triangulated category, semiorthogonal decomposition, recollement, mutation, Serre functor.}

\thanks{JF would like to thank Svetlana Makarova for helpful discussions.}

\begin{document}

\begin{abstract}
	We investigate the triangulated structure of stable monomorphism categories (filtered chain categories) over a Frobenius category. The high degree of symmetry of linear quivers leads to a plethora of semiorthogonal decompositions into smaller categories of the same type. These form polygons of recollements, in which a full turn of mutations is a power of a particular auto-equivalence of the stable monomorphism category. A certain power of this auto-equivalence is the square of the suspension functor. We describe the infinite chains of adjoint pairs obtained from the polygons. As an application, we explicate the construction of Bondal and Kapranov for lifting representing objects of dualized hom-functors in our setup.
\end{abstract}

\maketitle

\tableofcontents

\section{Introduction}

 A celebrated result of Buchweitz's gives a triangle equivalence between the singularity category of a Gorenstein ring $R$ and the stable category of maximal Cohen-Macaulay, that is, Gorenstein projective $R$-modules, see {\cite[Thm.~4.4.1]{Buc21}}. Over a hypersurface ring $R=S/\langle f \rangle$, the latter category is triangle equivalent to the homotopy category of matrix factorizations of $f$ (with two factors) due to Eisenbud's Theorem, see {\cite[Thm.~7.4]{Yos90}}.\\
 In {\cite[Thm.~A]{FS24}} and {\cite[Thms.~4.12, 5.3]{BM24}}, Buchweitz's theorem was generalized, replacing the category of Gorenstein projective modules by the monomorphism category $\mMor_{l}(\F)$, consisting of chains
 \[
 \begin{tikzcd}
 	X=(X, \alpha)\colon \; X^0 \ar[r, tail, "\alpha^0"] & X^1 \ar[r, tail, "\alpha^1"] & \cdots \ar[r, tail, "\alpha^{l-1}"] & X^l
 \end{tikzcd}
 \]
 of $l$ admissible monomorphisms over a Frobenius subcategory $\F$ of an exact category $\E$. This is a fully exact subcategory of the category $\Mor_{l}(\E)$ of chains of $l$ morphisms in $\E$. 
 For the subcategory $\F$ of Gorenstein projectives of suitable module categories $\E$, $\mMor_{l}(\F)$ is the subcategory of Gorenstein projectives of $\Mor_{l}(\E)$, see {\cite[Cor.~3.6]{JK11}} and {\cite[Thm.~4.1]{SZ26}}. The triangle equivalence with matrix factorizations generalizes to factorizations with $l+2$ factors, see {\cite[Thm.~4.6]{SZ26}}. Notably, the (homotopy) category of such factorizations has a cyclic symmetry.\\
 In this article, we investigate the triangulated structure of the above mentioned stable category $\M_l := \stab \mMor_{l}(\F)$. Our approach is inspired by the work of Iyama, Kato, and Miyachi on the homotopy category of complexes over split monomorphism categories, see {\cite[\S4]{IKM16}}.
 We consider two operations relating chains of monomorphisms of different lengths: By \emph{contraction} of an interval $[s,t] \subseteq \{0,\dots,l\}$, a chain $X = (X,\alpha) \in \mMor_{l}(\F)$ as above is sent to 
\[\begin{tikzcd}[cramped] \gamma^{[s,t]}(X) \colon \; X^0 \ar[r, tail, "\alpha^0"] & \cdots \ar[r, tail, "\alpha^{s-2}"] & X^{s-1} \ar[r, tail, "\alpha^t \cdots \alpha^{s-1}"] & X^{t+1} \ar[r, tail, "\alpha^{t+1}"] & \cdots \ar[r, tail, "\alpha^{l-1}"] & X^l \end{tikzcd} \in \mMor_{l-t+s-1}(\F),\]
by \emph{expansion} with identities to 
\[\begin{tikzcd}[cramped] \delta^{[s,t]}(X) \colon \; X^0 \ar[r, tail, "\alpha^0"] & \cdots \ar[r, tail, "\alpha^{s-1}"] & X^s \ar[r, equal] & \cdots \ar[r, equal] & X^s \ar[r, tail, "\alpha^s"] & \cdots \ar[r, tail, "\alpha^{l-1}"] & X^l \end{tikzcd} \in \mMor_{l+t-s}(\F).\]
Both assignments give rise to exact functors inducing triangulated functors $\stab \gamma^{[s,t]}$ and $\stab \delta^{[s,t]}$ between the respective stable categories. The kernel $\Gamma^{[s,t]}$ of $\stab \gamma^{[s,t]}$ and the image $\Delta^{[s,t]}$ of $\stab \delta^{[s,t]}$ (with $l$ replaced by $l-t+s$) are triangulated subcategories of $\stab \mMor_{l}(\F)$. These occur in various \emph{semiorthogonal decompositions}:

\begin{thmA} \label{thmA: SODs}
	The category $\stab \mMor_{l}(\F)$ admits the following semiorthogonal decompositions:
	\begin{enumerate}
		\item $\left(\Gamma^{[s+1, l]}, \Gamma^{[0, s]}\right)$ if $s < l$,
		\item $\left(\Gamma^{[s, t]}, \Delta^{[s, t+1]}\right)$ if $t < l$,
		\item $\left(\Delta^{[s-1, t]}, \Gamma^{[s, t]}\right)$ if $s > 0$.
	\end{enumerate}
\end{thmA}

From these decompositions, we derive further results. Two semiorthogonal decompositions $(\U, \V)$ and $(\V, \W)$ define a \emph{recollement}, see {\cite[Prop.~1.2]{IKM16}}. As a direct consequence, we find a family of \emph{polygons of recollements}:

\begin{thmA} \label{thmA: polygon}
There are the following $(2l+4)$-gons of recollements in $\stab \mMor_{l}(\F)$ for $s < l$:
	\[
		\begin{tikzcd}[sep=tiny]
			& \Delta^{[0, s+1]} \ar[r, -, marrow=>] & \cdots \ar[r, -, marrow=>] & \Gamma^{[k, k+s]} \ar[r, -, marrow=>] & \Delta^{[k, k+s+1]} \ar[r, -, marrow=>] & \Gamma^{[k+1,k+s+1]} \ar[r, -, marrow=>] & \cdots \ar[r, -, marrow=>] & \Gamma^{[l-s, l]} \ar[rd, -, marrow=>, bend left=5mm] & \\
			\Gamma^{[0, s]} \ar[ru, -, marrow=>, bend left=5mm] &&&&&&&& \Gamma^{[0, l-s-1]} \ar[ld, -, marrow=>, bend left=5mm] \\
			&\Gamma^{[s+1, l]} \ar[lu, -, marrow=>, bend left=5mm] & \cdots \ar[l, -, marrow=>] & \Delta^{[k+1, k+l-s+1]} \ar[l, -, marrow=>] & \Gamma^{[k+1, k+l-s]} \ar[l, -, marrow=>] & \Delta^{[k, k+l-s]} \ar[l, -, marrow=>] & \cdots \ar[l, -, marrow=>] & \Delta^{[0, l-s]} \ar[l, -, marrow=>]
		\end{tikzcd}
	\]
\end{thmA}

We describe the \emph{left mutations} $L$ defined by the recollements in these polygons explicitly in terms of an auto-equivalence $\Theta$ of $\stab \mMor_{l}(\F)$, see \Cref{con: theta}, which corresponds to the rotation functor for matrix factorizations with $l+2$ factors. The result can be visualized by two nested polygons of triangle equivalences, see \Cref{fig: poly-mut}, where $\stab \delta^{[s,t]^c}$ is another type of expansion functor with image $\Gamma^{[s,t]}$, see \Cref{con: delta-compl}.

\begin{figure}[h] 
	\begin{footnotesize}
		\begin{center}
			\begin{tikzpicture}[commutative diagrams/every diagram]
				\node (P1) at (90:4.5cm) {$\Gamma^{[0, s]}$};
				\node (P2) at (90-180/11:4.5cm) {$\Delta^{[0, s+1]}$} ;
				\node (P3) at (90-2*180/11:4.5cm) {$\Gamma^{[1, s+1]}$};
				\node (P4) at (90-3*180/11:4.5cm) {$\Delta^{[1, s+2]}$};
				\node (P5) at (90-4*180/11:4.5cm) {\rotatebox{120}{$\cdots$}};
				\node (P6) at (90-5*180/11:4.5cm) {$\Gamma^{[k, s+k]}$};
				\node (P7) at (90-6*180/11:4.5cm) {$\Delta^{[k, s+k+1]}$};
				\node (P8) at (90-7*180/11:4.5cm) {\rotatebox{60}{$\cdots$}};
				\node (P9) at (90-8*180/11:4.5cm) {$\Gamma^{[l-s-1, l-1]}$};
				\node (P10) at (90-9*180/11:4.5cm) {$\Delta^{[l-s-1, l]}$};
				\node (P11) at (90-10*180/11:4.5cm) {$\Gamma^{[l-s, l]}$};
				\node (P12) at (90-11*180/11:4.5cm) {$\Gamma^{[0, l-s-1]}$};
				\node (P13) at (90-12*180/11:4.5cm) {$\Delta^{[0, l-s]}$};
				\node (P14) at (90-13*180/11:4.5cm) {$\Gamma^{[1, l-s]}$};
				\node (P15) at (90-14*180/11:4.5cm) {$\Delta^{[1, l-s+1]}$};
				\node (P16) at (90-15*180/11:4.5cm) {\rotatebox{120}{$\cdots$}};
				\node (P17) at (90-16*180/11:4.5cm) {$\Gamma^{[k, l-s+k-1]}$};
				\node (P18) at (90-17*180/11:4.5cm) {$\Delta^{[k, l-s+k]}$};
				\node (P19) at (90-18*180/11:4.5cm) {\rotatebox{60}{$\cdots$}};
				\node (P20) at (90-19*180/11:4.5cm) {$\Gamma^{[s, l-1]}$};
				\node (P21) at (90-20*180/11:4.5cm) {$\Delta^{[s, l]}$};
				\node (P22) at (90-21*180/11:4.5cm) {$\Gamma^{[s+1, l]}$};
				\node (I1) at (90:2.5cm) {$\M_s$};
				\node (I3) at (90-2*180/11:2.5cm) {$\M_s$};
				\node (I6) at (90-5*180/11:2.5cm) {$\M_s$};
				\node (I9) at (90-8*180/11:2.5cm) {$\M_s$};
				\node (I11) at (90-10*180/11:2.5cm) {$\M_s$};
				\node (I13) at (90-12*180/11:2.5cm) {$\M_s$};
				\node (I15) at (90-14*180/11:2.5cm) {$\M_s$};
				\node (I18) at (90-17*180/11:2.5cm) {$\M_s$};
				\node (I21) at (90-20*180/11:2.5cm) {$\M_s$};
				\node (O2) at (90-180/11:6.5cm) {$\M_{l-s-1}$};
				\node (O4) at (90-3*180/11:6.5cm) {$\M_{l-s-1}$};
				\node (O7) at (90-6*180/11:6.5cm) {$\M_{l-s-1}$};
				\node (O10) at (90-9*180/11:6.5cm) {$\M_{l-s-1}$};
				\node (O12) at (90-11*180/11:6.5cm) {$\M_{l-s-1}$};
				\node (O14) at (90-13*180/11:6.5cm) {$\M_{l-s-1}$};
				\node (O17) at (90-16*180/11:6.5cm) {$\M_{l-s-1}$};
				\node (O20) at (90-19*180/11:6.5cm) {$\M_{l-s-1}$};
				\node (O22) at (90-21*180/11:6.5cm) {$\M_{l-s-1}$};
				\path[commutative diagrams/.cd, every arrow, every label]
				
				(P1) edge [bend right=10mm] node[below] {$L$} (P3)
				(P3) edge [bend right=10mm] node[left] {$L^{\{k-1\}}$} (P6)
				(P6) edge [bend right=10mm] node[left] {$L^{\{l-s-k-1\}}$} (P9)
				(P9) edge [bend right=10mm] node[above] {$L$} (P11)
				(P11) edge [bend right=10mm] node[above] {$L$} (P13)
				(P13) edge [bend right=10mm] node[above] {$L$} (P15)
				(P15) edge [bend right=10mm] node[right] {$L^{\{k-1\}}$} (P18)
				(P18) edge [bend right=10mm] node[right] {$L^{\{s-k\}}$} (P21)
				(P21) edge [bend right=10mm] node[below] {$L$} (P1)
				
				(P2) edge [bend left=20mm] node[above] {$L$} (P4)
				(P4) edge [bend left=20mm] node[right] {$L^{\{k-1\}}$} (P7)
				(P7) edge [bend left=20mm] node[right] {$L^{\{l-s-k-1\}}$} (P10)
				(P10) edge [bend left=20mm] node[below] {$L$} (P12)
				(P12) edge [bend left=20mm] node[below] {$L$} (P14)
				(P14) edge [bend left=20mm] node[left] {$L^{\{k-1\}}$} (P17)
				(P17) edge [bend left=20mm] node[left] {$L^{\{s-k\}}$} (P20)
				(P20) edge [bend left=20mm] node[above] {$L$} (P22)
				(P22) edge [bend left=20mm] node[above] {$L$} (P2)
				
				(I1) edge [bend right=10mm] node[below] {$\Theta$} (I3)
				(I3) edge [bend right=10mm] node[left] {$\Theta^{\{k-1\}}$} (I6)
				(I6) edge [bend right=10mm] node[left] {$\Theta^{\{l-s-k-1\}}$} (I9)
				(I9) edge [bend right=10mm] node[above] {$\Theta$} (I11)
				(I11) edge [bend right=10mm] node[above] {$\id$} (I13)
				(I13) edge [bend right=10mm] node[above] {$\id$} (I15)
				(I15) edge [bend right=10mm] node[right] {$\id$} (I18)
				(I18) edge [bend right=10mm] node[right] {$\id$} (I21)
				(I21) edge [bend right=10mm] node[below] {$\id$} (I1)
				
				(O2) edge [bend left=10mm] node[above] {$\id$} (O4)
				(O4) edge [bend left=10mm] node[right] {$\id$} (O7)
				(O7) edge [bend left=10mm] node[right] {$\id$} (O10)
				(O10) edge [bend left=10mm] node[below] {$\id$} (O12)
				(O12) edge [bend left=10mm] node[below] {$\Theta$} (O14)
				(O14) edge [bend left=10mm] node[left] {$\Theta^{\{k-1\}}$} (O17)
				(O17) edge [bend left=10mm] node[left] {$\Theta^{\{s-k\}}$} (O20)
				(O20) edge [bend left=10mm] node[above] {$\Theta$} (O22)
				(O22) edge [bend left=10mm] node[above] {$\id$} (O2)
				
				(I1) edge node {$\stab \delta^{[0, s]^c}$} (P1)
				(I3) edge node[yshift=-2.5pt] {$\stab \delta^{[1, s+1]^c}$} (P3)
				(I6) edge node[xshift=5pt] {$\stab \delta^{[k, s+k]^c}$} (P6)
				(I9) edge node[left, xshift=-2.5pt] {$\stab \delta^{[l-s-1, l-1]^c}$} (P9)
				(I11) edge node {$\stab \delta^{[l-s, l]^c}$} (P11)
				(I13) edge node {$\stab \delta^{[0, l-s]}$} (P13)
				(I15) edge node {$\stab \delta^{[1, l-s+1]}$} (P15)
				(I18) edge[below, yshift=-2.5pt] node {$\stab \delta^{[k, l-s+k]}$} (P18)
				(I21) edge node {$\stab \delta^{[s, l]}$} (P21)
				
				(O2) edge node {$\stab \delta^{[0, s+1]}$} (P2)
				(O4) edge node {$\stab \delta^{[1, s+2]}$} (P4)
				(O7) edge node[below, yshift=-1pt] {$\stab \delta^{[k, s+k+1]}$} (P7)
				(O10) edge node {$\stab \delta^{[l-s-1, l]}$} (P10)
				(O12) edge node {$\stab \delta^{[0, l-s-1]^c}$} (P12)
				(O14) edge node {$\stab \delta^{[1, l-s]^c}$} (P14)
				(O17) edge node[above] {$\stab \delta^{[k, l-s+k-1]^c}$} (P17)
				(O20) edge node {$\stab \delta^{[s, l-1]^c}$} (P20)
				(O22) edge node {$\stab \delta^{[s+1, l]^c}$} (P22);
			\end{tikzpicture}
		\end{center}
	\end{footnotesize}
	\caption{Polygons of recollements and mutations} \label{fig: poly-mut}
\end{figure}
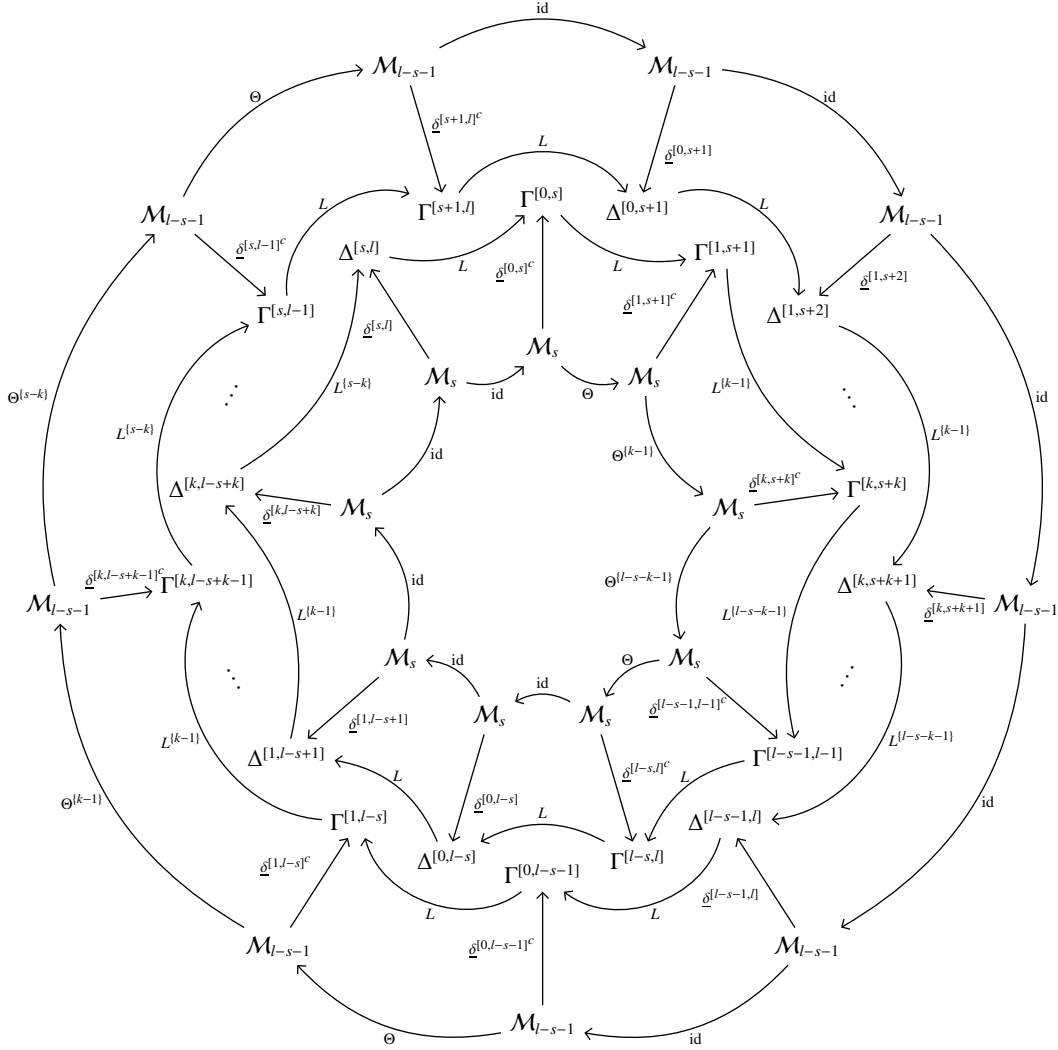

 A full turn of left mutations in each such polygons corresponds to a certain power of $\Theta$. We relate such powers to the suspension functor $\Sigma$ of $\stab \mMor_{l}(\F)$, which shows, in particular, that $\Sigma^2 \cong \id$ for matrix factorizations with $l+2$ factors, see {\cite[Prop.~5.3]{Tri21}}:

\begin{thmA} \label{thmA: Sigma-Theta}
There is an isomorphism $\Theta^{l+2} \cong \Sigma^2$ of endofunctors on $\stab \mMor_{l}(\F)$. 
\end{thmA}

There is a triangle equivalence between $\stab \mMor_{l}(\F)$ and the homotopy category of acyclic $(l+2)$-complexes over $\F$ with projective objects, see {\cite[Thm.~A]{FS24}}, under which $\Theta$ corresponds to the shift functor and \Cref{thmA: Sigma-Theta} to {\cite[Thm.~2.4]{IKM17}}.\\
We include two further types of contraction functors $\stab {\hat \gamma}^{[s,t]^c}$ and $\stab {\check \gamma}^{[s,t]^c}$, see \Cref{con: gamma-compl}, to form \emph{infinite adjoint chains}:
\begin{thmA} \label{thmA: adjoint-chain}
	There are the following infinite adjoint chains:
	{\setlength{\jot}{10pt}
		\begin{gather*}
			\cdots \dashv \Theta^{t-s+1} \stab {\hat \gamma}^{[0, l-t+s-1]^c}\\
			\dashv \stab \delta^{[0, l-t+s-1]^c} \Theta^{-t+s-1} \dashv \cdots \dashv \Theta \stab {\hat \gamma}^{[t-s, l-1]^c} \dashv \stab \delta^{[t-s, l-1]^c} \Theta^{-1} \dashv \stab {\hat \gamma}^{[t-s+1, l]^c} \dashv \stab \delta^{[t-s+1, l]^c} \\
			\dashv \stab \gamma^{[0, t-s]} \dashv \cdots \dashv \stab \gamma^{[s-1,t-1]} \dashv \stab \delta^{[s-1, t]} \dashv \stab \gamma^{[s, t]} \dashv \stab \delta^{[s, t+1]} \dashv \stab \gamma^{[s+1,t+1]} \dashv \cdots \dashv \stab \gamma^{[l-t+s, l]} \\
			\dashv \stab \delta^{[0, l-t+s-1]^c} \dashv \stab{\check \gamma}^{[0, l-t+s-1]^c} \dashv \stab \delta^{[1, l-t+s]^c} \Theta \dashv \Theta^{-1} \stab{\check \gamma}^{[1, l-t+s]^c} \dashv \cdots \dashv \stab \delta^{[t-s+1, l]^c} \Theta^{t-s+1} \\
			\dashv \Theta^{-t+s-1} \stab{\check \gamma}^{[t-s+1, l]^c} \dashv \cdots,
		\end{gather*}
	}
	where $\stab {\hat \gamma}^{[0, l-t+s-1]^c}=\stab { \gamma}^{[l-t+s,l]}$ and $\stab{\check \gamma}^{[t-s+1, l]^c}=\stab \gamma^{[0,t-s]}$.
\end{thmA}

As an application of our results, we explicitly describe \emph{representing objects} of \emph{dualized hom-functors} on $\stab \mMor_{l}(\F)$ lifted from $\stab \F$ using the construction of Bondal and Kapranov:

\begin{thmA} \label{thmA: rep}
	Let $\F$ be a Frobenius category, linear over a field, such that $\stab \F$ is hom-finite. Suppose that the dualized hom-functors $\Hom(A, -)^\ast \colon \stab \F \to \Vect$ for all $A \in \F$ are representable. Then the dualized hom-functor $\Hom(X, -)^\ast \colon \stab \mMor_{l}(\F) \to \Vect$ for any $X = (X, \alpha) \in \mMor_{l}(\F)$ is representable by an object $\tilde X$, obtained as the rightmost column of any diagram
	\[
		\begin{tikzcd}[sep={17.5mm,between origins}, ampersand replacement=\&]
			\tilde X^{0,0} \ar[r] \ar[d, tail] \ar[rd, phantom, "\square"] \& \tilde X^{0,1} \ar[r] \ar[d, tail, ] \ar[rd, phantom, "\square"] \& \tilde X^{0,2} \ar[r] \ar[d, tail] \& \cdots \ar[r] \ar[rd, phantom, "\square"] \& \tilde X^{0,l-2} \ar[r] \ar[d, tail] \ar[rd, phantom, "\square"] \& \tilde X^{0,l-1} \ar[r] \ar[d, tail] \ar[rd, phantom, "\square"] \& \tilde X^{0,l} \ar[d, tail] \\
			I^0 \ar[r] \& \tilde X^{1,1} \ar[r] \ar[d, tail] \ar[rd, phantom, "\square"] \& \tilde X^{1,2} \ar[r] \ar[d, tail] \ar[ru, phantom, "\square"] \& \cdots \ar[r] \ar[rd, phantom, "\square"] \& \tilde X^{1,l-2} \ar[r] \ar[d, tail] \ar[rd, phantom, "\square"] \& \tilde X^{1,l-1} \ar[r] \ar[d, tail] \ar[rd, phantom, "\square"] \& \tilde X^{1,l} \ar[d, tail] \\
			\& I^1 \ar[r] \& \tilde X^{2,2} \ar[r] \ar[d, tail] \ar[ru, phantom, "\square"] \& \cdots \ar[r] \& \tilde X^{2,l-2} \ar[r] \ar[d, tail] \& \tilde X^{2,l-1} \ar[r] \ar[d, tail] \& \tilde X^{2,l} \ar[d, tail] \\
			\&\& I^2 \ar[r] \ar[ru, phantom, "\square"] \& \cdots \& \myvdots \ar[d, tail] \ar[rd, phantom, "\square"] \ar[ru, phantom, "\square"] \& \myvdots \ar[d, tail] \ar[rd, phantom, "\square"] \ar[ru, phantom, "\square"] \& \myvdots \ar[d, tail] \\
			\&\&\&\& I^{l-2} \ar[r] \& \tilde X^{l-1,l-1} \ar[r] \ar[d, tail] \ar[rd, phantom, "\square"] \& \tilde X^{l-1,l} \ar[d, tail] \\
			\&\&\&\&\& I^{l-1} \ar[r] \& \tilde X^{l,l}
		\end{tikzcd}
	\]
	of bicartesian squares in $\F$ with $I^j \in \Inj(\F)$, where the first row consists of the representing objects $\tilde X^{0,j}$ of $\Hom(X^j, -)^\ast$ and the morphisms corresponding to $\alpha^j$.
\end{thmA}

In particular, if $\stab \F$ admits a Serre functor sending the diagram in $\stab \F$ defined by any $X \in \mMor_{l}(\F)$ to the one defined by the first row of the diagram in \Cref{thmA: rep}, then $\stab \mMor_{l}(\F)$ admits a Serre functor, sending $X$ to $\tilde X$, see {\cite[Prop.~3.4.(a)]{BK89}}. It remains to describe this Serre functor on morphisms.

\smallskip

\Cref{thmA: SODs} summarizes \Cref{prp: SOD-Gamma-Gamma,prp: SODs} of the main part. \Cref{thmA: polygon,thmA: adjoint-chain,thmA: Sigma-Theta} correspond to \Cref{cor: polygon}, \Cref{prp: Theta-Sigma-mMor}, and \Cref{thm: adj-chain}. \Cref{thmA: rep} joins \Cref{thm: mMor-rep} and \Cref{lem: mMor-hyp}.

\section{Triangulated categories of monomorphisms}

In this section, we review preliminaries on monomorphism categories in the context of exact and triangulated categories.

\smallskip

Unless stated otherwise, all (sub)categories and functors considered are assumed to be (full) additive.\\
Our main reference on the topic of \emph{triangulated categories} is Neeman's book {\cite{Nee01}}. However, we require the more general definition of a triangulated category whose suspension functor is only an auto-equivalence instead of an automorphism. These two definitions agree up to a triangulated equivalence, see {\cite[\S 2]{KV87}} and {\cite[\S 2]{May01}}.\\
Recall that a \emph{triangle equivalence} is a triangulated functor which is an equivalence of categories. Its quasi-inverse is automatically a triangulated functor, see {\cite[Prop.~1.4]{BK89}} for a more general statement.

\begin{dfn} \label{dfn: sod}
	A pair $(\U, \V)$ of triangulated subcategories of a triangulated category $\T$ is called a \textbf{semiorthogonal decomposition} of $\T$ if
	\begin{enumerate}
		\item \label{dfn: sod-1} $\Hom_\T(\U, \V) = 0$ and
		\item \label{dfn: sod-2} each $T \in \T$ fits into a distinguished triangle $U \to T \to V$, where $U \in \U$ and $V \in \V$.
	\end{enumerate}
\end{dfn}

\begin{rmk}\label{rmk: Hom=0}
	Consider two triangulated subcategories $\U$ and $\V$ of $\T$ with $\Hom_\T(\U, \V) = 0$. Then any solid diagram 
	\[
		\begin{tikzcd}[sep={17.5mm,between origins}]
			U \ar[r] \ar[d, dashed] & X \ar[r] \ar[d, "f"]& V \ar[d, dashed] \\
			U' \ar[r] & X' \ar[r] & V'
		\end{tikzcd}
	\]
	in $\T$ with $U, U' \in \U$ and $V, V' \in \V$ whose rows are distinguished triangles extends uniquely by dashed arrows to a commutative diagram.
\end{rmk}

This leads to the following

\begin{prp}[{\cite[Prop.~1.2]{IKM11}}] \label{prp: SOD-adjoints}
	Let $(\U, \V)$ be a semiorthogonal decomposition of a triangulated category $\T$. Then the inclusion functors $i_!\colon \U \to \T$ and $j_\ast\colon \V \to \T$ have a respective right adjoint $i^!\colon \T \to \U$ and left adjoint $j^\ast\colon \T \to \V$. They are given by fixing for each $X \in \T$ a distinguished triangle
	\[
	\begin{tikzcd} i_! i^! X \ar[r] & X \ar[r] & j_\ast j^\ast X \end{tikzcd}
	\]
	in $\T$, whose morphisms are then given by the respective unit and counit. These adjoints induce triangle equivalences $\T/\V \to \U$ and $\T/\U \to \V$, quasi-inverse to the composed canonical functors $\inj{\U}{\T}\to \T/\V$ and $\inj{\V}{\T} \to \T/\U$, respectively. \qed
\end{prp}

\begin{ntn}
	Adjoint functors $L \dashv R$ between categories $\C$ and $\D$ will be displayed as
	\[
		\begin{tikzcd}
			\D \ar[rr, bend right=7.5mm, "R"'] && \C, \ar[ll, bend right=7.5mm, "L"']
		\end{tikzcd}
	\]
	where the left adjoint $L$ is always the upper arrow, the right adjoint $R$ the lower arrow.
\end{ntn}

\begin{dfn} \label{dfn: mutations}
	Due to \Cref{prp: SOD-adjoints}, two semiorthogonal decompositions $(\U, \V)$ and $(\V, \W)$ of a triangulated category $\T$ patch together to form a \textbf{recollement}, see {\cite[Prop.~1.2]{IKM17}},
	\[
		\begin{tikzcd}[sep={15mm,between origins}]
			&&& \U \ar[ld, hook, bend right=5mm] \\
			\V \ar[rr, hook] && \T \ar[rr, epi] \ar[ru, bend right=5mm] \ar[rd, bend left=5mm] \ar[ll, bend left=10mm] \ar[ll, bend right=10mm] && \T/\V \ar[lu, ->, "\simeq"'] \ar[ld, ->, "\simeq"] \\
			&&& \W. \ar[lu, hook, bend left=5mm]
		\end{tikzcd}
	\]
	 The composed triangle equivalence 
	 \[L := L_\V\colon\U \hookrightarrow \T \to \W =: L_\V(\U) =: L(\U)\] is the \textbf{left mutation} of $\U$ through $\V$, its quasi-inverse
	\[R := R_\V\colon\W \hookrightarrow \T \to \U =: R_\V(\W) =: R(\W)\] the \textbf{right mutation} of $\W$ through $\V$.
\end{dfn}

\begin{rmk} \label{rmk: mutations}
	In \Cref{dfn: mutations}, the composition $\inj{\U \xrightarrow{L} \W}{\T}$ is right adjoint to $\T \to \U$, the composition $\inj{\W \xrightarrow{R} \U}{\T}$ is left adjoint to $\T \to \W$.
\end{rmk}

\begin{dfn}[{\cite[Def.~1.3]{IKM16}}]
An \textbf{$\boldsymbol n$-gon of recollements} of a triangulated category $\T$ consists of $n \geq 2$ triangulated subcategories $\U_1,\dots,\U_n$ such that $(\U_i, \U_{i+1})$ is a semiorthogonal decomposition for all $i \in \ZZ/n\ZZ$.
\end{dfn}

Prominent examples of triangulated categories, dubbed \emph{algebraic} by Keller {\cite[\S3.6]{Kel06}}, are the stable categories of Frobenius (exact) categories, see {\cite[\S2]{Hap88}}. Our main reference on the topic of \emph{exact categories} is Bühler's expository article {\cite{Buh10}}. We review selected material on these types of categories. Admissible monics and epics are represented by $\monic{\!}{\!}$ and $\epic{\!}{\!}$, respectively.

\begin{prp}[{\cite[Prop.~2.9]{Buh10}}] \label{prp: Buehler2.9}
	In an exact category, finite direct sums of short exact sequences are again short exact. In particular, any split short exact sequence is short exact. \qed
\end{prp}

\begin{prp}[{\cite[Prop.~2.12]{Buh10}}] \label{prp: Buehler2.12} \
	\begin{enumerate}[leftmargin=*]
		\item \label{prp: Buehler2.12-push} For a square
		\[
			\begin{tikzcd}[sep={17.5mm,between origins}]
				A \ar[r, tail, "i"] \ar[d, "f"] & B \ar[d, "f'"] \\
				A' \ar[r, tail, "i'"] & B'
			\end{tikzcd}
		\]
		in an exact category, the following statements are equivalent:
		\begin{enumerate}[label=(\arabic*)]
			\item \label{prp: Buehler2.12-push-1} The square is a pushout.
			\item \label{prp: Buehler2.12-push-2} The square is bicartesian.
			\item \label{prp: Buehler2.12-push-3} The sequence \begin{tikzcd}[sep=large] A \ar[r, "\begin{pmatrix} i \\ -f \end{pmatrix}", tail] & B \oplus A' \ar[r, "\begin{pmatrix} f' \hspace{2mm} i' \end{pmatrix}", two heads] & B' \end{tikzcd} is short exact.
			\item \label{prp: Buehler2.12-push-4} The square is part of a commutative diagram
			\[
				\begin{tikzcd}[sep={17.5mm,between origins}]
					A \ar[r, tail, "i"] \ar[d, "f"] & B \ar[d, "f'"] \ar[r, two heads] & C \ar[d, equal] \\
					A' \ar[r, tail, "i'"] & B' \ar[r, two heads] & C.
				\end{tikzcd}
			\]
		\end{enumerate}
		
		\smallskip
		
		\item \label{prp: Buehler2.12-pull} For a square
		\[
			\begin{tikzcd}[sep={17.5mm,between origins}]
				A \ar[r, two heads, "p'"] \ar[d, "g'"] & B \ar[d, "g"] \\
				A' \ar[r, two heads, "p"] & B'
			\end{tikzcd}
		\]
		in an exact category, the following statements are equivalent:
		\begin{enumerate}[label=(\arabic*)]
			\item The square is a pullback.
			\item The square is bicartesian.
			\item The sequence \begin{tikzcd}[sep=large] A \ar[r, "\begin{pmatrix} p' \\ g' \end{pmatrix}", tail] & B \oplus A' \ar[r, "\begin{pmatrix} - g \hspace{2mm} p \end{pmatrix}", two heads] & B' \end{tikzcd} is short exact.
			\item The square is part of a commutative diagram
			\[
				\begin{tikzcd}[sep={17.5mm,between origins}]
					K \ar[r, tail] \ar[d, equal] & A \ar[r, two heads, "p'"] \ar[d, "g'"] & B \ar[d, "g"] \\
					K \ar[r, tail] & A' \ar[r, two heads, "p"] & B'.
				\end{tikzcd}
			\]
		\end{enumerate}
	\end{enumerate}
	\qed 
\end{prp}

\begin{lem}[Noether lemma, {\cite[Ex.~3.7]{Buh10}}] \label{lem: Noether}
	Any solid commutative diagram
	\[
		\begin{tikzcd}[sep={17.5mm,between origins}]
			A' \ar[r, tail] \ar[d, tail] & B' \ar[r, two heads] \ar[d, tail] & C' \ar[d, tail, dashed]\\
			A \ar[r, tail] \ar[d, two heads] & B \ar[r, two heads] \ar[d, two heads] & C \ar[d, two heads, dashed]\\
			A'' \ar[r, tail] & B'' \ar[r, two heads] & C'' 
		\end{tikzcd}
	\]
	
	in an exact category with short exact rows and columns can be uniquely completed by a short exact sequence $\begin{tikzcd}[cramped, sep=small]
		C' \ar[r, tail] & C \ar[r, two heads] & C''
	\end{tikzcd}$. \qed
\end{lem}

\begin{dfn}
	A subcategory $\E' \subseteq \E$ of an exact category is called \textbf{(fully) exact} if it is an exact category itself, and if the inclusion functor preserves (and reflects) short exact sequences.\footnote{Bühler uses the term \emph{fully exact} for the stronger notion of extension-closedness, see {\cite[Lem.~10.20]{Buh10}}.} 
\end{dfn}

\begin{prp}[{\cite[Ex.~13.5, Prop.~11.3, Cor.~11.4]{Buh10}}] \label{prp: Proj-Inj}
	The subcategories $\Proj(\E)$ of projective objects and $\Inj(\E)$ of injective objects are fully exact with the split exact structure, where $\Proj(\E)$ is closed under kernels of admissible epics and $\Inj(\E)$ under cokernels of admissible monics.
\end{prp}

\begin{cnv} \label{cnv: I-P}
	Let $\E$ be an exact category.
	\begin{enumerate}
		\item \label{cnv: I-P-i} If $\E$ has enough injectives, we fix for any $A \in \E$ an admissible monic $i=i_A\colon \monic{A}{I(A)}$ with $I(A) \in \Inj(\E)$ and a cokernel $\Sigma A$ of $i_A$. We choose $i_A=\id_A$ whenever $A \in \Inj(\E)$.
		
		\item \label{cnv: I-P-p} If $\E$ has enough projectives, we fix for any $A \in \E$ an admissible epic $p=p_A\colon \epic{P(A)}{A}$ with $P(A) \in \Proj(\E)$ and a kernel $\Sigma^{-1}A$ of $p_A$. We choose $p_A=\id_A$ whenever $A \in \Proj(\E)$.
	\end{enumerate}
\end{cnv}

\begin{con} \label{con: cone}
	Let $\F$ be a Frobenius category. Then $\Sigma$ becomes an endo-functor of the stable category $\stab \F$, defined on representatives of morphisms by a commutative diagram 
	\[
		\begin{tikzcd}[sep={17.5mm,between origins}]
			X \ar[r, tail, "i_X"] \ar[d] & I(X) \ar[r, two heads] \ar[d, dashed] & \Sigma X \ar[d, dashed]\\
			Y \ar[r, tail, "i_Y"] & I(Y) \ar[r, two heads] & \Sigma Y
		\end{tikzcd}
	\]
	in $\F$, independent of choices up to isomorphism of functors, see {\cite[Rem.~2.2]{Hap88}}. The dual diagram makes $\Sigma^{-1}$ an endo-functor of $\stab \F$. For any morphism $f\colon X \to Y$ in $\F$, \Cref{prp: Buehler2.12}.\ref{prp: Buehler2.12-push} yields a pushout diagram
	\begin{equation*}
		\begin{tikzcd}[sep={17.5mm,between origins}] 
			X \ar[r, tail, "i"] \ar[d, "f"] \ar[rd, phantom, "\square"] & I(X) \ar[r, two heads] \ar[d, "g"] & \Sigma X \ar[d, equal] \\
			Y \ar[r, tail, "j"] & C(f) \ar[r, two heads] & \Sigma X
		\end{tikzcd}
	\end{equation*}
	and a short exact sequence
	\begin{equation*}
		\begin{tikzcd}[sep=large]
			X \ar[r, "\begin{pmatrix} i \\ -f \end{pmatrix}", tail] & I(X) \oplus Y \ar[r, "\begin{pmatrix} g \hspace{2mm} j \end{pmatrix}", two heads] & C(f).
		\end{tikzcd}
	\end{equation*}
	The object $C(f) \in \F$ is called a \textbf{cone} of $f$ and the sequence
	\begin{equation*}
		\begin{tikzcd}[sep=large] 
				X \ar[r, "f"] & Y \ar[r] & C(f) \ar[r] & \Sigma X
			\end{tikzcd}
	\end{equation*}
	in $\underline \F$ a \textbf{standard triangle}. Dually, a \textbf{cocone} $C^\ast(f)$ of $f$ fits into a pullback diagram
	\begin{equation*}
		\begin{tikzcd}[sep={17.5mm,between origins}]
			\Sigma^{-1}Y \ar[r, tail] \ar[d, equal] & C^\ast(f) \ar[r, two heads] \ar[d] \ar[rd, phantom, "\square"] & X \ar[d, "f"] \\
			\Sigma^{-1}Y \ar[r, tail] & P(Y) \ar[r, two heads, "p"] & Y
		\end{tikzcd}
	\end{equation*}
	and into a short exact sequence
	\begin{equation} \label{eqn: std-sequence-dual}
		\begin{tikzcd}
			C^\ast(f) \ar[r, tail] & X \oplus P(Y) \ar[r, two heads, "\begin{pmatrix} -f \hspace{2mm} p \end{pmatrix}"] & Y.
		\end{tikzcd}
	\end{equation}
\end{con}

\begin{thm}[{\cite[Rem.~2.2, Thm.~2.6]{Hap88}}] \label{thm: stable-Frobenius}
	The stable category of a Frobenius category is triangulated. The suspension functor is $\Sigma$ with quasi-inverse $\Sigma^{-1}$, the distinguished triangles are the candidate triangles isomorphic to standard triangles, see \Cref{con: cone}. \qed
\end{thm}

\begin{dfn}
	We call a morphism in a Frobenius category $\F$ a \textbf{stable isomorphism} if it is an isomorphism in the stable category $\stab \F$.
\end{dfn}

\begin{lem}[{\cite[Lem.~2.7]{Hap88}}] \label{lem: ses-triangle}
	Any short exact sequence \begin{tikzcd}[cramped, sep=small] X \ar[r, tail, "i"] & Y \ar[r, two heads, "p"] & Z \end{tikzcd}
	in a Frobenius category $\F$ induces a distinguished triangle
	\begin{tikzcd}[cramped, sep=small] 
		X \ar[r, "i"] & Y \ar[r, "p"] & Z
	\end{tikzcd}
	in $\underline \F$. In particular, any morphism $f$ in $\F$ yields a distinguished triangle \begin{tikzcd}[cramped, sep=small]
		C^\ast(f) \ar[r] & X \ar[r, "f"] & Y
	\end{tikzcd} in $\underline \F$, see \eqref{eqn: std-sequence-dual}. \qed
\end{lem}

\begin{prp}[{\cite[Prop.~7.3]{IKM16}}] \label{prp: stable-functor-triang}
	Any exact functor $F\colon \F' \to \F$ of Frobenius categories preserving projective-injectives induces a triangulated functor $\underline F\colon \underline \F' \to \underline \F$ of the respective stable categories.\qed
\end{prp}

We further specialize to monomorphism categories, which are our objects of study.

\begin{ntn}[{\cite[Def.~3.1]{BM24}}, {\cite[Def.~4.1]{IKM17}}] 
	Let $\E$ be an exact category and $l \in \NN$.
	\begin{enumerate}
		\item Let $\Mor_{l}(\E)$ denote the category of diagrams of Dynkin type $A_{l+1}$ in $\E$, where $A_{l+1}$ is the linear quiver
		\[
			\begin{tikzcd} \underset 0 \bullet \ar[r] & \underset 1 \bullet \ar[r] & \cdots \ar[r] & \underset {l-1} \bullet \ar[r] & \underset l \bullet \end{tikzcd}
		\]
		with $l$ arrows. We denote elements of $\Mor_{l}(\E)$ by
		\[
			\begin{tikzcd}
				X=(X, \alpha)\colon \; X^0 \ar[r, "\alpha^0"] & X^1 \ar[r, "\alpha^1"] & \cdots \ar[r, "\alpha^{l-1}"] & X^l.
			\end{tikzcd}
		\]
		By $\mMor_l(\E)$ and $\smMor_l(\E)$ we denote the subcategories of $\Mor_l(\E)$, where all arrows are admissible or split monics, respectively.
		
		\item Given an object $A \in \E$ and $n \in \{1, \dots, l+1\}$, we define the object
		\[
			\begin{tikzcd}
				\mu_n(A)\colon \; 0 \ar[r, equal] & \cdots \ar[r, equal] & 0 \ar[r, tail] & A^{l-n+1} \ar[r, equal] & \cdots \ar[r, equal] & A^{l}
			\end{tikzcd}
		\]
		of $\smMor_{l}(\E)$ by $A^k := A$ for $k \in \{l-n+1, \dots, l\}$.
	\end{enumerate}
\end{ntn}

\begin{rmk}
	For an exact category $\E$ and $l \in \NN$, we have
	\[\mMor_l(\Inj(\E))=\smMor_l(\Inj(\E))=\mMor_l(\E) \cap \Mor_l(\Inj(\E))\] and \[\mMor_l(\Proj(\E)) =\smMor_l(\Proj(\E)) \subseteq \mMor_l(\E) \cap \Mor_l(\Proj(\E)).\]
	Equality holds if $\E$ is Frobenius.
\end{rmk}

\begin{thm}[{\cite[Props.~3.5, 3.9, 3.11, Thm.~3.12]{BM24}}] \label{thm: mMor}
	Let $\E$ be an exact category.
	\begin{enumerate}
		\item \label{thm: mMor-exact} The category $\mMor_l(\E)$ is exact with the termwise exact structure.
		
		\item \label{thm: mMor-Proj} We have $\Proj(\mMor_{l}(\E))=\mMor_l(\Proj(\E))$ and $\Inj(\mMor_{l}(\E))=\mMor_l(\Inj(\E))$.
		
		\item \label{thm: mMor-enough} If $\E$ has enough projectives resp. injectives, then so has $\mMor_l(\E)$.
	\end{enumerate}
 In particular, if $\F$ is a Frobenius category, then so is $\mMor_l(\F)$. \qed
\end{thm}

For later reference, we construct sections that witness the injectivity of objects of $\mMor_l(\Inj(\E))$:

\begin{lem} \label{lem: extend-left-inverse}
	Let $\E$ be an exact category, $l \in \NN$, and $I \in \mMor_{l}(\Inj(\E))$. For an admissible monic $s \colon \monic{I}{X}$ in $\mMor_{l}(\E)$, any left-inverse $r^0 \colon X^0 \to I^0$ of $s^0$ extends to a left-inverse $r \colon X \to I$ of $s$.
\end{lem}

\begin{proof}
	Write $I=(I, \iota)$ and $X=(X, \alpha)$. We may assume that $l=1$. Due to \Cref{lem: Noether}, the admissible monic $s$ gives rise to a commutative diagram
	\[
		\begin{tikzcd}[sep={17.5mm,between origins}]
			I^0 \ar[r, tail, "\iota^0"] \ar[d, tail, "s^0"] & I^1 \ar[d, tail, "s^1"] \ar[r, two heads, "\kappa^1"] \ar[l, bend right=15mm, dotted, "\kappa^0"'] & J^1 \ar[d, dashed, tail, "j^1"] \ar[l, bend right=15mm, dotted, "\iota^1"'] \\
			X^0 \ar[r, tail, "\alpha^0"] \ar[d, two heads] \ar[u, bend left=15mm, "r^0"]& X^1 \ar[r, two heads, "\beta^1"] \ar[d, two heads] \ar[lu, "\tilde r^0"', dotted] \ar[u, bend right=15mm, dotted, "r^1"'] & Y^1 \ar[d, dashed, two heads] \ar[u, bend right=15mm, dotted, "k^1"'] \\
			\bullet \ar[r, tail] & \bullet \ar[r, two heads] & \bullet
		\end{tikzcd}
	\]
	of solid and dashed arrows in $\E$ with short exact rows and columns. By \Cref{prp: Proj-Inj}, we have $J^1 \in \Inj(\E)$. The dotted arrows are now obtained as follows: The upper row and the right-hand column split since $I^0, J^1 \in \Inj(\E)$. Hence, there are a left-inverse $k^1$ of $j^1$ and morphisms $\iota^1$ and $\kappa^0$ such that $\iota^0 \kappa^0 + \iota^1\kappa^1 = \id_{I^1}$. Using $I^0 \in \Inj(\E)$ again, we lift $r^0$ along $\alpha^0$ to obtain $\tilde r^0 \colon X^1 \to I^0$ with $\tilde r^0 \alpha^0 = r^0$. Set $r^1 := \iota^0 \tilde r^0 - \iota^0\tilde r^0 s^1 \iota^1 k^1 \beta^1 +\iota^1 k^1 \beta^1$. Then $r:=(r^0, r^1) \colon X \to I$ is a morphism in $\mMor_{l}(\E)$ since $\beta^1 \alpha^0 = 0$ and hence $r^1 \alpha^0 = \iota^0 \tilde r^0 \alpha^0 = \iota^0 r^0$. Using that $k^1 \beta^1 s^1 = k^1 j^1 \kappa^1 = \kappa^1$, we compute
	\begin{align*}
		r^1 s^1 & = \iota^0 \tilde r^0 s^1 - \iota^0 \tilde r^0 s^1 \iota^1 \circ \left(k^1 \beta^1 s^1 \right) + \iota^1 \circ \left(k^1 \beta^1 s^1\right) = \iota^0 \tilde r^0 s^1 \circ \left(\iota^0 \kappa^0 + \iota^1\kappa^1\right) - \iota^0 \tilde r^0 s^1 \iota^1 \kappa^1 + \iota^1\kappa^1 \\
		&= \iota^0 \tilde r^0 s^1 \iota^0 \kappa^0 + \iota^1\kappa^1 = \iota^0 \tilde r^0 \alpha^0 s^0 \kappa^0 + \iota^1\kappa^1 = \iota^0 r^0 s^0 \kappa^0 + \iota^1\kappa^1 = \iota^0 \kappa^0 + \iota^1\kappa^1 = \id_{I^1},
	\end{align*}
	as desired.
\end{proof}

\begin{ntn}
	For a Frobenius category $\F$ and $l \in \NN$, we denote the stable category of $\mMor_{l}(\F)$ by $\M_l := \stab \mMor_{l}(\F)$. It is a triangulated category, see \Cref{thm: stable-Frobenius}.
\end{ntn}

\begin{con}[Projectives and injectives in $\mMor_{l}(\E)$] \label{con: mMor-enough}
	Let $\E$ be an exact category, $l \in \NN$, and $X=(X, \alpha) \in \mMor_{l}(\E)$. If $\E$ has enough projectives or injectives, we construct $i_X$ and $p_X$ based on $i_A$ and $p_A$, where $A \in \E$, see \Cref{cnv: I-P}:
	\begin{enumerate}[wide]
		\item \label{con: mMor-enough-proj} There is an admissible epic $p_X=(p^k)_{k=0,\dots,l}\colon \epic P X$ in $\mMor_l( \E)$ with $P := \bigoplus_{k=0}^l \mu_{l-k+1} P(X^k)\in \Proj(\mMor_l( \E))$ defined by
		\[p^k := \begin{pmatrix}
			\alpha^{k-1} \cdots \alpha^0 p_{X^0} & \cdots & \alpha^{k-1} p_{X^{k-1}} & p_{X^k}
		\end{pmatrix} \colon P^k := P(X^0) \oplus \dots \oplus P(X^k) \to X^k.\]
		
		\item \label{con: mMor-enough-inj} Consider $i^0 := i_{X^0} \colon \monic{X^0}{I(X^0) =: I^0}$ and form a diagram
		\[
			\begin{tikzcd}[sep={17.5mm,between origins}]
				X^0 \ar[r, tail, "\alpha^0"] \ar[d, tail, "i^0"] \ar[rd, phantom, "\square"] & X^1 \ar[r, equal] \ar[d, tail] & X^1 \ar[r, tail, "\alpha^1"] \ar[d, tail] \ar[rd, phantom, "\square"] & X^2 \ar[r, equal] \ar[d, tail] & \cdots \ar[r, equal] & X^{l-1} \ar[r, tail, "\alpha^{l-1}"] \ar[d, tail] \ar[rd, phantom, "\square"] & X^l \ar[r, equal] \ar[d, tail] & X^l \ar[d, tail] \\
				I^0 \ar[r, tail] \ar[d, two heads] & W^1 \ar[r, tail, "i^1"] \ar[d, two heads] \ar[rd, phantom, "\square"] & I^1 \ar[r, tail] \ar[d, two heads] & W^2 \ar[r, tail, "i^2"] \ar[d, two heads] & \cdots \ar[r, tail, "i^{l-1}"] \ar[rd, phantom, "\square"] & I^{l-1} \ar[r, tail] \ar[d, two heads] & W^l \ar[r, tail, "i^l"] \ar[d, two heads] \ar[rd, phantom, "\square"] & I^l \ar[d, two heads] \\
				Y^0 \ar[r, equal] & Y^0 \ar[r, tail] & Y^1 \ar[r, equal] & Y^1 \ar[r, tail] \ar[ru, phantom, "\square"] & \cdots \ar[r, tail] & Y^{l-1} \ar[r, equal] & Y^{l-1} \ar[r, tail] & Y^l
			\end{tikzcd}
		\]
		of bicartesian squares with $i^k := i_{W^k} \colon \monic{W^k}{I(W^k) =: I^k}$ for $k \in \{1, \dots, l\}$ and short exact columns, see \Cref{prp: Buehler2.12}. Composing horizontal monics to eliminate the columns involving $W^1, \dots, W^l$ yields a short exact sequence \begin{tikzcd}[cramped, sep=small] X \ar[r, tail, "i_X"] & I \ar[r, two heads] & Y \end{tikzcd} in $\mMor_{l}( \E)$ with $I \in \Inj(\mMor_{l}( \E))$ as desired, see \Cref{thm: mMor}.\ref{thm: mMor-Proj}.
	\end{enumerate}
\end{con}

\begin{rmk} \label{rmk: cone-termwise}
	Any pushout along an admissible monic in $\mMor_l(\E)$ for $l \in \NN$ yields termwise pushouts along admissible monics in an exact category $\E$, see \Cref{prp: Buehler2.12}.\ref{prp: Buehler2.12-push} and \Cref{thm: mMor}.\ref{thm: mMor-exact}. The obvious dual statement holds for pullbacks along admissible epics. In particular, for any Frobenius category $\F$, the (co)cone of a morphism in $\mMor_l(\F)$ yields termwise (co)cones in $\F$, see \Cref{con: cone} and \Cref{thm: mMor}.\ref{thm: mMor-Proj}.
\end{rmk}

\section{Contraction and expansion} \label{sec: con-exp}

In this section, we introduce contraction and expansion functors of stable monomorphism categories. Their respective kernels and images are the triangulated subcategories, which form various semiorthogonal decompositions, see \Cref{sec: SOD}. We give a convenient representation of their objects up to special choices of stable isomorphisms.

\begin{dfn} \label{dfn: contraction-expansion}
	Let $\E$ be an exact category, $l \in \NN$, and $s, t \in \{0, \dots, l\}$ with $s \leq t$.
	\begin{enumerate}
		\item \label{dfn: contraction-expansion-gamma} We define the \textbf{contraction} functor $\gamma^{[s, t]}:=\gamma^{[s, t]}_l\colon \mMor_l(\E) \to \mMor_{l-t+s-1}(\E)$ by sending an object $X=(X, \alpha) \in \mMor_l(\E)$ to \[\begin{tikzcd}[cramped] X^0 \ar[r, tail, "\alpha^0"] & \cdots \ar[r, tail, "\alpha^{s-2}"] & X^{s-1} \ar[r, tail, "\alpha^t \cdots \alpha^{s-1}"] & X^{t+1} \ar[r, tail, "\alpha^{t+1}"] & \cdots \ar[r, tail, "\alpha^{l-1}"] & X^l \end{tikzcd},\]
		and a morphism $(f^0, \dots, f^l)$ to $(f^0, \dots, f^{s-1}, f^{t+1} , \dots, f^l)$. We write $\gamma^s := \gamma^s_l := \gamma^{[s,s]}_l$, \[\gamma^{[s, t]^c} := \gamma^{[s, t]^c}_l := \gamma^{[0, s-1]}_t \circ \gamma^{[t+1, l]}_l \colon \mMor_{l}(\E) \to \mMor_{t-s}(\E),\] and $\gamma^{s^c} := \gamma^{s^c}_l := \gamma^{[s,s]^c}_l$.
		\item \label{dfn: contraction-expansion-delta} We define the \textbf{expansion} functor $\delta^{[s,t]}:=\delta^{[s,t]}_l \colon \mMor_l(\E) \to \mMor_{l+t-s}(\E)$ by sending an object $X=(X, \alpha) \in \mMor_{l}(\E)$ to
		\[\begin{tikzcd}[cramped] X^0 \ar[r, tail, "\alpha^0"] & \cdots \ar[r, tail, "\alpha^{s-1}"] & X^s \ar[r, equal] & \cdots \ar[r, equal] & X^s \ar[r, tail, "\alpha^s"] & \cdots \ar[r, tail, "\alpha^{l-1}"] & X^l \end{tikzcd},\]
		and a morphism $(f^0, \dots, f^l)$ to $(f^0, \dots, f^s, \dots, f^s , \dots, f^l)$, where the $X^s$ resp.~$f^s$ are placed at positions $s, \dots, t$. We write $\delta^s := \delta^s_l := \delta^{[s,s +1]}_l$. Note that $\delta^{[s,s]}_l=\id_{\mMor_l(\E)}$.
	\end{enumerate}
\end{dfn}

\begin{rmk} \label{rmk: gamma-delta}
	Let $\E$ be an exact category, $l \in \NN$, and $s, t \in \{0, \dots, l\}$ with $s \leq t$.
	\begin{enumerate}
		\item \label{rmk: delta-gamma-rels}
		We have $\gamma^s_{l-1} \circ \gamma^t_l = \gamma^{t-1}_{l-1} \circ \gamma^s_l$ and $\delta^t_{l+1} \circ \delta^s_l = \delta^s_{l+1} \circ \delta^{t-1}_l$ if $s < t$.
		
		\item \label{rmk: gamma-delta-comp} We have $\gamma^{[s, t]}_l = \gamma^s_{l-t+s} \circ \dots \circ \gamma^{t-1}_{l-1} \circ \gamma^t_l$ and $\delta^{[s,t]}_l = \delta^{t-1}_{l-s+t-1} \circ \dots \circ\delta^{s+1}_{l+1 } \circ \delta^s_l$. These representations are not unique, see \ref{rmk: delta-gamma-rels}.
		 
		\item \label{rmk: gamma-delta-id} We have $\gamma^{[s, t]}_{l+t-s+1} \circ \delta^{[s, t+1]}_l = \id_{\mMor_l(\E)}$ if $t < l$, and $\gamma^{[s+1, t]}_{l+t-s} \circ \delta^{[s, t]}_l = \id_{\mMor_l(\E)}$ if $s < t$. In particular, $\delta^{[s, t]}$ is fully faithful and $\gamma^{[s, t]}$ full.
	\end{enumerate}
\end{rmk}

\begin{rmk} \label{rmk: stab-delta-gamma}
	Let $\F$ be a Frobenius category, $l \in \NN$, and $s, t \in \{0, \dots, l\}$ with $s \leq t$.
	\begin{enumerate}
		\item \label{rmk: stab-delta-gamma-triang} The functors $\gamma^{[s, t]}_l$ and $\delta^{[s, t]}_l$ from \Cref{dfn: contraction-expansion} are exact functors and preserve projectives, see \Cref{thm: mMor}\ref{thm: mMor-Proj}. As such they induce triangulated functors \[\stab \gamma^{[s, t]}:=\stab\gamma^{[s, t]}_l\colon \stab\mMor_l(\F) \to \stab\mMor_{l-t+s-1}(\F) \hspace{2.5mm} \textup{ and } \hspace{2.5mm} \stab \delta^{[s,t]}:= \stab\delta^{[s,t]}_l \colon \stab \mMor_l(\F) \to \stab \mMor_{l+t-s}(\F)\] of stable categories, see \Cref{prp: stable-functor-triang}. We write $\stab \gamma^s := \stab \gamma^s_l := \stab \gamma^{[s,s]}_l$, $\stab \gamma^{[s,t]^c} := \stab \gamma^{[s,t]^c}_l := \stab \gamma^{[0, s-1]}_t \circ \stab \gamma^{[t+1, l]}_l$, $\stab \gamma^{s^c} := \stab \gamma^{s^c}_l := \stab \gamma^{[s,s]^c}_l$, and $\stab \delta^s := \stab \delta^s_l := \stab \delta^{[s,s+1]}$.
		
		\item \label{rmk: stab-delta-gamma-id} Due to \Cref{rmk: gamma-delta}.\ref{rmk: gamma-delta-id}, $\stab \gamma^{[s, t]}_{l+t-s+1} \circ \stab \delta^{[s, t+1]}_l = \id_{\stab \mMor_l(\F)}$ if $t < l$, and $\stab \gamma^{[s+1, t]}_{l+t-s} \circ \stab \delta^{[s, t]}_l = \id_{\stab \mMor_l(\F)}$ if $s < t$. In particular, $\stab \delta^{[s, t]}$ is fully faithful and $\stab \gamma^{[s, t]}$ full.
	\end{enumerate}
\end{rmk}

\begin{dfn}
	Let $\F$ be a Frobenius category, $l \in \NN$, and $s, t \in \{0, \dots, l\}$ with $s \leq t$.
	\begin{enumerate}
		\item By $\Gamma^{[s,t]} := \Gamma^{[s,t]}_l$ we denote the kernel of $\stab \gamma^{[s,t]}_l$, that is, the subcategory of $\stab \mMor_{l}(\F)$ with objects $X$, where $X^k \in \Proj(\F)=\Inj(\F)$ for all $k \in \{0, \dots, s-1, t+1, \dots, l\}$. We set $\Gamma^s := \Gamma^s_l := \Gamma^{[s,s]}_l$.
		\item By $\Delta^{[s,t]} := \Delta^{[s,t]}_l$ we denote the image of $\stab \delta^{[s, t]}_{l-t+s}$.
	\end{enumerate}
	By \Cref{rmk: stab-delta-gamma}.\ref{rmk: stab-delta-gamma-triang}, both $\Gamma^{[s,t]}_l$ and $\Delta^{[s,t]}_l$ are triangulated subcategories of $\stab \mMor_l(\F)$, see {\cite[Lem.~2.1.4]{Nee01}}
\end{dfn}

\begin{lem} \label{lem: Delta-mMor}
	Let $\F$ be a Frobenius category and $s, t \in \{0, \dots, l\}$ with $s \leq t$. Then, $\Delta^{[s,t]}$ is the replete hull of the subcategory of $\stab \mMor_l(\F)$ with objects of the form
	\[
		\begin{tikzcd}
			X^0 \ar[r, tail] & \cdots \ar[r, tail] & X^s \ar[r, equal] & \cdots \ar[r, equal] & X^t \ar[r, tail] & \cdots \ar[r, tail] & X^l.
		\end{tikzcd}
	\]
	The restriction of $\stab \gamma^{[s, t-1]}$ or, equally, $\stab \gamma^{[s+1, t]}$ to $\Delta^{[s, t]}$ is quasi-inverse to the triangle equivalence $\stab \mMor_{l+s-t}(\F) \xrightarrow{\simeq} \Delta^{[s, t]}$ given by $\stab \delta^{[s, t]}$.
\end{lem}

\begin{proof}
	The first claim holds by definition, the second due to Remarks \ref{rmk: stab-delta-gamma}.\ref{rmk: stab-delta-gamma-id} and \ref{rmk: quasi-inverse}.
\end{proof}

\begin{rmk} \label{rmk: quasi-inverse}
	Let $\eta\colon F \to G$ be a natural transformation of functors $F, G \colon \A \to \B$. If $H\colon \B \to \C$ is another functor, then the \emph{whiskering} $H\eta := (H(\eta_X))_{X \in \A}$ of $\eta$ on the right by $H$ is a natural transformation $H F \to H G$. It is an isomorphism if $\eta$ is so. In particular, if $F$ is an equivalence with quasi-inverse $F^{-1}$ and $F': \B \to \A$ a functor with $F' F = \id_\A$, then $F^{-1} = F'FF^{-1} \simeq F'$. So, $F'$ is already a quasi-inverse of $F$.
\end{rmk}

\begin{lem} \label{lem: Gamma-s-t}
	Let $\E$ be an exact category, $l \in \NN$, and $s, t \in \{0, \dots, l\}$ with $s \leq t$. Given $X = (X, \alpha) \in \mMor_{l}(\E)$ with $X^k \in \Inj(\E)$ for all $k \in \{0, \dots, s-1, t+1, \dots, l\}$, there is 
	\begin{enumerate}
		\item \label{lem: Gamma-s-t-epi} a split short exact sequence
		\[\begin{tikzcd}[sep={15mm,between origins}]
			I \ar[d, tail] & X^0 \ar[r, tail, "\alpha^0"] \ar[d, equal] & \cdots \ar[r, tail, "\alpha^{s-2}"] \ar[rd, phantom, "\square"] & X^{s-1} \ar[r, equal] \ar[d, equal] & X^{s-1} \ar[r, equal] \ar[d, tail, "\alpha^{s-1}"] & \cdots \ar[r, equal] & X^{s-1} \ar[r, equal] \ar[d, tail, "\alpha^{t-1} \cdots \alpha^{s-1}"'] & I^{t+1} \ar[r, tail] \ar[d, tail, "\alpha^t \cdots \alpha^{s-1}"'] & \cdots \ar[r, tail] \ar[rd, phantom, "\square"] & I^l \ar[d, tail]\\
			X \ar[d, two heads, "p"] & X^0 \ar[r, tail, "\alpha^0"] \ar[d, two heads, "p^0"] \ar[ru, phantom, "\square"] & \cdots \ar[r, tail, "\alpha^{s-2}"] & X^{s-1} \ar[r, tail, "\alpha^{s-1}"] \ar[d, two heads, "p^{s-1}"'] \ar[rd, phantom, "\square"] & X^s \ar[r, tail, "\alpha^{s}"] \ar[d, two heads, "p^s"] & \cdots \ar[r, tail, "\alpha^{t-1}"] \ar[rd, phantom, "\square"] & X^t \ar[r, tail, "\alpha^t"] \ar[d, two heads, "p^t"] \ar[rd, phantom, "\square"] & X^{t+1} \ar[r, tail, "\alpha^{t+1}"] \ar[d, two heads, "p^{t+1}"] \ar[ru, phantom, "\square"] & \cdots \ar[r, tail, "\alpha^{l-1}"] & X^l \ar[d, two heads, "p^l"] \\
			\tilde X & 0 \ar[r, equal] & \cdots \ar[r, equal] & 0 \ar[r, tail] & \tilde X^s \ar[r, tail] \ar[ru, phantom, "\square"] & \cdots \ar[r, tail] & \tilde X^t \ar[r, tail] & \tilde X^{t+1} \ar[r, equal] & \cdots \ar[r, equal] & \tilde X^{t+1}
		\end{tikzcd}\]
		
		\item \label{lem: Gamma-s-t-mono} with reverse split short exact sequence
		\[\begin{tikzcd}[sep={15mm,between origins}]
			\tilde X \ar[d, tail, "i"] & 0 \ar[r, equal] \ar[d, tail, "i^{0}"] & \cdots \ar[r, equal] & 0 \ar[r, tail] \ar[d, tail, "i^{s-1}"] \ar[rd, phantom, "\square"] & \tilde X^s \ar[r, tail] \ar[d, tail, "i^{s}"] & \cdots \ar[r, tail] \ar[rd, phantom, "\square"] & \tilde X^t \ar[r, tail] \ar[d, tail, "i^{t}"] \ar[rd, phantom, "\square"] & \tilde X^{t+1} \ar[r, equal] \ar[d, tail, "i^{t+1}"] & \cdots \ar[r, equal] & \tilde X^{t+1} \ar[d, tail, "i^{l}"] \\
			X \ar[d, two heads] & X^0 \ar[r, tail, "\alpha^0"] \ar[d, equal] & \cdots \ar[r, tail, "\alpha^{s-2}"] \ar[rd, phantom, "\square"] & X^{s-1} \ar[r, tail, "\alpha^{s-1}"] \ar[d, equal] & X^s \ar[r, tail, "\alpha^s"] \ar[d, two heads] \ar[ru, phantom, "\square"] & \cdots \ar[r, tail, "\alpha^{t-1}"] & X^t \ar[r, tail, "\alpha^t"] \ar[d, two heads] & X^{t+1} \ar[r, tail, "\alpha^{t+1}"] \ar[d, two heads] & \cdots \ar[r, tail, "\alpha^{l-1}"] \ar[rd, phantom, "\square"] & X^l \ar[d, two heads] \\
			I & X^0 \ar[r, tail, "\alpha^0"] \ar[ru, phantom, "\square"] & \cdots \ar[r, tail, "\alpha^{s-2}"] & X^{s-1} \ar[r, equal] & X^{s-1} \ar[r, equal]& \cdots \ar[r, equal] & X^{s-1} \ar[r, equal] & I^{t+1} \ar[r, tail] \ar[ru, phantom, "\square"] & \cdots \ar[r, tail] & I^l
		\end{tikzcd}\]
	\end{enumerate}
	in $\mMor_{l}(\E)$, where $\tilde X^{t+1} \in \Inj(\E)$, $I \in \Inj(\mMor_{l}(\E))$, and $X^{-1}:=0$. The morphism $i^{t+1}$ can be chosen as an arbitrary right-inverse of $p^{t+1}$.\\
	In particular, if $\E$ is Frobenius, then $p$ and $i$ are stable isomorphisms and $\Gamma^{[s,t]}_l$ is the replete hull of the subcategory of $\stab \mMor_l(\E)$ with objects of the form
	\[\begin{tikzcd}[sep={15mm,between origins}]
		0 \ar[r, equal] & \cdots \ar[r, equal] & 0 \ar[r, tail] & X^s \ar[r, tail] & \cdots \ar[r, tail] & X^t \ar[r, tail] & X^{t+1} \ar[r, equal] & \cdots \ar[r, equal] & X^{t+1},
	\end{tikzcd}\]
	where $X^{t+1} \in \Proj(\E)=\Inj(\E)$.
\end{lem}

\begin{proof} \
	\begin{enumerate}[wide]
		\item Due to \Cref{prp: Buehler2.12}.\ref{prp: Buehler2.12-push}, forming successive pushouts yields a short exact sequence
		\begin{gather} \label{diag: split-off-inj}
			\begin{tikzcd}[sep={15mm,between origins}, ampersand replacement=\&]
				J \ar[d, tail] \& X^0 \ar[r, tail, "\alpha^0"] \ar[d, equal] \& \cdots \ar[r, tail, "\alpha^{s-2}"] \& X^{s-1} \ar[r, equal] \ar[d, equal] \& X^{s-1} \ar[r, equal] \ar[d, tail, "\alpha^{s-1}"] \& \cdots \ar[r, equal] \& X^{s-1} \ar[r, equal] \ar[d, tail, "\alpha^{t-1} \cdots \alpha^{s-1}"] \& X^{s-1} \ar[r, equal] \ar[d, tail, "\alpha^t \cdots \alpha^{s-1}"] \& \cdots \ar[r, equal] \& X^{s-1} \ar[d, tail, "\alpha^{l-1} \cdots \alpha^{s-1}"] \\
				X \ar[d, two heads] \& X^0 \ar[r, tail, "\alpha^0"] \ar[d, two heads] \& \cdots \ar[r, tail, "\alpha^{s-2}"] \& X^{s-1} \ar[r, tail, "\alpha^{s-1}"] \ar[d, two heads] \ar[rd, phantom, "\square"] \& X^s \ar[r, tail, "\alpha^s"] \ar[d, two heads] \& \cdots \ar[r, tail, "\alpha^{t-1}"] \ar[rd, phantom, "\square"] \& X^t \ar[r, tail, "\alpha^t"] \ar[d, two heads]\ar[rd, phantom, "\square"] \& X^{t+1} \ar[r, tail, "\alpha^{t+1}"] \ar[d, two heads] \& \cdots \ar[r, tail, "\alpha^{l-1}"] \ar[rd, phantom, "\square"] \& X^l \ar[d, two heads] \\
				Y \& 0 \ar[r, equal] \& \cdots \ar[r, equal] \& 0 \ar[r, tail] \& \tilde X^s \ar[r, tail] \ar[ru, phantom, "\square"] \& \cdots \ar[r, tail] \& \tilde X^t \ar[r, tail] \& \tilde X^{t+1} \ar[r, tail] \ar[ru, phantom, "\square"] \& \cdots \ar[r, tail] \& \tilde X^l
			\end{tikzcd}
		\end{gather}
		in $\mMor_{l}(\E)$. If $s=0$, use $\id_{X^0}\colon X^0 \to X^0 =: \tilde X^0$ for the leftmost pushout. By \Cref{thm: mMor}.\ref{thm: mMor-Proj}, $J \in \Inj(\mMor_{l}(\E))$ and the sequence splits. In particular, $\epic{X}{Y}$ is a stable isomorphism if $\E$ is Frobenius, see \Cref{lem: ses-triangle}. If $t=l$, this proves the claim with $\tilde X = Y$ and $I=J$.\\
		Otherwise, for each $k \in \{t+1, \dots, l\}$, the short exact sequence 
		\begin{tikzcd}[cramped, sep=small]
			X^{s-1} \ar[r, tail] & X^k \ar[r, two heads] & \tilde X^k
		\end{tikzcd}
		yields $\tilde X^k \in \Inj(\E)$ such that the admissible monic $\monic{\tilde X^{t+1}}{\tilde X^k}$ splits, which gives rise to a biproduct $\tilde X^k \cong \tilde X^{t+1} \oplus \tilde J^k$ with $\tilde J^k \in \Inj(\E)$, see \Cref{prp: Proj-Inj}. \Cref{lem: split-off-equalities} applied to $\gamma^{[0,t]}(Y)$ then yields a short exact sequence
		\[\begin{tikzcd}[sep={15mm,between origins}]
				\tilde J \ar[d, tail] & 0 \ar[r, equal] \ar[d, equal] & \cdots \ar[r, equal] & 0 \ar[r, equal] \ar[d, equal] & 0 \ar[r, equal] \ar[d, tail] & \cdots \ar[r, equal] & 0 \ar[r, equal] \ar[d, tail] & 0 \ar[r, tail] \ar[d, tail] & \tilde J^{t+2} \ar[r, tail] \ar[d, tail] & \cdots \ar[r, tail] \ar[rd, phantom, "\square"] & \tilde J^l \ar[d, tail]\\
				Y \ar[d, two heads] & 0 \ar[r, equal] \ar[d, equal] & \cdots \ar[r, equal] & 0 \ar[r, tail] \ar[d, equal] & \tilde X^s \ar[r, tail] \ar[d, equal] & \cdots \ar[r, tail] & \tilde X^t \ar[r, tail] \ar[d, equal] & \tilde X^{t+1} \ar[r, tail] \ar[d, equal] \ar[ru, phantom, "\square"] & \tilde X^{t+2} \ar[r, tail] \ar[d, two heads] \ar[ru, phantom, "\square"] & \cdots \ar[r, tail] & \tilde X^l \ar[d, two heads]\\
				\tilde X & 0 \ar[r, equal] & \cdots \ar[r, equal] & 0 \ar[r, tail] & \tilde X^s \ar[r, tail] & \cdots \ar[r, tail] & \tilde X^t \ar[r, tail] & \tilde X^{t+1} \ar[r, equal] & \tilde X^{t+1} \ar[r, equal] & \cdots \ar[r, equal]& \tilde X^{t+1}
			\end{tikzcd}\]
		in $\mMor_l(\E)$ with $\tilde J \in \Inj(\mMor_l(\E))$, see \Cref{thm: mMor}.\ref{thm: mMor-Proj}. As above, $\epic{Y}{\tilde X}$ splits and is a stable isomorphism if $\E$ is Frobenius. Due to \Cref{lem: epi-comp}, $I := J \oplus \tilde J$ is the kernel of the composition $\begin{tikzcd}[cramped, sep=small] p \colon X \ar[r, two heads] & Y \ar[r, two heads] & \tilde X \end{tikzcd}$.
		
		\item A right-inverse $i^{t+1}$ of $p^{t+1}$ corresponds to a left-inverse $r^{t+1} \colon \epic{X^{t+1}}{I^{t+1}}$ of $\alpha^t \cdots \alpha^{s-1}$. Set $r^k := r^{t+1} \alpha^t \cdots \alpha^k$ for $k \in \{s, \dots, t\}$ and $r^k := \id_{X^k}$ for $k \in \{0, \dots, s-1\}$. Apply $\gamma^{[0,t]}$ followed by \Cref{lem: extend-left-inverse} to define $r^k$ for $k \in \{t+2, \dots, l\}$ starting from $r^{t+1}$. Then $r := \left(r^0, \dots, r^l\right) \colon X \to I$ is a left-inverse of $\monic{I}{X}$, which corresponds to the desired right-inverse of $p$. The claimed bicartesian squares are due to \Cref{prp: Buehler2.12}. \qedhere
	\end{enumerate}
\end{proof}

\begin{lem} \label{lem: split-off-equalities}
	Let $\E$ be an exact category and $l \in \NN$. Then any $X=(X, \alpha) \in \smMor_l(\E)$ fits into a termwise split short exact sequence
	\[\begin{tikzcd}[sep={17.5mm,between origins}]
			0 \ar[r, tail] \ar[d, tail] \ar[rd, phantom, "\square"] & \tilde X^1 \ar[r, tail] \ar[d, tail] \ar[rd, phantom, "\square"] & \tilde X^2 \ar[r, tail] \ar[d, tail] & \cdots \ar[r, tail] \ar[rd, phantom, "\square"] & \tilde X^l \ar[d, tail] \\
			X^0 \ar[r, tail, "\alpha^0"] \ar[d, equal] & X^1 \ar[r, tail, "\alpha^1"] \ar[d, two heads, "\tilde \beta^1"] & X^2 \ar[r, tail, "\alpha^2"] \ar[d, two heads, "\tilde \beta^2"] \ar[ru, phantom, "\square"] & \cdots \ar[r, tail, "\alpha^{l-1}"] & X^l \ar[d, two heads, "\tilde \beta^l"]\\
			X^0 \ar[r, equal] & X^0 \ar[r, equal] & X^0 \ar[r, equal] & \cdots \ar[r, equal] & X^0
		\end{tikzcd}\] 
	in $\mMor_{l}(\E)$, where $\tilde X^k$ is the cokernel of $\tilde \alpha^{k-1} := \alpha^{k-1}\cdots \alpha^0 \colon \monic{X^0}{X^k}$ for $k \in \{1, \dots, l\}$.
\end{lem}

\begin{proof}
	By assumption, for any $k \in \{1, \dots, l\}$, there is $\beta^k \colon X^k \to X^{k-1}$ such that $\beta^k \alpha^{k-1} = \id_{X^{k-1}}$. In particular, $\tilde \beta^k := \beta^1 \cdots \beta^k \colon X^k \to X^0$ satisfies $\tilde \beta^k \tilde \alpha^{k-1} = \id_{X^0}$ and $\tilde \beta^k \alpha^{k-1} = \tilde \beta^{k-1}$, which establishes the vertical split short exact sequences and the lower half of the diagram. \Cref{lem: Noether} applied to
	\[
		\begin{tikzcd}[sep={17.5mm,between origins}]
			\tilde X^k \ar[r, tail, dashed] \ar[d, tail] & \tilde X^{k+1} \ar[r, two heads, dashed] \ar[d, tail] & \bullet \ar[d, equal] \\
			X^k \ar[r, tail, "\alpha^k"] \ar[d, two heads, "\tilde \beta^k"] & X^{k+1} \ar[r, two heads] \ar[d, two heads, "\tilde \beta^k"] & \bullet \ar[d, two heads] \\
			X^0 \ar[r, equal] & X^0 \ar[r, two heads] & 0
		\end{tikzcd}
	\]
	for $k \in \{1, \dots, l-1\}$ then completes the diagram. The bicartesian squares are due to \Cref{prp: Buehler2.12}.\ref{prp: Buehler2.12-pull}.
\end{proof}

\begin{lem} \label{lem: epi-comp}
	In an exact category, any split short exact sequence $\begin{tikzcd}[cramped, sep=small] A \ar[r, tail, "i"] & B \ar[r, two heads, "p"] & C \end{tikzcd}$ and short exact sequence $\begin{tikzcd}[cramped, sep=small] A' \ar[r, tail, "i'"] & C \ar[r, two heads, "p'"] & D \end{tikzcd}$ give rise to a short exact sequence
	\[\begin{tikzcd}[ampersand replacement=\&] A \oplus A' \ar[>->]{r}{\begin{pmatrix}
				i & ji'
		\end{pmatrix}} \& B \ar[r, two heads, "p'p"] \& D, \end{tikzcd}\]
	where $j\colon C \to B$ is any right-inverse of $p$.
\end{lem}

\begin{proof}
	Given a kernel $k \colon \monic{E}{B}$ of $p'p$, \Cref{lem: Noether} yields a commutative diagram
	\[
		\begin{tikzcd}[sep={17.5mm,between origins}]
			A \ar[r, equal] \ar[d, tail, dashed, "t"] & A \ar[r, two heads] \ar[d, tail, "i"] & 0 \ar[d, tail] \\
			E \ar[r, tail, "k"] \ar[d, two heads, dashed] & B \ar[r, two heads, "p'p"] \ar[d, two heads, "p"] \ar[u, bend right=15mm, dotted, "q"'] & D \ar[d, equal] \\
			A' \ar[r, tail, "i'"] & C \ar[r, two heads, "p'"] \ar[u, bend right=15mm, dotted, "j"'] & D
		\end{tikzcd}
	\]
	of solid and dashed arrows with short exact rows and columns. With $q$ the left-inverse of $i$ corresponding to $j$, then $qkt=qi=\id_A$ and the left column splits. With respect to this splitting, $k$ corresponds to $\begin{pmatrix}
		i & ji'
	\end{pmatrix}$ and the claim follows.
\end{proof}

\section{Semiorthogonal decompositions} \label{sec: SOD}

In this section, we establish various semiorthogonal decompositions of stable monomorphism categories using the subcategories introduced in \Cref{sec: con-exp}. These decompositions form polygons of recollements.

\begin{prp} \label{prp: SOD-Gamma-Gamma}
	Let $\F$ be a Frobenius category, $l \in \NN$, and $s \in \{0, \dots, l-1\}$. Then $\stab \mMor_{l}(\F)$ admits the semiorthogonal decomposition $\left(\Gamma^{[s+1, l]}, \Gamma^{[0, s]}\right)$ as follows:
	
	\begin{enumerate}
		\item \label{prp: SOD-Gamma-Gamma-diag} Any $X= (X, \alpha) \in \mMor_{l}(\F)$ fits into a distinguished triangle in $\stab \mMor_{l}(\F)$ represented by the commutative diagram in $\F$
		\[
			\begin{tikzcd}[row sep={15mm,between origins}, column sep={17.5mm,between origins}]
				0 \ar[r, equal] \ar[d] & \cdots \ar[r, equal] & 0 \ar[r, tail] \ar[d] & X^{s+1} \ar[r, tail, "\alpha^{s+1}"] \ar[d, equal] & X^{s+2} \ar[r, tail, "\alpha^{s+2}"] \ar[d, equal] & \cdots \ar[r, tail, "\alpha^{l-1}"] & X^l \ar[d, equal] \\
				X^0 \ar[r, tail, "\alpha^0"] \ar[d, equal] & \cdots \ar[r, tail, "\alpha^{s-1}"] & X^{s} \ar[r, tail, "\alpha^{s}"] \ar[d, equal] & X^{s+1} \ar[r, tail, "\alpha^{s+1}"] \ar[d, "i_{X^{s+1}}", tail] & X^{s+2} \ar[r, tail, "\alpha^{s+2}"] \ar[d] & \cdots \ar[r, tail, "\alpha^{l-1}"] & X^l \ar[d] \\
				X^0 \ar[r, tail, "\alpha^0"] & \cdots \ar[r, tail, "\alpha^{s-1}"] & X^{s} \ar[r, tail] & I(X^{s+1}) \ar[r, equal] & I(X^{s+1}) \ar[r, equal] & \cdots \ar[r, equal] & I(X^{s+1}).
			\end{tikzcd}
		\]
		
		\item \label{prp: SOD-Gamma-Gamma-right} The inclusion $\inj{\Gamma^{[s+1, l]}}{\stab \mMor_{l}(\F)}$ has a right adjoint which sends any $X \in \mMor_{l}(\F)$ to the object in the upper row of the diagram in \ref{prp: SOD-Gamma-Gamma-diag}.
		
		\item \label{prp: SOD-Gamma-Gamma-left} The inclusion $\inj{\Gamma^{[0, s]}}{\stab \mMor_{l}(\F)}$ has a left adjoint which sends any $X \in \mMor_{l}(\F)$ to the object in the lower row of the diagram in \ref{prp: SOD-Gamma-Gamma-diag}.
	\end{enumerate}
\end{prp}

\begin{proof}
	For condition \ref{dfn: sod-1} in \Cref{dfn: sod}, consider a morphism $f\colon X \to Y$ in $\mMor_{l}(\F)$ with $X \in \Gamma^{[s+1,l]}$ and $Y \in \Gamma^{[0, s]}$. Due to \Cref{lem: Gamma-s-t}, we may assume that $X^k = 0$ for $k \in \{0, \dots, s\}$ and $Y^k \in \Proj(\F)$ for $k \in \{s+1, \dots, l\}$. Then $f$ factors as follows:
	\[
		\begin{tikzcd}[sep={15mm,between origins}]
			0 \ar[r, equal] \ar[d, equal] & \cdots \ar[r, equal] & 0 \ar[r, tail] \ar[d, equal] & X^{s+1} \ar[r, tail] \ar[d, "f^{s+1}"] & \cdots \ar[r, tail] & X^l \ar[d, "f^l"] \\
			0 \ar[r, equal] \ar[d, "f^0"] & \cdots \ar[r, equal] & 0 \ar[r, tail] \ar[d, "f^s"] & Y^{s+1} \ar[r, tail] \ar[d, equal] & \cdots \ar[r, tail] & Y^l \ar[d, equal] \\
			Y^{0} \ar[r, tail] & \cdots \ar[r, tail] & Y^{s} \ar[r, tail] & Y^{s+1} \ar[r, tail] & \cdots \ar[r, tail] & Y^l
		\end{tikzcd}
	\]
	The middle object is projective, see \Cref{thm: mMor}.\ref{thm: mMor-Proj}, and hence $f = 0$ in $\stab \mMor_l(\F)$.\\
	For condition \ref{dfn: sod-2} in \Cref{dfn: sod}, we prove that for any $X \in \mMor_l(\F)$ the morphism
	\[
		\begin{tikzcd}[sep={15mm,between origins}]
			\tilde X \ar[d, "f"] & 0 \ar[r, equal] \ar[d] & \cdots \ar[r, equal] & 0 \ar[r, tail] \ar[d] & X^{s+1} \ar[r, tail] \ar[d, equal] & \cdots \ar[r, tail] & X^l \ar[d, equal] \\
			X & X^0 \ar[r, tail] & \cdots \ar[r, tail] & X^{s} \ar[r, tail] & X^{s+1} \ar[r, tail] & \cdots \ar[r, tail] & X^l
		\end{tikzcd}
	\]
	in $\mMor_l(\F)$ with $\tilde X \in \Gamma^{[s+1, l]}$ has its cone $C := C(f)$ in $\Gamma^{[0, s]}$. To this end, we consider the pushout of $f$ along the admissible monic $i:=i_{\tilde X}\colon\monic{\tilde X}{I(\tilde X) =: I}$, where $I^k = 0$ for all $k \in \{0, \dots, s\}$ and $i^{s+1}=i_{X^{s+1}} \colon X^{s+1} \to I(X^{s+1}) = I^{s+1}$, see \Cref{cnv: I-P}, \Cref{con: mMor-enough}.\ref{con: mMor-enough-inj}, and \Cref{rmk: cone-termwise}:
	\[
		\begin{tikzcd}[column sep={7.5mm,between origins}, row sep={10mm,between origins}]
			& I \ar[dd] && 0 \ar[rr, equal] \ar[dd] && \cdots \ar[rr, equal] && 0 \ar[rr, tail] \ar[dd] && I^{s+1} \ar[rr, tail] \ar[dd, equal] && I^{s+2} \ar[rr, tail] \ar[dd, equal] && \cdots \ar[rr, tail] && I^l \ar[dd, equal]\\
			\tilde X \ar[ru, tail, "i"] && 0 \ar[rr, equal, crossing over] \ar[ru, equal] && \cdots \ar[rr, equal] && 0 \ar[rr, tail, crossing over] \ar[ru, equal] && X^{s+1} \ar[rr, tail, crossing over] \ar[ru, tail] && X^{s+2} \ar[rr, tail, crossing over] \ar[ru, tail] && \cdots \ar[rr, tail] && X^l \ar[ru, tail]\\
			& C && X^0 \ar[rr, tail] && \cdots \ar[rr, tail] && X^{s} \ar[rr, tail] && I^{s+1} \ar[rr, tail] && I^{s+2} \ar[rr, tail] && \cdots \ar[rr, tail] && I^l\\
			X \ar[ru, tail] \ar[uu, <-, "f"] && X^{0} \ar[rr, tail] \ar[ru, equal] \ar[uu, <-] && \cdots \ar[rr, tail] && X^{s} \ar[rr, tail] \ar[ru, equal] \ar[uu, <-, crossing over] && X^{s+1} \ar[rr, tail] \ar[ru, tail, "i^{s+1}"', pos=1.25] \ar[uu, equal, crossing over] && X^{s+2} \ar[rr, tail] \ar[ru, tail] \ar[uu, equal, crossing over] && \cdots \ar[rr, tail] && X^l \ar[ru, tail] \ar[uu, equal, crossing over]
		\end{tikzcd}
	\]
	We postcompose $X \to C$ with the stable isomorphism $C \to \tilde C$ from \Cref{lem: Gamma-s-t}.\ref{lem: Gamma-s-t-epi} to obtain a distinguished triangle
	\begin{align}
		\begin{tikzcd}[row sep={15mm,between origins}, column sep={17.5mm,between origins}, ampersand replacement=\&]
			\tilde X \ar[d, "f"] \& 0 \ar[r, equal] \ar[d] \& \cdots \ar[r, equal] \& 0 \ar[r, tail] \ar[d] \& X^{s+1} \ar[r, tail] \ar[d, equal] \& X^{s+2} \ar[r, tail] \ar[d, equal] \& \cdots \ar[r, tail] \& X^l \ar[d, equal] \\
			X \ar[d] \& X^0 \ar[r, tail] \ar[d, equal] \& \cdots \ar[r, tail] \& X^{s} \ar[r, tail] \ar[d, equal] \& X^{s+1} \ar[r, tail] \ar[d, "i_{X^{s+1}}", tail] \& X^{s+2} \ar[r, tail] \ar[d] \& \cdots \ar[r, tail] \& X^l \ar[d] \\
			\tilde C \& X^0 \ar[r, tail] \& \cdots \ar[r, tail] \& X^{s} \ar[r, tail] \& I(X^{s+1}) \ar[r, equal] \& I(X^{s+1}) \ar[r, equal] \& \cdots \ar[r, equal] \& I(X^{s+1})
		\end{tikzcd}
	\end{align}
	in $\stab \mMor_{l}(\F)$ with $\tilde X \in \Gamma^{[s+1, l]}$ and $\tilde C \in \Gamma^{[0, s]}$. The claims on adjoints are due to \Cref{prp: SOD-adjoints}.
\end{proof}

\begin{cor} \label{cor: append-I}
	Let $\F$ be a Frobenius category, $l \in \NN$, and $s \in \{0, \dots, l\}$. Given $f\colon X \to Y$ in $\mMor_{s}(\F)$ and two admissible monics $i\colon\monic{X^s}{I}$ and $j\colon\monic{Y^s}{J}$ in $\F$ with $I, J \in \Proj(\F)$, there is a unique morphism of the form
	\[
		\begin{tikzcd}[sep={15mm,between origins}]
			\tilde X \ar[d, "\tilde f"] & X^0 \ar[r, tail] \ar[d, "f^0"] & \cdots \ar[r, tail] & X^s \ar[r, tail, "i"] \ar[d, "f^s"] & I \ar[r, equal] \ar[d] & \cdots \ar[r, equal] & I \ar[d] \\
			\tilde Y & Y^0 \ar[r, tail] & \cdots \ar[r, tail] & Y^s \ar[r, tail, "j"] & J \ar[r, equal] & \cdots \ar[r, equal] & J
		\end{tikzcd}
	\]
	in $\stab \mMor_{l}(\F)$. In particular, $\tilde f$ is an isomorphism if $f=\id_X$.
\end{cor}

\begin{proof}
	The objects $\tilde X, \tilde Y \in \Gamma^{[0, s-1]}$ appear in the distinguished triangles of \Cref{prp: SOD-Gamma-Gamma}.\ref{prp: SOD-Gamma-Gamma-diag} applied to $\delta^{[s, l]}_s(X)$ with $i_{X^s}=i$ and $\delta^{[s, l]}_s(Y)$ with $i_{Y^s}=j$. \Cref{rmk: Hom=0} yields the claim.
\end{proof}

Combining the particular claims of \Cref{lem: Gamma-s-t} and \Cref{cor: append-I} yields

\begin{cor} \label{cor: Gamma-s-t}
	Let $\F$ be a Frobenius category, $l \in \NN$, and $s, t \in \{0, \dots, l\}$ with $s \leq t$. Then $\Gamma^{[s,t]} \subseteq \stab \mMor_{l}(\F)$ is the replete hull of the subcategory of $\stab \mMor_l(\F)$ with objects of the form
	\pushQED{\qed}
	\[\begin{tikzcd}[sep={15mm,between origins}]
		0 \ar[r, equal] & \cdots \ar[r, equal] & 0 \ar[r, tail] & X^s \ar[r, tail] & \cdots \ar[r, tail] & X^t \ar[r, tail, "i_{X^t}"] & I(X^t) \ar[r, equal] & \cdots \ar[r, equal] & I(X^t). 
	\end{tikzcd}\qedhere\]
\end{cor}

\begin{prp} \label{prp: SODs}
	Let $\F$ be a Frobenius category and $l \in \NN$.
	\begin{enumerate}
		\item \label{prp: SODs-Gamma-Delta} For $s, t \in \{0, \dots, l-1\}$ with $s \leq t$, $\stab \mMor_{l}(\F)$ admits the semiorthogonal decomposition $\left(\Gamma^{[s,t]}, \Delta^{[s, t+1]}\right)$ as follows:
		\begin{enumerate}
			\item \label{prp: SODs-Gamma-Delta-diag} Any $X = (X, \alpha) \in \mMor_{l}(\F)$ fits into a distinguished triangle in $\stab \mMor_{l}(\F)$ represented by the commutative diagram in $\F$
			\[
				\begin{tikzcd}[row sep={15mm,between origins}, column sep={16.5mm,between origins}]
					0 \ar[r, equal] \ar[d] & \cdots \ar[r, equal] & 0 \ar[r, tail] \ar[d] & Y^s \ar[r, tail] \ar[d, two heads] & \cdots \ar[r, tail] \ar[rd, phantom, "\square"] & Y^t \ar[r, tail] \ar[d, two heads] \ar[rd, phantom, "\square"] & P(X^{t+1}) \ar[r, equal] \ar[d, two heads, "p_{X^{t+1}}"] & \cdots \ar[r, equal] & P(X^{t+1}) \ar[d]\\
					X^0 \ar[r, tail, "\alpha^0"] \ar[d, equal] & \cdots \ar[r, tail, "\alpha^{s-2}"] & X^{s-1} \ar[r, tail, "\alpha^{s-1}"] \ar[d, equal] & X^s \ar[r, tail, "\alpha^s"] \ar[d, tail, "\alpha^t \cdots \alpha^s"] \ar[ru, phantom, "\square"] & \cdots \ar[r, tail, "\alpha^{t-1}"] & X^t \ar[r, tail, "\alpha^t"] \ar[d, tail, "\alpha^t"] & X^{t+1} \ar[r, tail, "\alpha^{t+1}"] \ar[d, equal] & \cdots \ar[r, tail, "\alpha^{l-2}"] & X^l \ar[d, equal]\\
					X^0 \ar[r, tail, "\alpha^0"] & \cdots \ar[r, tail, "\alpha^{s-2}"] & X^{s-1} \ar[r, tail, "\alpha^t \cdots \alpha^{s-1}"] & X^{t+1} \ar[r, equal] & \cdots \ar[r, equal] & X^{t+1} \ar[r, equal] & X^{t+1} \ar[r, tail, "\alpha^{t+1}"] & \cdots \ar[r, tail, "\alpha^{l-1}"] & X^l.
				\end{tikzcd}
			\]
			\item \label{prp: SODs-Gamma-Delta-right} The inclusion $\inj{\Gamma^{[s,t]}}{\stab \mMor_{l}(\F)}$ has a right adjoint which sends any $X \in \mMor_{l}(\F)$ to the object in the upper row of the diagram in \ref{prp: SODs-Gamma-Delta-diag}.
			
			\item \label{prp: SODs-Gamma-Delta-left} For $s<t$, the inclusion $\inj{\Delta^{[s, t]}}{\stab \mMor_{l}(\F)}$ has a left adjoint which sends any $X \in \mMor_{l}(\F)$ to the object $\delta^{[s, t]} \gamma^{[s, t-1]}(X)$ in the lower row of the diagram in \ref{prp: SODs-Gamma-Delta-diag}.
		\end{enumerate}
		
		\item \label{prp: SODs-Delta-Gamma} For $s, t \in \{1, \dots, l\}$ with $s \leq t$, $\stab \mMor_{l}(\F)$ admits the semiorthogonal decomposition $\left(\Delta^{[s-1,t]}, \Gamma^{[s, t]}\right)$ as follows:
		\begin{enumerate}
			\item \label{prp: SODs-Delta-Gamma-diag} Any $X = (X, \alpha) \in \mMor_{l}(\F)$ fits into a distinguished triangle in $\stab \mMor_{l}(\F)$ represented by the commutative diagram in $\F$
			\[
				\begin{tikzcd}[row sep={15mm,between origins}, column sep={16.5mm,between origins}]
					X^0 \ar[r, tail, "\alpha^0"] \ar[d, equal] & \cdots \ar[r, tail, "\alpha^{s-2}"] & X^{s-1} \ar[r, equal] \ar[d, equal] & X^{s-1} \ar[r, equal] \ar[d, tail, "\alpha^{s-1}"] & \cdots \ar[r, equal] & X^{s-1} \ar[r, tail, "\alpha^t \cdots \alpha^{s-1}"] \ar[d, tail, "\alpha^{t-1} \cdots \alpha^{s-1}"] & X^{t+1} \ar[r, tail, "\alpha^{t+1}"] \ar[d, equal] & \cdots \ar[r, tail, "\alpha^{l-1}"] & X^l \ar[d, equal] \\
					X^0 \ar[r, tail, "\alpha^0"] \ar[d] & \cdots \ar[r, tail, "\alpha^{s-2}"] & X^{s-1} \ar[r, tail, "\alpha^{s-1}"] \ar[d, two heads] \ar[rd, phantom, "\square"] & X^s \ar[r, tail, "\alpha^s"] \ar[d, two heads] & \cdots \ar[r, tail, "\alpha^{t-1}"] \ar[rd, phantom, "\square"] & X^{t} \ar[r, tail, "\alpha^t"] \ar[d, two heads] & X^{t+1} \ar[r, tail, "\alpha^{t+1}"] \ar[d] & \cdots \ar[r, tail, "\alpha^{l-1}"] & X^l \ar[d] \\
					0 \ar[r, equal] & \cdots \ar[r, equal] & 0 \ar[r, tail] & Y^s \ar[r, tail] \ar[ru, phantom, "\square"] & \cdots \ar[r, tail] & Y^t \ar[r, tail, "i_{Y^t}"] & I(Y^t) \ar[r, equal] & \cdots \ar[r, equal] & I(Y^t).
				\end{tikzcd}
			\]
			
			\item \label{prp: SODs-Delta-Gamma-right} For $s<t$, the inclusion $\inj{\Delta^{[s, t]}}{\stab \mMor_{l}(\F)}$ has a right adjoint which sends any $X \in \mMor_{l}(\F)$ to the object $\delta^{[s, t]} \gamma^{[s+1, t]}(X)$ in the upper row of the diagram in \ref{prp: SODs-Delta-Gamma-diag}.
			
			\item \label{prp: SODs-Delta-Gamma-left} The inclusion $\inj{\Gamma^{[s,t]}}{\stab \mMor_{l}(\F)}$ has a left adjoint which sends any $X \in \mMor_{l}(\F)$ to the object in the lower row of the diagram in \ref{prp: SODs-Delta-Gamma-diag}.
		\end{enumerate}
	\end{enumerate}
\end{prp}

\begin{proof} \
	\begin{enumerate}[wide]
		\item For condition \ref{dfn: sod-1} in \Cref{dfn: sod}, consider a morphism $f\colon (X, \alpha) \to (Y, \beta)$ in $\mMor_{l}(\F)$ with $X \in \Gamma^{[s,t]}$ and $Y \in \Delta^{[s, t+1]}$. It factors as
		\[
			\begin{tikzcd}[row sep={15mm,between origins}, column sep={16.5mm,between origins}]
				X^{0} \ar[r, tail, "\alpha^0"] \ar[d, equal] & \cdots \ar[r, tail, "\alpha^{s-2}"] & X^{s-1} \ar[r, tail, "\alpha^{s-1}"] \ar[d, equal] & X^s \ar[r, tail, "\alpha^s"] \ar[d, tail, "\alpha^t \cdots \alpha^s"] & \cdots \ar[r, tail, "\alpha^{t-1}"] & X^t \ar[r, tail, "\alpha^{t}"] \ar[d, tail, "\alpha^t"] & X^{t+1} \ar[r, tail, "\alpha^{t+1}"] \ar[d, equal] & X^{t+2} \ar[r, tail, "\alpha^{t+2}"] \ar[d, equal] & \cdots \ar[r, tail, "\alpha^{l-1}"] & X^l \ar[d, equal]\\
				X^{0} \ar[r, tail, "\alpha^0"] \ar[d, "f^{0}"] & \cdots \ar[r, tail, "\alpha^{s-2}"] & X^{s-1} \ar[r, tail, "\alpha^t \cdots \alpha^{s-1}"] \ar[d, "f^{s-1}"] & X^{t+1} \ar[r, equal] \ar[d, "f^{t+1}"] & \cdots \ar[r, equal] & X^{t+1} \ar[r, equal] \ar[d, "f^{t+1}"] & X^{t+1} \ar[r, tail, "\alpha^{t+1}"] \ar[d, "f^{t+1}"] & X^{t+2} \ar[r, tail, "\alpha^{t+2}"] \ar[d, "f^{t+2}"] & \cdots \ar[r, tail, "\alpha^{l-1}"] & X^l \ar[d, "f^l"]\\
				Y^{0} \ar[r, tail, "\beta^0"] & \cdots \ar[r, tail, "\beta^{s-2}"] & Y^{s-1} \ar[r, tail, "\beta^{s-1}"] & Y^s \ar[r, equal] & \cdots \ar[r, equal] & Y^t \ar[r, equal] & Y^{t+1} \ar[r, tail, "\beta^{t+1}"] & Y^{t+2} \ar[r, tail, "\beta^{t+2}"] & \cdots \ar[r, tail, "\beta^{l-1}"] & Y^l,
			\end{tikzcd}
		\]
		since $\beta^k = \id_{Y^k}$ and hence $f^k = f^{t+1} \alpha^t \cdots \alpha^k$ for $k \in \{s, \dots, t\}$. The object in the middle row is projective since $X^k \in \Proj(\F)$ for $k \in \{0, \dots, s-1, t+1, \dots, l\}$, see \Cref{thm: mMor}.\ref{thm: mMor-Proj}, and hence $f = 0$ in $\stab \mMor_l(\F)$.\\
		For condition \ref{dfn: sod-2} in \Cref{dfn: sod}, let $X = (X, \alpha) \in \mMor_l(\F)$ be arbitrary, and set $\tilde X := \delta^{[s, t+1]}\gamma^{[s, t]}(X) \in \Delta^{[s, t+1]}$. We prove that the morphism
		\[
			\begin{tikzcd}[row sep={15mm,between origins}, column sep={16.5mm,between origins}]
				X \ar[d, "f"'] & X^0 \ar[r, tail, "\alpha^0"] \ar[d, equal] & \cdots \ar[r, tail, "\alpha^{s-2}"] & X^{s-1} \ar[r, tail, "\alpha^{s-1}"] \ar[d, equal] & X^s \ar[r, tail, "\alpha^s"] \ar[d, tail, "\alpha^t \cdots \alpha^s"] & \cdots \ar[r, tail, "\alpha^{t-1}"] & X^t \ar[r, tail, "\alpha^t"] \ar[d, tail, "\alpha^t"] & X^{t+1} \ar[r, tail, "\alpha^{t+1}"] \ar[d, equal] & \cdots \ar[r, tail, "\alpha^{l-1}"] & X^l \ar[d, equal]\\
				\tilde X & X^0 \ar[r, tail, "\alpha^0"] & \cdots \ar[r, tail, "\alpha^{s-2}"] & X^{s-1} \ar[r, tail, "\alpha^t \cdots \alpha^{s-1}"] & X^{t+1} \ar[r, equal] & \cdots \ar[r, equal] & X^{t+1} \ar[r, equal] & X^{t+1} \ar[r, tail, "\alpha^{t+1}"] & \cdots \ar[r, tail, "\alpha^l"] & X^l
			\end{tikzcd}
		\]
		in $\mMor_{l}(\F)$ has its cocone $C^\ast := C^\ast(f)$ in $\Gamma^{[s, t]}$. To this end, we consider the pullback of $f$ along the admissible epic $p:=\delta^{[s, t+1]}(p_{\gamma^{[s,t]}(X)})\colon\epic{P }{\tilde X}$ with $P \in \mMor_{l}(\Proj(\F))$ and
		\begin{align} \label{diag: p}
			p^k = \begin{pmatrix}
				\alpha^{t} \cdots \alpha^{s-1} p^{s-1} & p_{X^{t+1}}
			\end{pmatrix}\colon \epic{P^k = P^{s-1} \oplus P(X^{t+1})}{X^{t+1}}
		\end{align}
		for all $k \in \{s, \dots, t+1\}$, see \Cref{con: mMor-enough}.\ref{con: mMor-enough-proj}. Due to \Cref{rmk: cone-termwise}, it displays as a termwise pullback:
		\[
			\begin{tikzcd}[column sep={7.5mm,between origins}, row sep={10mm,between origins}]
				& X \ar[dd, "f"'] && X^{0} \ar[rr, tail] \ar[dd, equal] && \cdots \ar[rr, tail] && X^{s-1} \ar[rr, tail] \ar[dd, equal] && X^s \ar[rr, tail] \ar[dd, tail] && \cdots \ar[rr, tail] && X^t \ar[rr, tail] \ar[dd, tail] && X^{t+1} \ar[rr, tail] \ar[dd, equal] && \cdots \ar[rr, tail] && X^l \ar[dd, equal] \\
				C^\ast \ar[ru, two heads] && P^0 \ar[rr, tail, crossing over]\ar[ru, two heads] && \cdots \ar[rr, tail] && P^{s-1} \ar[rr, tail, crossing over] \ar[ru, two heads] && C^s \ar[rr, tail, crossing over] \ar[ru, two heads] && \cdots \ar[rr, tail] && C^{t} \ar[rr, tail, crossing over] \ar[ru, two heads] && P^{t+1} \ar[rr, tail, crossing over] \ar[ru, two heads] && \cdots \ar[rr, tail] && P^l \ar[ru, two heads] \\
				& \tilde X && X^{0} \ar[rr, tail] && \cdots \ar[rr, tail] && X^{s-1} \ar[rr, tail] && X^{t+1} \ar[rr, equal] && \cdots \ar[rr, equal] && X^{t+1} \ar[rr, equal] && X^{t+1} \ar[rr, tail] && \cdots \ar[rr, tail] && X^l \\
				P \ar[uu, <-] \ar[ru, two heads, "p"'] && P^0 \ar[rr, tail] \ar[uu, equal] \ar[ru, two heads] && \cdots \ar[rr, tail] && P^{s-1} \ar[rr, tail] \ar[uu, equal, crossing over] \ar[ru, two heads] && P^{s} \ar[rr, equal] \ar[uu, <-<, crossing over] \ar[ru, two heads] && \cdots \ar[rr, equal] && P^{t} \ar[rr, equal] \ar[uu, <-<, crossing over] \ar[ru, two heads] && P^{t+1} \ar[rr, tail] \ar[uu, equal, crossing over] \ar[ru, two heads] && \cdots \ar[rr, tail] && P^l \ar[uu, equal, crossing over] \ar[ru, two heads]
			\end{tikzcd}
		\]
		In particular, for each $k \in \{s, \dots, t\}$, the square
		\[
			\begin{tikzcd}[row sep={15mm,between origins}, column sep={16.5mm,between origins}]
				C^k \ar[r, tail] \ar[d, two heads] \ar[rd, phantom, "\square"]& P^k \mathrlap{\; = P^{t+1}} \ar[d, two heads] \\
				X^k \ar[r, tail, "\alpha^t \cdots \alpha^k"] & X^{t+1}
			\end{tikzcd}
		\]
		is bicartesian and, for $k < t$, so is the left square in the diagram
		\[
			\begin{tikzcd}[row sep={15mm,between origins}, column sep={16.5mm,between origins}]
				C^k \ar[r, tail] \ar[d, two heads] \ar[rd, phantom, "\square"] & C^{k+1} \ar[r, tail] \ar[d, two heads] & P^{t+1} \ar[d, two heads] \\
				X^k \ar[r, tail, "\alpha^k"] & X^{k+1} \ar[r, tail, "\alpha^t \cdots \alpha^{k+1}"] & X^{t+1}
			\end{tikzcd}
		\]
		due to 	\Cref{lem: pushpull-sep}.\ref{lem: pushpull-sep-pullback}. Using the right-inverse of $\epic{P^{t+1}}{P(X^{t+1})}$ given by the biproduct in \eqref{diag: p}, \Cref{lem: Gamma-s-t}.\ref{lem: Gamma-s-t-mono} yields a stable isomorphism $\tilde C^\ast \to C^\ast$. By composition we obtain a stably isomorphic cocone of $f$:
		\[ 
			\begin{tikzcd}[row sep={15mm,between origins}, column sep={16.5mm,between origins}]
				\tilde C^\ast \ar[d] & 0 \ar[r, equal] \ar[d, tail] & \cdots \ar[r, equal] & 0 \ar[r, tail] \ar[d] \ar[rd, phantom, "\square"] & \tilde C^s \ar[r, tail] \ar[d, tail] & \cdots \ar[r, tail] \ar[rd, phantom, "\square"] & \tilde C^t \ar[r, tail] \ar[d, tail] \ar[rd, phantom, "\square"] & P(X^{t+1}) \ar[r, equal] \ar[d, tail] & \cdots \ar[r, equal] & P(X^{t+1}) \ar[d, tail]\\
				C^\ast \ar[d] & P^0 \ar[r, tail] \ar[d, two heads] & \cdots \ar[r, tail] & P^{s-1} \ar[r, tail] \ar[d, two heads] & C^s \ar[r, tail] \ar[d, two heads] \ar[ru, phantom, "\square"] & \cdots \ar[r, tail] \ar[rd, phantom, "\square"] & C^t \ar[r, tail] \ar[d, two heads] \ar[rd, phantom, "\square"] & P^{t+1} \ar[r, tail] \ar[d, two heads] & \cdots \ar[r, tail] & P^l \ar[d, two heads]\\
				X & X^0 \ar[r, tail, "\alpha^0"] & \cdots \ar[r, tail, "\alpha^{s-2}"] & X^{s-1} \ar[r, tail, "\alpha^{s-1}"] & X^{s} \ar[r, tail, "\alpha^s"] \ar[ru, phantom, "\square"] & \cdots \ar[r, tail, "\alpha^{t-1}"] & X^{t} \ar[r, tail, "\alpha^t"] & X^{t+1} \ar[r, tail, "\alpha^{t+1}"] & \cdots \ar[r, tail, "\alpha^{l-1}"] & X^l.
			\end{tikzcd}
		\]
		Due to \Cref{lem: ses-triangle}, this yields a distinguished triangle
		\[
			\begin{tikzcd}[row sep={15mm,between origins}, column sep={16.5mm,between origins}]
				\tilde C^\ast \ar[d] & 0 \ar[r, equal] \ar[d] & \cdots \ar[r, equal] & 0 \ar[r, tail] \ar[d,] & 	\tilde C^s \ar[r, tail] \ar[d, two heads] & \cdots \ar[r, tail] \ar[rd, phantom, "\square"] & \tilde C^t \ar[r, tail] \ar[d, two heads] \ar[rd, phantom, "\square"] & P(X^{t+1}) \ar[d, two heads, "p_{X^{t+1}}"] \ar[r, equal] & \cdots \ar[r, equal] & P(X^{t+1}) \ar[d]\\
				X \ar[d, "f"'] & X^0 \ar[r, tail, "\alpha^0"] \ar[d, equal] & \cdots \ar[r, tail, "\alpha^{s-2}"] & X^{s-1} \ar[r, tail, "\alpha^{s-1}"] \ar[d, equal] & X^s \ar[r, tail, "\alpha^s"] \ar[d, tail, "\alpha^t \cdots \alpha^s"] \ar[ru, phantom, "\square"] & \cdots \ar[r, tail, "\alpha^{t-1}"] & X^t \ar[r, tail, "\alpha^t"] \ar[d, tail, "\alpha^t"] & X^{t+1} \ar[r, tail, "\alpha^{t+1}"] \ar[d, equal] & \cdots \ar[r, tail, "\alpha^{l-1}"] & X^l \ar[d, equal]\\
				\tilde X & X^0 \ar[r, tail, "\alpha^0"] & \cdots \ar[r, tail, "\alpha^{s-2}"] & X^{s-1} \ar[r, tail, "\alpha^t \cdots \alpha^{s-1}"] & X^{t+1} \ar[r, equal] & \cdots \ar[r, equal] & X^{t+1} \ar[r, equal] & X^{t+1} \ar[r, tail, "\alpha^{t+1}"] & \cdots \ar[r, tail, "\alpha^l"] & X^l
			\end{tikzcd}
		\]
		in $\mMor_{l}(\F)$ with $\tilde X = \delta^{[s, t+1]} \gamma^{[s, t]}(X) \in \Delta^{[s, t+1]}$ and $\tilde C^\ast \in \Gamma^{[s, t]}$.
		\Cref{prp: SOD-adjoints} yields the claims on adjoints. 		
		
		\item For condition \ref{dfn: sod-1} in \Cref{dfn: sod}, consider a morphism $f\colon (X, \alpha) \to (Y, \beta)$ in $\mMor_{l}(\F)$ with $X \in \Delta^{[s-1, t]}$ and $Y \in \Gamma^{[s,t]}$. It factors as
		\[
			\begin{tikzcd}[row sep={15mm,between origins}, column sep={16.5mm,between origins}]
				X^0 \ar[r, tail, "\alpha^0"] \ar[d, "f^0"] & \cdots \ar[r, tail, "\alpha^{s-3}"] & X^{s-2} \ar[r, tail, "\alpha^{s-2}"] \ar[d, "f^{s-2}"] & X^{s-1} \ar[r, equal] \ar[d, "f^{s-1}"] & X^s \ar[r, equal] \ar[d, "f^{s-1}"] & \cdots \ar[r, equal] & X^{t} \ar[r, tail, "\alpha^t"] \ar[d, "f^{s-1}"] & X^{t+1} \ar[r, tail, "\alpha^{t+1}"] \ar[d, "f^{t+1}"] & \cdots \ar[r, tail, "\alpha^{l-1}"] & X^l \ar[d, "f^{l}"] \\
				Y^0 \ar[r, tail, "\beta^0"] \ar[d, equal] & \cdots \ar[r, tail, "\beta^{s-3}"] & Y^{s-2} \ar[r, tail, "\beta^{s-2}"] \ar[d, equal] & Y^{s-1} \ar[r, equal] \ar[d, equal] & Y^{s-1} \ar[r, equal] \ar[d, "\beta^{s-1}"] & \cdots \ar[r, equal] & Y^{s-1} \ar[r, tail, "\beta^t \cdots \beta^{s-1}"] \ar[d, "\beta^{t-1} \cdots \beta^{s-1}"] & Y^{t+1} \ar[r, tail, "\beta^{t+1}"] \ar[d, equal] & \cdots \ar[r, tail, "\beta^{l-1}"] & Y^l \ar[d, equal] \\
				Y^0 \ar[r, tail, "\beta^0"] & \cdots \ar[r, tail, "\beta^{s-3}"] & Y^{s-2} \ar[r, tail, "\beta^{s-2}"] & Y^{s-1} \ar[r, tail, "\beta^{s-1}"] & Y^s \ar[r, tail, "\beta^s"] & \cdots \ar[r, tail, "\beta^{t-1}"] & Y^t \ar[r, tail, "\beta^{t}"] & Y^{t+1} \ar[r, tail, "\beta^{t+1}"] & \cdots \ar[r, tail, "\beta^{l-1}"] & Y^l,
			\end{tikzcd}
		\]
		since $\alpha^{k-1}=\id_{X^{k-1}}$ and hence $f^k=\beta^{k-1} \cdots \beta^{s-1} f^{s-1}$ for $k \in \{s,\dots, t\}$. The object in the middle row is projective since $Y^k \in \Proj(\F)$ for all $k \in \{0, \dots, s-1, t+1, \dots, l\}$, see \Cref{thm: mMor}.\ref{thm: mMor-Proj}, and hence $f = 0$ in $\stab \mMor_l(\F)$.\\
		For condition \ref{dfn: sod-2} in \Cref{dfn: sod}, let $X = (X, \alpha) \in \mMor_l(\F)$ be arbitrary, and set $\tilde X := \delta^{[s-1, t]}\gamma^{[s,t]}(X) \in \Delta^{[s-1, t]}$. We prove that the morphism 
		\[
			\begin{tikzcd}[row sep={15mm,between origins}, column sep={16.5mm,between origins}]
				\tilde X \ar[d, "f"'] & X^0 \ar[r, tail, "\alpha^0"] \ar[d, equal] & \cdots \ar[r, tail, "\alpha^{s-2}"] & X^{s-1} \ar[r, equal] \ar[d, equal] & X^{s-1} \ar[r, equal] \ar[d, tail, "\alpha^{s-1}"] & \cdots \ar[r, equal] & X^{s-1} \ar[r, tail, "\alpha^t \cdots \alpha^{s-1}"] \ar[d, tail, "\alpha^{t-1} \cdots \alpha^{s-1}"] & X^{t+1} \ar[r, tail, "\alpha^{t+1}"] \ar[d, equal] & \cdots \ar[r, tail, "\alpha^{l-1}"] & X^l \ar[d, equal] \\
				X & X^0 \ar[r, tail, "\alpha^0"] & \cdots \ar[r, tail, "\alpha^{s-2}"] & X^{s-1} \ar[r, tail, "\alpha^{s-1}"] & X^s \ar[r, tail, "\alpha^s"] & \cdots \ar[r, tail, "\alpha^{t-1}"] & X^{t} \ar[r, tail, "\alpha^t"] & X^{t+1} \ar[r, tail, "\alpha^{t+1}"] & \cdots \ar[r, tail, "\alpha^{l-1}"] & X^l 
			\end{tikzcd}
		\]
		in $\stab \mMor_{l}(\F)$ has its cone $C := C(f)$ in $\Gamma^{[s, t]}$. To this end, we consider the pushout of $f$ along the admissible monic $i := i_{\tilde X}\colon\monic{\tilde X}{I(\tilde X) := I}$ with $i^k =i^{s-1}$ for all $k \in \{s, \dots, t\}$, see \Cref{con: mMor-enough}.\ref{con: mMor-enough-inj}. Due to \Cref{rmk: cone-termwise}, it displays as a termwise pushout:
		\[
			\begin{tikzcd}[column sep={7.5mm,between origins}, row sep={10mm,between origins}]
				& I \ar[dd] && I^{0} \ar[rr, tail] \ar[dd, equal] && \cdots \ar[rr, tail] && I^{s-1} \ar[rr, equal] \ar[dd, equal] && I^{s} \ar[rr, equal] \ar[dd, tail] && \cdots \ar[rr, equal] && I^{t} \ar[rr, tail] \ar[dd, tail] && I^{t+1} \ar[rr, tail] \ar[dd, equal] && \cdots \ar[rr, tail] && I^l \ar[dd, equal]\\
				\tilde X \ar[ru, tail, "i"] && X^{0} \ar[rr, tail, crossing over] \ar[ru, tail] && \cdots \ar[rr, tail] && X^{s-1} \ar[rr, equal, crossing over] \ar[ru, tail] && X^{s-1} \ar[rr, equal, crossing over] \ar[ru, tail] && \cdots \ar[rr, equal] && X^{s-1} \ar[rr, tail, crossing over] \ar[ru, tail] && X^{t+1} \ar[rr, tail, crossing over] \ar[ru, tail] && \cdots \ar[rr, tail] && X^l \ar[ru, tail]\\
				& C && I^{0} \ar[rr, tail] && \cdots \ar[rr, tail] && I^{s-1} \ar[rr, tail] && C^s \ar[rr, tail] && \cdots \ar[rr, tail] && C^t \ar[rr, tail] && I^{t+1} \ar[rr, tail] && \cdots \ar[rr, tail] && I^l\\
				X \ar[ru, tail] \ar[uu, "f", <-] && X^{0} \ar[rr, tail] \ar[ru, tail] \ar[uu, equal] && \cdots \ar[rr, tail] && X^{s-1} \ar[rr, tail] \ar[ru, tail] \ar[uu, equal, crossing over] && X^s \ar[rr, tail] \ar[ru, tail] \ar[uu, <-<, crossing over] && \cdots \ar[rr, tail] && X^t \ar[rr, tail] \ar[ru, tail] \ar[uu, <-<, crossing over] && X^{t+1} \ar[rr, tail] \ar[ru, tail] \ar[uu, equal, crossing over] && \cdots \ar[rr, tail] && X^l \ar[ru, tail] \ar[uu, equal, crossing over]
			\end{tikzcd}
		\]
		In particular, for each $k \in \{s, \dots, t\}$, the square
		\[
			\begin{tikzcd}[row sep={15mm,between origins}, column sep={18.5mm,between origins}]
				X^{s-1} \ar[r, tail, "\alpha^{k-1} \cdots \alpha^{s-1}"] \ar[d, tail] \ar[rd, phantom, "\square"] & X^k \ar[d, tail] \\
				\mathllap{I^{s-1} = \;} I^k \ar[r, tail] & C^k
			\end{tikzcd}
		\]
		is bicartesian and, for $k>s$, so is right square in the diagram
		\[
			\begin{tikzcd}[row sep={15mm,between origins}, column sep={19.5mm,between origins}]
				X^{s-1} \ar[r, tail, "\alpha^{k-2} \cdots \alpha^{s-1}"] \ar[d, tail] \ar[rd, phantom, "\square"] & X^{k-1} \ar[d, tail] \ar[r, tail, "\alpha^{k-1}"] \ar[rd, phantom, "\square"] & X^k \ar[d, tail] \\
				I^{s-1} \ar[r, tail] & C^{k-1} \ar[r, tail] & C^k
			\end{tikzcd}
		\]
		due to \Cref{prp: Buehler2.12}.\ref{prp: Buehler2.12-push} and \Cref{lem: pushpull-sep}.\ref{lem: pushpull-sep-pushout}. We postcompose $X \to C$ with the stable isomorphism $C \to \tilde C$ from \Cref{lem: Gamma-s-t}.\ref{lem: Gamma-s-t-epi} and the one from \Cref{cor: append-I} (applied to the identity morphism of $\gamma^{[t+1,l]}(\tilde C)$ and the monics $ \monic{\tilde C^t}{\tilde C^{t+1} =: J \in \Proj(\F)}$ and $ i_{\tilde C^t}$) to obtain a stably isomorphic cone of $f$:
		\[
			\begin{tikzcd}[row sep={15mm,between origins}, column sep={16.5mm,between origins}]
				X \ar[d] & X^0 \ar[r, tail, "\alpha^0"] \ar[d, tail] & \cdots \ar[r, tail, "\alpha^{s-2}"] & X^{s-1} \ar[r, tail, "\alpha^{s-1}"] \ar[d, tail] \ar[rd, phantom, "\square"] & X^s \ar[r, tail, "\alpha^s"] \ar[d, tail] & \cdots \ar[r, tail, "\alpha^{t-1}"] \ar[rd, phantom, "\square"] & X^{t} \ar[r, tail, "\alpha^t"] \ar[d, tail] & X^{t+1} \ar[r, tail, "\alpha^{t+1}"] \ar[d, tail] & \cdots \ar[r, tail, "\alpha^{l-1}"] & X^l \ar[d, tail] \\
				C \ar[d] & I^{0} \ar[r, tail] \ar[d] & \cdots \ar[r, tail] & I^{s-1} \ar[r, tail] \ar[d, two heads] \ar[rd, phantom, "\square"] & C^s \ar[r, tail] \ar[d, two heads] \ar[ru, phantom, "\square"] & \cdots \ar[r, tail] \ar[rd, phantom, "\square"] & C^t \ar[r, tail] \ar[d, two heads] \ar[rd, phantom, "\square"] & I^{t+1} \ar[r, tail] \ar[d, two heads] & \cdots \ar[r, tail] & I^l \ar[d, two heads] \\
				\tilde C \ar[d] & 0 \ar[r, equal] \ar[d, equal] & \cdots \ar[r, equal] & 0 \ar[r, tail] \ar[d, equal] \ar[rd, phantom, "\square"] & \tilde C^s \ar[r, tail] \ar[ru, phantom, "\square"] \ar[d, equal] & \cdots \ar[r, tail] \ar[rd, phantom, "\square"] & \tilde C^t \ar[r, tail] \ar[d, equal] & J \ar[r, equal] \ar[d] & \cdots \ar[r, equal] & J \ar[d] \\
				\hat C & 0 \ar[r, equal] & \cdots \ar[r, equal] & 0 \ar[r, tail] & \tilde C^s \ar[r, tail] \ar[ru, phantom, "\square"] & \cdots \ar[r, tail] & \tilde C^t \ar[r, tail, "i_{\tilde C^t}"] & I(\tilde C^t) \ar[r, equal] & \cdots \ar[r, equal] & I(\tilde C^t)
			\end{tikzcd}
		\]
		This yields a distinguished triangle
		\[
			\begin{tikzcd}[row sep={15mm,between origins}, column sep={16.5mm,between origins}]
				\tilde X \ar[d, "f"'] & X^0 \ar[r, tail, "\alpha^0"] \ar[d, equal] & \cdots \ar[r, tail, "\alpha^{s-2}"] & X^{s-1} \ar[r, equal] \ar[d, equal] & X^{s-1} \ar[r, equal] \ar[d, tail, "\alpha^{s-1}"] & \cdots \ar[r, equal] & X^{s-1} \ar[r, tail, "\alpha^t \cdots \alpha^{s-1}"] \ar[d, tail, "\alpha^{t-1} \cdots \alpha^{s-1}"] & X^{t+1} \ar[r, tail, "\alpha^{t+1}"] \ar[d, equal] & \cdots \ar[r, tail, "\alpha^{l-1}"] & X^l \ar[d, equal] \\
				X \ar[d] & X^0 \ar[r, tail, "\alpha^0"] \ar[d] & \cdots \ar[r, tail, "\alpha^{s-2}"] & X^{s-1} \ar[r, tail, "\alpha^{s-1}"] \ar[rd, phantom, "\square"] \ar[d, two heads] & X^s \ar[r, tail, "\alpha^s"] \ar[d, two heads] & \cdots \ar[r, tail, "\alpha^{t-1}"] \ar[rd, phantom, "\square"] & X^{t} \ar[r, tail, "\alpha^t"] \ar[d, two heads] & X^{t+1} \ar[r, tail, "\alpha^{t+1}"] \ar[d] & \cdots \ar[r, tail, "\alpha^{l-1}"] & X^l \ar[d] \\
				\hat C & 0 \ar[r, equal] & \cdots \ar[r, equal] & 0 \ar[r, tail] & \tilde C^s \ar[r, tail] \ar[ru, phantom, "\square"] & \cdots \ar[r, tail] & \tilde C^t \ar[r, tail, "i_{\tilde C^t}"] & I(\tilde C^t) \ar[r, equal] & \cdots \ar[r, equal] & I(\tilde C^t)
			\end{tikzcd}
		\]
		in $\stab \mMor_{l}(\F)$ with $\tilde X = \delta^{[s-1, t]} \gamma^{[s,t]}(X) \in \Delta^{[s-1, t]}$ and $\hat C \in \Gamma^{[s,t]}$. \Cref{prp: SOD-adjoints} yields the claim on adjoints. \qedhere
	\end{enumerate}
\end{proof}

\begin{rmk} \label{rmk: alt-adjoint}
	There is an alternative distinguished triangle in \Cref{prp: SODs}.\ref{prp: SODs-Delta-Gamma}: In the proof, we replace the stable isomorphism from \Cref{lem: Gamma-s-t}.\ref{lem: Gamma-s-t-epi} by the one obtained from \Cref{lem: split-off-equalities} applied to $\gamma^{[0, t]}(C)$. This results in the distinguished triangle
	\[
		\begin{tikzcd}[row sep={15mm,between origins}, column sep={16.5mm,between origins}]
		\tilde X \ar[d] & X^0 \ar[r, tail, "\alpha^0"] \ar[d, equal] & \cdots \ar[r, tail, "\alpha^{s-2}"] & X^{s-1} \ar[r, equal] \ar[d, equal] & X^{s-1} \ar[r, equal] \ar[d, tail, "\alpha^{s-1}"] & \cdots \ar[r, equal] & X^{s-1} \ar[r, tail, "\alpha^t \cdots \alpha^{s-1}"] \ar[d, tail, "\alpha^{t-1} \cdots \alpha^{s-1}"] & X^{t+1} \ar[r, tail, "\alpha^{t+1}"] \ar[d, equal] & \cdots \ar[r, tail, "\alpha^{l-1}"] & X^l \ar[d, equal] \\
			X \ar[d, "g"] & X^0 \ar[r, tail, "\alpha^0"] \ar[d, tail] & \cdots \ar[r, tail, "\alpha^{s-2}"] & X^{s-1} \ar[r, tail, "\alpha^{s-1}"] \ar[d, tail] \ar[rd, phantom, "\square"] & X^s \ar[r, tail, "\alpha^s"] \ar[d, tail] & \cdots \ar[r, tail, "\alpha^{t-1}"] \ar[rd, phantom, "\square"] & X^{t} \ar[r, tail, "\alpha^t"] \ar[d, tail] & X^{t+1} \ar[r, tail, "\alpha^{t+1}"] \ar[d] & \cdots \ar[r, tail, "\alpha^{l-1}"] & X^l \ar[d] \\
			Y & I^0 \ar[r, tail] & \cdots \ar[r, tail] & I^{s-1} \ar[r, tail] & Y^s \ar[r, tail] \ar[ru, phantom, "\square"] & \cdots \ar[r, tail] & Y^t \ar[r, tail, "i_{Y^t}"] & I(Y^t) \ar[r, equal] & \cdots \ar[r, equal] & I(Y^t)
		\end{tikzcd}
	\]
	in $\stab \mMor_{l}(\F)$, where $\gamma^{[s,l]}(g)=i_{\gamma^{[s,l]}(X)}$. Thus, the inclusion $\inj{\Gamma^{[s,t]}}{\stab \mMor_{l}(\F)}$ has another left adjoint which sends any $X \in \mMor_{l}(\F)$ to $Y$.
\end{rmk}

\begin{lem}[{\cite[Lem.~1.3]{FS24}}] \label{lem: pushpull-sep}
	Consider the following commutative diagram in an additive category:
	\[
		\begin{tikzcd}[sep={17.5mm,between origins}]
			A \ar[r, "r"] \ar[d, "a"] 
			& B \ar[r, "s"] \ar[d, "b"] & C \ar[d, "c"]\\
			A' \ar[r, "r'"] & B' \ar[r, "s'"] & C'
		\end{tikzcd}
	\]
	\begin{enumerate}
		\item \label{lem: pushpull-sep-pushout} If the outer square is a pushout and $\begin{pmatrix} b & r'\end{pmatrix}\colon B \oplus A' \to B'$ is an epic, then the right square is a pushout.
		\item \label{lem: pushpull-sep-pullback} If the outer square is a pullback and $\begin{pmatrix} s \\ b \end{pmatrix}\colon B \to C \oplus B'$ is a monic, then the left square is a pullback. \qed
	\end{enumerate}
\end{lem}

Combining \Cref{prp: SOD-Gamma-Gamma} and \Cref{prp: SODs} yields

\begin{cor} \label{cor: polygon}
	Let $\F$ be a Frobenius category and $l \in \NN_{\geq 1}$. For each $s\in \{0, \dots, l-1\}$, there is the following $(2l+4)$-gon of recollements in $\stab \mMor_{l}(\F)$:
	\[\Gamma^{[0, s]}, \dots, \underbrace{\Delta^{[k, k+s+1]}, \Gamma^{[k+1, k+s+1]}}_{\text{for } k \, = \, 0, \, \dots, \, l-s-1}, \dots, \Gamma^{[0, l-s-1]}, \dots, \underbrace{\Delta^{[k, k+l-s]}, \Gamma^{[k+1, k+l-s]}}_{\text{for } k \, = \, 0, \, \dots, \, s} \]
	\[
		\begin{tikzcd}[sep=tiny]
			& \Delta^{[0, s+1]} \ar[r, -, marrow=>] & \cdots \ar[r, -, marrow=>] & \Gamma^{[k, k+s]} \ar[r, -, marrow=>] & \Delta^{[k, k+s+1]} \ar[r, -, marrow=>] & \Gamma^{[k+1,k+s+1]} \ar[r, -, marrow=>] & \cdots \ar[r, -, marrow=>] & \Gamma^{[l-s, l]} \ar[rd, -, marrow=>, bend left=5mm] & \\
			\Gamma^{[0, s]} \ar[ru, -, marrow=>, bend left=5mm] &&&&&&&& \Gamma^{[0, l-s-1]} \ar[ld, -, marrow=>, bend left=5mm] \\
			&\Gamma^{[s+1, l]} \ar[lu, -, marrow=>, bend left=5mm] & \cdots \ar[l, -, marrow=>] & \Delta^{[k+1, k+l-s+1]} \ar[l, -, marrow=>] & \Gamma^{[k+1, k+l-s]} \ar[l, -, marrow=>] & \Delta^{[k, k+l-s]} \ar[l, -, marrow=>] & \cdots \ar[l, -, marrow=>] & \Delta^{[0, l-s]} \ar[l, -, marrow=>]
		\end{tikzcd}
	\]
	\pushQED{\qed}
	If $l$ is odd and $s = \frac{l-1}{2}$, it is invariant under index shift by $l+2$ and reduces to the $(l+2)$-gon \[\Gamma^{[0, s]}, \underbrace{\Delta^{[k, k+s+1]}, \dots, \Gamma^{[k+1, k+s+1]}}_{\text{for } k \, = \, 0, \, \dots, \, s}. \qedhere\] 
	 
\end{cor}

\section{Mutation and suspension} \label{sec: mut}

In this section, we explicitly describe the mutations occurring in the polygon of recollements from \Cref{cor: polygon}. We identify the subcategories from \Cref{sec: con-exp} with smaller stable monomorphism categories using a further type of expansion functors. Under these identifications, certain mutations become the identity functor, while the others agree with one particular, non-trivial auto-equivalence. We show that its $(l+2)$nd power coincides with the square of the suspension functor.

\begin{con}[Mutations] \label{con: mutations}
	The adjoints obtained in \Cref{prp: SOD-Gamma-Gamma,prp: SODs} combined with the description of $\Gamma^{[s,t]}$ in \Cref{cor: Gamma-s-t} allow us to explicate the mutations occurring in the polygon from \Cref{cor: polygon}, see \Cref{dfn: mutations}:
	\begin{enumerate}[wide]
		\item \label{con: mutations-1} Due to Propositions \ref{prp: SODs}.\ref{prp: SODs-Delta-Gamma}.\ref{prp: SODs-Delta-Gamma-left} and \ref{prp: SODs}.\ref{prp: SODs-Gamma-Delta}.\ref{prp: SODs-Gamma-Delta-right}, the mutations $L_{\Delta^{[s, t+1]}}$ and $R_{\Delta^{[s, t+1]}}$ are given by\\
			\begin{tikzcd}[column sep=small]
				\Gamma^{[s, t]} \ar[d, "\simeq", "L_{\Delta^{[s, t+1]}}"'] & \cdots \ar[r, equal] & 0 \ar[r, tail] \ar[d, equal] & X^s \ar[r, tail] \ar[d, two heads] \ar[rd, phantom, "\square"] & X^{s+1} \ar[r, tail] \ar[d, two heads] & \cdots \ar[r, tail] \ar[rd, phantom, "\square"] & X^t \ar[r, tail, "i_{X^t}"] \ar[d, two heads] \ar[rd, phantom, "\square"] & I(X^t) \ar[r, equal] \ar[d, two heads] & I(X^t) \ar[r, equal] \ar[d] & \cdots \\
				\Gamma^{[s+1, t+1]} & \cdots \ar[r, equal] & 0 \ar[r, equal] & 0 \ar[r, tail] & Y^{s+1} \ar[r, tail] \ar[ru, phantom, "\square"] & \cdots \ar[r, tail] & Y^t \ar[r, tail] & Y^{t+1} \ar[r, tail, "i_{Y^{t+1}}"] & I(Y^{t+1}) \ar[r, equal] & \cdots,
			\end{tikzcd}\\
			\begin{tikzcd}[column sep=small]
				\Gamma^{[s,t]} & \cdots \ar[r, equal] & 0 \ar[r, tail] \ar[d, equal] & Y^s \ar[r, tail] \ar[d, two heads] \ar[rd, phantom, "\square"] & Y^{s+1} \ar[r, tail] \ar[d, two heads] & \cdots \ar[r, tail] \ar[rd, phantom, "\square"] & Y^t \ar[r, tail] \ar[d, two heads] \ar[rd, phantom, "\square"] & P(X^{t+1}) \ar[r, equal] \ar[d, two heads, "p_{X^{t+1}}"] & P(X^{t+1}) \ar[r, equal] \ar[d] & \cdots\\
				\Gamma^{[s+1, t+1]} \ar[u, "\simeq"', "R_{\Delta^{[s, t+1]}}"]& \cdots \ar[r, equal] & 0 \ar[r, equal] & 0 \ar[r, tail] & X^{s+1} \ar[r, tail] \ar[ru, phantom, "\square"] & \cdots \ar[r, tail] & X^t \ar[r, tail] & X^{t+1} \ar[r, tail, "i_{X^{t+1}}"] & I(X^{t+1}) \ar[r, equal] & \cdots.
			\end{tikzcd}
			
		\item \label{con: mutations-2} Due to Propositions \ref{prp: SODs}.\ref{prp: SODs-Gamma-Delta}.\ref{prp: SODs-Gamma-Delta-left} and \ref{prp: SODs}.\ref{prp: SODs-Delta-Gamma}.\ref{prp: SODs-Delta-Gamma-right}, the mutations $L_{\Gamma^{[s+1, t]}}$ and $R_{\Gamma^{[s+1, t]}}$ are
		\[
			\begin{tikzcd}[column sep=small, row sep={15mm,between origins}]
				\Delta^{[s, t]} \ar[d, "L_{\Gamma^{[s+1, t]}}"', shift right=2.5mm] \ar[d, phantom, "\simeq"]& \cdots \ar[r, tail] & X^{s-1} \ar[r, tail] & X^s \ar[r, equal] & X^s \ar[r, equal] & \cdots \ar[r, equal] & X^s \ar[r, tail] & X^{t+1} \ar[r, tail] & X^{t+2} \ar[r, tail] & \cdots \\
				\Delta^{[s+1, t+1]} \ar[u, "R_{\Gamma^{[s+1, t]}}"', shift right=2.5mm] & \cdots \ar[r, tail] & X^{s-1} \ar[r, tail] & X^s \ar[r, tail] & X^{t+1} \ar[r, equal] & \cdots \ar[r, equal] & X^{t+1} \ar[r, equal] & X^{t+1} \ar[r, tail] & X^{t+2} \ar[r, tail] & \cdots.
			\end{tikzcd}
		\]

		\item \label{con: mutations-3} Due to Propositions \ref{prp: SOD-Gamma-Gamma}.\ref{prp: SOD-Gamma-Gamma-left} and \ref{prp: SODs}.\ref{prp: SODs-Delta-Gamma}.\ref{prp: SODs-Delta-Gamma-right}, the mutations $L_{\Gamma^{[s+1, l]}}$ and $R_{\Gamma^{[s+1, l]}}$ are 
		\[
			\begin{tikzcd}[column sep=small]
				\Delta^{[s, l]} \ar[d, "L_{\Gamma^{[s+1, l]}}"', shift right=2.5mm] \ar[d, phantom, "\simeq"] & X^0 \ar[r, tail] & \cdots \ar[r, tail] & X^{s-1} \ar[r, tail] & X^s \ar[r, equal] & X^s \ar[r, equal] & \cdots \ar[r, equal] & X^s\\
				\Gamma^{[0, s]} \ar[u, "R_{\Gamma^{[s+1, l]}}"', shift right=2.5mm] & X^0 \ar[r, tail] & \cdots \ar[r, tail] & X^{s-1} \ar[r, tail] & X^s \ar[r, tail] & I(X^s) \ar[r, equal] & \cdots \ar[r, equal] & I(X^s),
			\end{tikzcd}
		\]
		where $R_{\Gamma^{[s+1, l]}}$ is given by $\stab \delta^{[s, l]} \circ \stab \gamma^{[s+1, l]}$.

		\item \label{con: mutations-4} Due to Propositions \ref{prp: SODs}.\ref{prp: SODs-Gamma-Delta}.\ref{prp: SODs-Gamma-Delta-left} and \ref{prp: SOD-Gamma-Gamma}.\ref{prp: SOD-Gamma-Gamma-right}, the mutations $L_{\Gamma^{[0,s-1]}}$ and $R_{\Gamma^{[0,s-1]}}$ are 
		\[
			\begin{tikzcd}[column sep=small]
				\Gamma^{[s, l]} \ar[d, "L_{\Gamma^{[0,s-1]}}"', shift right=2.5mm] \ar[d, phantom, "\simeq"] & 0 \ar[r, tail] & \cdots \ar[r, tail] & 0 \ar[r, tail] & X^s \ar[r, tail] & \cdots \ar[r, tail] & X^l \\
				\Delta^{[0, s]} \ar[u, "R_{\Gamma^{[0,s-1]}}"', shift right=2.5mm] & X^s \ar[r, equal] & \cdots \ar[r, equal] & X^s \ar[r, equal] & X^s \ar[r, tail] & \cdots \ar[r, tail] & X^l,
			\end{tikzcd}
		\]
		where $L_{\Gamma^{[0,s-1]}}$ is given by $\stab \delta^{[0, s]} \circ \stab \gamma^{[0, s-1]}$.
	\end{enumerate}
\end{con}

We use the above mutations to identify the subcategory $\Gamma^{[s,t]}$ of $\stab \mMor_{l}(\F)$ with $\stab \mMor_{t-s}(\F)$:

\begin{con}[The expansion functors $\stab \delta^{[s,t]^c}$] \label{con: delta-compl}
	Let $\F$ be a Frobenius category, $l \in \NN$, and $s, t \in \{0, \dots, l\}$ with $s \leq t$.
	\begin{enumerate}[wide]
		\item \label{con: delta-compl-I} We pre- and postcompose the left mutation $L_{\Gamma^{[t+1, l]}}$ from \Cref{con: mutations}.\ref{con: mutations-3} with the triangle equivalence from \Cref{lem: Delta-mMor} given by $\stab \delta^{[t, l]}_t$ and the inclusion $\Gamma^{[t, l]} \subseteq \stab \mMor_{l}(\F)$ to define a fully faithful triangle functor $\stab \delta^{[0, t]^c} := \stab \delta^{[0, t]^c}_t$:
		\[
		\begin{tikzcd}[sep=large]
			\stab \mMor_t(\F) \ar[r, "\stab \delta^{[t, l]}", "\simeq"'] \ar[rrr, dashed, bend left=7.5mm, "\stab \delta^{[0, t]^c}_t"]& \Delta^{[t, l]} \ar[r, "L_{\Gamma^{[t+1, l]}}", "\simeq"'] & \Gamma^{[0, t]} \ar[r, hook] \ar[ll, bend left=7.5mm, "\stab \gamma^{[t+1, l]}", "\simeq"'] & \stab \mMor_{l}(\F),
		\end{tikzcd}
		\]
		whose quasi-inverse on $\Gamma^{[0, t]}$ is the restriction of $\stab \gamma^{[t+1, l]}_l$, since the quasi-inverse $R_{\Gamma^{[t+1, l]}}$ of $L_{\Gamma^{[t+1, l]}}$ is given by $\stab \delta^{[t, l]} \circ \stab \gamma^{[t+1, l]}$. Explicitly, it sends an object $X \in \mMor_{t}(\F)$ to
		\[
			\begin{tikzcd}
				X^0 \ar[r, tail] & \cdots \ar[r, tail] & X^{t} \ar[r, tail] & I(X^{t}) \ar[r, equal] & \cdots \ar[r, equal] & I(X^{t}). \end{tikzcd}
		\]
	
	\item \label{con: delta-compl-0} We pre- and postcompose the right mutation $R_{\Gamma^{[0, s-1]}}$ from \Cref{con: mutations}.\ref{con: mutations-4} with the triangle equivalence from \Cref{lem: Delta-mMor} given by $\stab \delta^{[0, s]}_{l-s}$ and the inclusion $\Gamma^{[s, l]} \subseteq \stab \mMor_{l}(\F)$ to define a fully faithful triangle functor $\stab \delta^{[s, l]^c} := \stab \delta^{[s, l]^c}_{l-s}$:
	\[
	\begin{tikzcd}[sep=large]
		\stab \mMor_{l-s}(\F) \ar[r, "\stab \delta^{[0, s]}", "\simeq"'] \ar[rrr, dashed, bend left=7.5mm, "\stab \delta^{[s, l]^c}_{l-s}"]& \Delta^{[0, s]} \ar[r, "R_{\Gamma^{[0, s-1]}}", "\simeq"'] & \Gamma^{[s, l]} \ar[r, hook] \ar[ll, bend left=7.5mm, "\stab \gamma^{[0, s-1]}", "\simeq"'] & \stab \mMor_{l}(\F),
	\end{tikzcd}
	\]
	whose quasi-inverse on $\Gamma^{[s, l]}$ is the restriction of $\stab \gamma^{[0, s-1]}_l$, since the quasi-inverse $L_{\Gamma^{[0, s-1]}}$ of $R_{\Gamma^{[0, s-1]}}$ is given by $\stab \delta^{[0, s]} \circ \stab \gamma^{[0, s-1]}$. Explicitly, it sends an object $X \in \mMor_{l-s}(\F)$ to
	\[
		\begin{tikzcd}
			0 \ar[r, equal] & \cdots \ar[r, equal] & 0 \ar[r, tail] & X^{0} \ar[r, tail] & \cdots \ar[r, tail] & X^{l-s}. \end{tikzcd}
	\]
	
	\end{enumerate}
	
	Combining \ref{con: delta-compl-I} and \ref{con: delta-compl-0}, we define the fully faithful triangle functor 
	\[\stab \delta^{[s,t]^c} := \stab \delta^{[s,t]^c}_{t-s} := \stab \delta^{[t, l]^c}_t \circ \stab \delta^{[0, s]^c}_{t-s}\colon \inj{\stab \mMor_{t-s}(\F)}{\stab \mMor_{l}(\F)}\] with image $\Gamma^{[s,t]}$. Explicitly, it sends an object $X \in \mMor_{t-s}(\F)$ to
	\[
		\begin{tikzcd}[sep=small]
			0 \ar[r, equal] & \cdots \ar[r, equal] & 0 \ar[r, tail] & X^0 \ar[r, tail] & \cdots \ar[r, tail] & X^{t-s} \ar[r, tail] & I(X^{t-s}) \ar[r, equal] & \cdots \ar[r, equal] & I(X^{t-s}).
		\end{tikzcd}
	\]
	Its quasi-inverse on $\Gamma^{[s,t]}$ is the restriction of $\stab \gamma^{[s, t]^c} = \stab \gamma^{[0, s-1]}_t \circ \stab \gamma^{[t+1, l]}_l$, see \Cref{dfn: contraction-expansion}.\ref{dfn: contraction-expansion-gamma}. We set $\stab \delta^{s^c} := \stab \delta^{[s,s]^c}$.
\end{con}

\begin{con}[Abstract mutations] \label{con: theta}
	Let $\F$ be a Frobenius category, $l \in \NN$, and $s, t \in \{0,\dots, l\}$ with $s \leq t$. We use the equivalences $\stab \mMor_{l-t+s}(\F) \xrightarrow{\simeq} \Delta^{[s,t]}$ and $\stab \mMor_{t-s}(\F) \xrightarrow{\simeq} \Gamma^{[s, t]}$ given by $\stab \delta^{[s, t]}$ and $\stab \delta^{[s, t]^c}$ from \Cref{rmk: stab-delta-gamma}.\ref{rmk: stab-delta-gamma-triang} and \Cref{con: delta-compl}, respectively, as identifications. In this way, we realize the mutations in \Cref{con: mutations} as triangulated auto-equivalences of stable monomorphism categories. While the mutations in \ref{con: mutations-2}, \ref{con: mutations-3}, and \ref{con: mutations-4} become identity functors, \Cref{con: mutations}.\ref{con: mutations-1} yields a non-trivial auto-equivalence:
	\[\begin{tikzcd}[column sep={15mm,between origins}, row sep={20mm,between origins}]
		&& \Gamma^{[s,t]} \ar[rrrr, "L_{\Delta^{[s, t+1]}}", shift left=2mm] \ar[rrrr, phantom, "{\scriptstyle\simeq}"] \ar[rd, hook] &&&& \Gamma^{[s+1, t+1]} \ar[llll, "R_{\Delta^{[s, t+1]}}", shift left=2mm] \ar[llld, hook] \\
		&&& \stab \mMor_{l}(\F) \\
		\stab \mMor_{t-s}(\F) \ar[rrrr, dashed, shift left=2mm, "\Theta"] \ar[rrrr, phantom, "{\scriptstyle\simeq}"] \ar[rrru, "\delta^{[s,t]^c}"] \ar[rruu, "\simeq"] &&&& \stab \mMor_{t-s}(\F) \ar[llll,dashed, shift left=2mm, "\Theta^{-1}"] \ar[lu, "\delta^{[s+1,t+1]^c}"'] \ar[rruu, "\simeq"']
	\end{tikzcd}\]
	Given $X,Y \in \mMor_{t-s}(\F)$, it is determined by the diagram
	\[
		\begin{tikzcd}[sep={15mm,between origins}]
			\mathllap{\Theta^{-1} Y = \; } X & X^0 \ar[r, tail] \ar[d, two heads] \ar[rd, phantom, "\square"] & X^1 \ar[r, tail] \ar[d, two heads] & \cdots \ar[r, tail] \ar[rd, phantom, "\square"] & X^{t-s} \ar[r, tail, "i"] \ar[d, two heads] \ar[rd, phantom, "\square"] & P \ar[d, two heads, "p"] \\
			\mathllap{\Theta X = \; } Y & 0 \ar[r, tail] & Y^0 \ar[r, tail] \ar[ru, phantom, "\square"] & \cdots \ar[r, tail] & Y^{t-s-1} \ar[r, tail] & Y^{t-s}
		\end{tikzcd}
	\]
	of bicartesian squares in $\F$ with $i = i_{X^{t-s}}$ to define $\Theta$ and $p=p_{Y^{t-s}}$ to define $\Theta^{-1}$.
\end{con}

\begin{rmk} \label{rmk: tilde-Theta}
	Let $\F$ be a Frobenius category and $l \in \NN$. Using the left adjoint from \Cref{rmk: alt-adjoint} to define the left mutation $L_{\Delta^{[s, t+1]}}$ in \Cref{con: mutations}.\ref{con: mutations-1}, \Cref{con: theta} yields a functor $\tilde \Theta\colon \stab \mMor_{l}(\F) \to \stab \mMor_{l}(\F)$, isomorphic to $\Theta$, determined by the diagram
	\[
		\begin{tikzcd}[sep={15mm,between origins}]
			\tilde \Theta^{-1}Y = X & X^0 \ar[r, tail] \ar[d, tail, "i_{X^0}"'] \ar[rd, phantom, "\square"] & X^1 \ar[r, tail] \ar[d, tail] & \cdots \ar[r, tail] \ar[rd, phantom, "\square"] & X^l \ar[r, tail, "i_{X^l}"] \ar[d, tail] \ar[rd, phantom, "\square"] & I(X^l) \ar[d, tail] \\
			\tilde \Theta X = Y & I(X^0) \ar[r, tail] & Y^0 \ar[r, tail] \ar[ru, phantom, "\square"] & \cdots \ar[r, tail] & Y^{l-1} \ar[r, tail] & Y^{l}
		\end{tikzcd}
	\]
	of bicartesian squares in $\F$. For any object $X \in \mMor_{l}(\F)$, $\Theta$ and $\tilde \Theta$ are related by $\tilde \Theta X \cong \Theta X \oplus \mu_{l+1}(I(X^0))$, see \eqref{diag: split-off-inj}.
\end{rmk}

\begin{prp} \label{prp: Theta-Sigma-mMor}
	For any Frobenius category $\F$ and all $l \in \NN$, there is an isomorphism $\Theta^{l+2} \cong \Sigma^2$ of endofunctors on $\stab \mMor_{l}(\F)$, where $\Sigma$ is the suspension functor. 
\end{prp}

\begin{proof}
	We prove instead that $\tilde \Theta^{l+2} \cong \Sigma^2$, see \Cref{rmk: tilde-Theta}. For an arbitrary $X \in \mMor_{l}(\F)$, set $X^k_0 := X^k$ for $k \in \{0, \dots, l\}$. The $(l+2)$-fold application of $\tilde \Theta$ is given by a diagram
	\begin{align} \label{diag: Theta-Sigma}
		\begin{small}
			\begin{tikzcd}[sep={12.5mm,between origins}, ampersand replacement=\&]
				X^0_0 \ar[r, tail] \ar[d, tail, "i_{X^0_0}"] \ar[rd, phantom, "\square"] \& X^1_0 \ar[r, tail] \ar[d, tail] \ar[rd, phantom, "\square"] \& X^2_0 \ar[r, tail] \ar[d, tail] \& \cdots \ar[r, tail] \ar[rd, phantom, "\square"] \& X^l_0 \ar[r, tail, "i_{X^l_0}"] \ar[d, tail] \ar[rd, phantom, "\square"] \& J_0 \ar[d, tail] \ar[rd, dashed, tail] \&\&\&\&\& \\
				I_1 \ar[r, tail] \ar[rd, dashed, tail] \& X^0_1 \ar[r, tail] \ar[d, tail, "i_{X^0_1}"] \ar[rd, phantom, "\square"] \& X^1_1 \ar[r, tail] \ar[d, tail] \ar[ru, phantom, "\square"] \& \cdots \ar[r, tail] \ar[rd, phantom, "\square"] \& X^{l-1}_1 \ar[r, tail] \ar[d, tail] \ar[rd, phantom, "\square"] \& X^l_1 \ar[r, tail, "i_{X^l_1}"] \ar[d, tail] \ar[rd, phantom, "\square"] \& J_1 \ar[d, tail] \ar[rd, dashed, tail] \\
				\& I_2 \ar[r, tail] \ar[rd, dashed, tail] \& X^0_2 \ar[r, tail] \ar[ru, phantom, "\square"] \& \cdots \ar[r, tail] \& X^{l-2}_2 \ar[r, tail] \ar[d, tail] \& X^{l-1}_2 \ar[r, tail] \ar[d, tail] \& X^l_2 \ar[r, tail, "i_{X^l_2}"] \ar[d, tail] \& J_2 \ar[d, tail] \ar[rd, dashed, tail] \\
				\&\& \ddots \ar[rd, dashed, tail] \&\& \myvdots \ar[d, tail] \ar[rd, phantom, "\square"] \ar[ru, phantom, "\square"] \& \myvdots \ar[d, tail] \ar[rd, phantom, "\square"] \ar[ru, phantom, "\square"] \& \myvdots \ar[d, tail] \ar[rd, phantom, "\square"] \ar[ru, phantom, "\square"] \& \myvdots \ar[d, tail] \& \ddots \ar[rd, dashed, tail] \\
				\&\&\& I_l \ar[r, tail] \ar[rd, dashed, tail] \& X^0_l \ar[r, tail] \ar[d, tail, "i_{X^0_l}"] \ar[rd, phantom, "\square"] \& X^1_l \ar[r, tail] \ar[d, tail]\ar[rd, phantom, "\square"] \& X^2_l \ar[r, tail] \ar[d, tail] \ar[rd, phantom, "\square"] \& X^3_l \ar[r, tail] \ar[d, tail] \& \cdots \ar[r, tail, "i_{X^l_l}"] \ar[rd, phantom, "\square"] \& J_l \ar[d, tail] \ar[rd, dashed, tail] \\
				\&\&\&\& I_{l+1} \ar[r, tail] \ar[rd, dashed, tail] \& X^0_{l+1} \ar[r, tail] \ar[d, tail, "i_{X^0_{l+1}}"] \ar[rd, phantom, "\square"] \& X^1_{l+1} \ar[r, tail] \ar[d, tail] \ar[rd, phantom, "\square"] \& X^2_{l+1} \ar[r, tail] \ar[d, tail] \ar[ru, phantom, "\square"] \& \cdots \ar[r, tail] \ar[rd, phantom, "\square"] \& X^l_{l+1} \ar[r, tail, "i_{X^l_{l+1}}"] \ar[d, tail] \ar[rd, phantom, "\square"] \& J_{l+1} \ar[d, tail] \\
				\&\&\&\&\& I_{l+2} \ar[r, tail] \& X^0_{l+2} \ar[r, tail] \& X^1_{l+2} \ar[r, tail] \ar[ru, phantom, "\square"] \& \cdots \ar[r, tail] \& X^{l-1}_{l+2} \ar[r, tail] \& X^l_{l+2}
			\end{tikzcd}
		\end{small}
	\end{align}
	of bicartesian squares in $\F$, where $J_j := I(X^l_j)$, $I_{j+1} := I(X^0_j)$ and $X_j = \tilde \Theta^{j} X_0$ for $j \in \{0, \dots, l+2\}$. Consider the following objects of $\mMor_{l}(\F)$:
	\begin{align*}
		\begin{tikzcd}[row sep=small, ampersand replacement=\&]
			Y\colon \; X^l_1 \ar[r, tail] \& X^{l-1}_2 \ar[r, tail] \& \cdots \ar[r, tail] \& X^0_{l+1}, \\
			I\colon \; I_1 \ar[r, tail] \& I_2 \ar[r, tail] \& \cdots \ar[r, tail] \& I_{l+1},\\
			J \colon \; J_1 \ar[r, tail] \& J_2 \ar[r, tail] \& \cdots \ar[r, tail] \& J_{l+1}.
		\end{tikzcd}
	\end{align*}
	Note that $I=I(X)$ and $J=I(Y)$ and that the respective morphisms $X \to I$ and $Y \to J$ defined by \eqref{diag: Theta-Sigma} agree with $i_X$ and $i_Y$, respectively, see \Cref{con: mMor-enough}.\ref{con: mMor-enough-inj}. Using \Cref{lem: Theta-Sigma-prep}, \eqref{diag: Theta-Sigma} yields short exact sequences
	\[
		\begin{tikzcd}[row sep={17.5mm,between origins}]
			X \ar[tail]{r}{\begin{pmatrix} i_X \\ \ast \end{pmatrix}} & I(X) \oplus \mu_{l+1}(J_0) \ar[r, two heads] & Y, \\
			Y \ar[tail]{r}{\begin{pmatrix} i_Y \\ \ast \end{pmatrix}} & I(Y) \oplus \mu_{l+1}(I_{l+2}) \ar[r, two heads] & \tilde \Theta^{l+2} X
		\end{tikzcd}
	\]
	in $\mMor_{l}(\F)$, where $\ast$ denotes unspecified morphisms. As a consequence, we obtain isomorphisms $Y \cong \Sigma X $ and $\tilde \Theta^{l+2} X \cong \Sigma Y \cong \Sigma^2 X$ in $\stab \mMor_{l}(\F)$, functorial in $X$, see \Cref{con: cone}.
\end{proof}

\begin{lem} \label{lem: Theta-Sigma-prep}
	In an exact category $\E$, any diagram
	\[
		\begin{tikzcd}[sep={15mm,between origins}]
			A \ar[r, tail, "a"] \ar[d, tail, "i"] \ar[rd, phantom, "\square"] & A' \ar[r, tail, "a'"] \ar[d, tail, "i'"] \ar[rd, phantom, "\square"] & A'' \ar[d, tail, "i''"] \\
			B \ar[r, tail, "b"] & B' \ar[r, tail, "b'"] \ar[d, tail, "j'"] \ar[rd, phantom, "\square"] & B'' \ar[d, tail, "j''"] \\
			 & C' \ar[r, tail, "c'"] & C''
		\end{tikzcd}
	\]
	of bicartesian squares gives rise to a short exact sequence
	\[
		\begin{tikzcd}[sep={25mm,between origins}, ampersand replacement=\&]
			A \ar[r, tail, "\begin{pmatrix} a'a \\ -i \end{pmatrix}"] \ar[d, tail, "a"] \& A'' \oplus B \ar[two heads]{r}{\begin{pmatrix} i'' & b'b \end{pmatrix}} \ar[tail]{d}{\begin{pmatrix} \id_{A''} & 0 \\ 0 & j'b \end{pmatrix}} \& B'' \ar[d, tail, "j''"] \\
			A' \ar[r, tail, "\begin{pmatrix} a' \\ -j'i' \end{pmatrix}"] \& A'' \oplus C' \ar[two heads]{r}{\begin{pmatrix} j''i'' & c' \end{pmatrix}} \& C''
		\end{tikzcd}
	\]
	in $\mMor_1(\E)$.
\end{lem}

\begin{proof}
	The horizontal short exact sequences in $\E$ arise from \Cref{prp: Buehler2.12}.\ref{prp: Buehler2.12-push}, since concatenation preserves bicartesian squares due to the pasting laws, see {\cite[Ex.~III.4.8]{Mac98}}. The middle morphism is an admissible monic due to \Cref{prp: Buehler2.9}. Commutativity can be checked easily.
\end{proof}

\section{Infinite adjoint chains}

In this section, we establish adjunctions between the contraction and expansion functors from \Cref{sec: con-exp,sec: mut}. In order to form infinite adjoint chains, we introduce two further types of contraction functors. The triangles realizing the semiorthogonal decompositions in \Cref{sec: SOD} can be expressed in terms of the constructed functors.

\begin{lem} \label{lem: adj}
	For any Frobenius category $\F$, $l \in \NN$, and $s, t \in \{0, \dots, l\}$ with $s < t$, there are the following pairs of adjoint functors:
	\begin{enumerate}
		\item \label{lem: adj-standard-1} $\left(\stab \gamma^{[s, t-1]}_l, \stab \delta^{[s, t]}_{l-t+s}\right)$ between the categories $\stab \mMor_{l-t+s}(\F)$ and $\stab \mMor_l(\F)$,
		
		\item \label{lem: adj-standard-2} $\left(\stab \delta^{[s, t]}_{l-t+s}, \stab \gamma^{[s+1, t]}_l\right)$ between the categories $\stab \mMor_l(\F)$ and $\stab \mMor_{l-t+s}(\F)$,
		
		\item \label{lem: adj-boundary-right} $\left(\stab \gamma^{[t, l]}_l, \stab \delta^{[0, t-1]^c}_{t-1}\right)$ between the categories $\stab \mMor_{t-1}(\F)$ and $\stab \mMor_l(\F)$, 
		
		\item \label{lem: adj-boundary-left} $\left( \stab \delta^{[s+1, l]^c}_{l-s-1}, \stab \gamma^{[0, s]}_l \right)$ between the categories $\stab \mMor_l(\F)$ and $\stab \mMor_{l-s-1}(\F)$.
	\end{enumerate}
\end{lem}

\begin{proof}
	For \ref{lem: adj-standard-1} and \ref{lem: adj-standard-2}, note that the left adjoint of the embedding $\inj{\Delta^{[s,t]}}{\stab \mMor_{l}(\F)}$ is given by $\stab \delta^{[s, t]} \circ \stab \gamma^{[s, t-1]}$, the right adjoint by $\stab \delta^{[s, t]} \circ \stab \gamma^{[s+1, t]}$, see \Cref{prp: SODs}.\ref{prp: SODs-Gamma-Delta}.\ref{prp: SODs-Gamma-Delta-left} and \ref{prp: SODs-Delta-Gamma}.\ref{prp: SODs-Delta-Gamma-right}. Together with the equivalence from \Cref{lem: Delta-mMor} and its quasi-inverse, they fit in commutative diagrams of adjoint functors
	\[
		\begin{tikzcd}[sep={3cm,between origins}]
			\stab \mMor_{l-t+s}(\F) \ar[r, "\simeq"] \ar[rr, bend right=5mm, "\stab \delta^{[s,t]}"']& \Delta^{[s, t]} \ar[r, hook] \ar[l, bend right=10mm, "\simeq"'] & \stab \mMor_l(\F), \ar[l, bend right=10mm, "\stab \delta^{[s, t]} \circ \stab \gamma^{[s, t-1]}"', pos=0.5325] \ar[ll, bend right=15mm, "\stab \gamma^{[s, t-1]}"'] 
		\end{tikzcd}
		\begin{tikzcd}[sep={3cm,between origins}]
			\stab \mMor_{l-t+s}(\F) \ar[r, "\simeq"] \ar[rr, bend left=5mm, "\stab \delta^{[s,t]}"] & \Delta^{[s, t]} \ar[r, hook] \ar[l, bend left=10mm, "\simeq"] & \stab \mMor_l(\F). \ar[l, bend left=10mm, "\stab \delta^{[s, t]} \circ \stab \gamma^{[s+1, t]}"] \ar[ll, bend left=15mm, "\stab \gamma^{[s+1, t]}"]
		\end{tikzcd}
	\]
	By composition, this yields the desired pairs of adjoint functors. For \ref{lem: adj-boundary-right} and \ref{lem: adj-boundary-left}, note that the left adjoint of $\inj{\Gamma^{[0, t-1]}}{\stab \mMor_{l}(\F)}$ is given by $\stab \delta^{[0, t-1]^c} \circ \stab \gamma^{[t, l]}$, the right adjoint of $\inj{\Gamma^{[s+1, l]}}{\stab \mMor_{l}(\F)}$ by $\stab \delta^{[s+1, l]^c} \circ \stab \gamma^{[0, s]}$, see \Cref{prp: SOD-Gamma-Gamma}.\ref{prp: SOD-Gamma-Gamma-right} and \ref{prp: SOD-Gamma-Gamma-left} and \Cref{con: delta-compl}. We obtain the following commutative diagrams of adjoint functors, which yield the claim by composition:
	\[
	\begin{tikzcd}[sep={1.5cm,between origins}]
		\stab \mMor_{t-1}(\F) \ar[rr, "\simeq"] \ar[rrrr, bend right=5mm, "\stab \delta^{[0,t-1]^c}"'] && \Gamma^{[0, t-1]} \ar[rr, hook] \ar[ll, bend right=10mm, "\simeq"'] && \stab \mMor_{l}(\F) \ar[ll, bend right=10mm, "\stab \delta^{[0, t-1]^c} \circ \stab \gamma^{[t, l]}"', pos=0.6] \ar[llll, bend right=15mm, "\stab \gamma^{[t, l]}"']
	\end{tikzcd}
	\]
	\[
	\begin{tikzcd}[sep={1.5cm,between origins}]
		\stab \mMor_{l-s-1}(\F) \ar[rr, "\simeq"] \ar[rrrr, bend left=5mm, "\stab \delta^{[s+1,l]^c}"] && \Gamma^{[s+1, l]} \ar[rr, hook] \ar[ll, bend left=10mm, "\simeq"] && \stab \mMor_{l}(\F) \ar[ll, bend left=10mm, "\stab \delta^{[s+1, l]^c} \circ \stab \gamma^{[0, s]}"] \ar[llll, bend left=15mm, "\stab \gamma^{[0, s]}"]
	\end{tikzcd}
	\]
\end{proof}

We need two more functors to form an infinite adjoint chain:

\begin{con}[The contraction functors $\stab {\hat \gamma}^{[s, t]^c}$ and $\stab {\check \gamma}^{[s, t]^c}$] \label{con: gamma-compl}
	Let $\F$ be a Frobenius category, $l \in \NN$, and $s, t \in \{0, \dots, l\}$ with $s \leq t$.
	\begin{enumerate}[wide]
		\item \label{con: gamma-compl-I} Postcomposing the left adjoint of $\inj{\Gamma^{[s, t]}}{\stab \mMor_{l}(\F)}$, see Propositions \ref{prp: SODs}.\ref{prp: SODs-Delta-Gamma}.\ref{prp: SODs-Delta-Gamma-left} and \ref{prp: SOD-Gamma-Gamma}.\ref{prp: SOD-Gamma-Gamma-left}, with the triangle functor $\stab \gamma^{[s, t]^c}$ restricted to $\Gamma^{[s, t]}$, see \Cref{con: delta-compl}, yields a triangle functor $\stab {\hat \gamma}^{[s, t]^c} := \stab {\hat \gamma}^{[s, t]^c}_l$:
		\[
			\begin{tikzcd}[sep={1.5cm,between origins}]
				\stab \mMor_{t-s}(\F) \ar[rr, "\simeq"'] \ar[rrrr, "\stab \delta^{[s, t]^c}"', bend right=5mm] && \Gamma^{[s, t]} \ar[rr, hook] \ar[ll, bend right=10mm, "\stab \gamma^{[s, t]^c}"' near start, "\simeq" near start] && \stab \mMor_l(\F). \ar[ll, bend right=10mm] \ar[llll, bend right=15mm, "\stab {\hat \gamma}^{[s, t]^c}_l"', dashed]
			\end{tikzcd}
		\]
		If $s>0$, it sends an object $X \in \mMor_{l}(\F)$ to $Y \in \mMor_{t-s}(\F)$ given by the following commutative diagram:
		\[
			\begin{tikzcd}[sep={15mm,between origins}]
				X^{s-1} \ar[r, tail] \ar[d, two heads] \ar[rd, phantom, "\square"] & X^s \ar[r, tail] \ar[d, two heads] & \cdots \ar[r, tail] \ar[rd, phantom, "\square"] & X^{t} \ar[d, two heads]\\
				0 \ar[r, tail] & Y^0 \ar[r, tail] \ar[ru, phantom, "\square"] & \cdots \ar[r, tail] & Y^{t-s}
			\end{tikzcd}
		\]
		For $s=0$, we have $\stab {\hat \gamma}^{[0, t]^c} = \stab \gamma^{[0,t]^c} = \stab \gamma^{[t+1,l]}$.
		We set $\stab {\hat \gamma}_l^{s^c} := \stab {\hat \gamma}^{s^c}_l := \stab {\hat \gamma}^{[s,s]^c}_l$.
		
		\item \label{con: gamma-compl-P} Postcomposing the right adjoint of $\inj{\Gamma^{[s, t]}}{\stab \mMor_{l}(\F)}$, see Propositions \ref{prp: SODs}.\ref{prp: SODs-Gamma-Delta}.\ref{prp: SODs-Gamma-Delta-right} and \ref{prp: SOD-Gamma-Gamma}.\ref{prp: SOD-Gamma-Gamma-right}, with the triangle functor $\stab \gamma^{[s, t]^c}$ restricted to $\Gamma^{[s, t]}$, see \Cref{con: delta-compl}, yields a triangle functor $\stab {\check \gamma}^{[s, t]^c} := \stab {\check \gamma}^{[s, t]^c}_l$:
		\[
			\begin{tikzcd}[sep={1.5cm,between origins}]
				\stab \mMor_{t-s}(\F) \ar[rr, "\simeq"] \ar[rrrr, "\stab \delta^{[s, t]^c}", bend left=5mm] && \Gamma^{[s, t]} \ar[rr, hook] \ar[ll, bend left=10mm, "\stab \gamma^{[s, t]^c}" near start, "\simeq"' near start] && \stab \mMor_l(\F). \ar[ll, bend left=10mm] \ar[llll, bend left=15mm, "\stab {\check \gamma}^{[s, t]^c}_l", dashed]
			\end{tikzcd}
		\]
		If $t <l$, it sends an object $X \in \mMor_{l}(\F)$ to $Y \in \mMor_{t-s}(\F)$ given by the following commutative diagram:
		\[
			\begin{tikzcd}[column sep=small]
				Y^0 \ar[r, tail] \ar[d, two heads] & \cdots \ar[r, tail] \ar[rd, phantom, "\square"] & Y^{t-s} \ar[r, tail] \ar[d, two heads] \ar[rd, phantom, "\square"] & P(X^{t+1}) \ar[d, two heads, "p_{X^{t+1}}"] \\
				X^s \ar[r, tail] \ar[ru, phantom, "\square"] & \cdots \ar[r, tail] & X^t \ar[r, tail] & X^{t+1}
			\end{tikzcd}
		\]
		For $t=l$, we have $\stab {\check \gamma}^{[s, l]^c}_l = \stab {\gamma}^{[s, l]^c}_l = \stab {\gamma}^{[0, s-1]}_l$. We set $\stab {\check \gamma}^{s^c} := \stab {\check \gamma}^{s^c}_l := \stab {\check \gamma}^{[s,s]^c}_l$.
	\end{enumerate}
\end{con}

\begin{lem} \label{lem: adj-compl}
	For any Frobenius category $\F$, $l \in \NN$, and $s, t \in \{0, \dots, l\}$ with $s < t$, there are the following pairs of adjoint functors:
	\begin{enumerate}
		\item $\left(\stab {\hat \gamma}^{[s, t]^c}_l, \stab \delta^{[s, t]^c}_{t-s}\right)$ between the categories $\stab \mMor_{t-s}(\F)$ and $\stab \mMor_l(\F)$,
		
		\item $\left(\stab \delta^{[s, t]^c}_{t-s}, \stab {\check \gamma}^{[s, t]^c}_l\right)$ between the categories $\stab \mMor_l(\F)$ and $\stab \mMor_{t-s}(\F)$.
	\end{enumerate}
	
\end{lem}

\begin{proof}
	This is immediate from \Cref{con: gamma-compl}.
\end{proof}

\begin{lem} \label{lem: adj-compl-mut}
	For any Frobenius category $\F$, $l \in \NN$, and $s, t \in \{0, \dots, l\}$ with $s \leq t$, there are the following pairs of adjoint functors:
	\begin{enumerate}
		\item $\left(\stab \delta^{[s-1, t-1]^c}_{t-s} \circ \Theta^{-1}, \stab {\hat \gamma}^{[s, t]^c}_l\right)$ between the categories $\stab \mMor_l(\F)$ and $\stab \mMor_{t-s}(\F)$ for $s>0$,
		\item $\left(\stab {\check \gamma}^{[s, t]^c}_l, \stab \delta^{[s+1, t+1]^c}_{t-s} \circ \Theta \right)$ between the categories $\stab \mMor_{t-s}(\F)$ and $\stab \mMor_l(\F)$ for $t < l$.
	\end{enumerate}
\end{lem}

\begin{proof}
	Due to \Cref{con: theta} and \Cref{rmk: mutations}, we have the following commutative diagrams of adjoint functors, see \Cref{con: delta-compl}:
	\[
		\begin{tikzcd}[sep={2cm,between origins}]
			& \stab \mMor_{t-s}(\F) \ar[rr, "\simeq"] \ar[rrrd, bend left=25mm, "\stab \delta^{[s-1,t-1]^c}"] && \Gamma^{[s-1, t-1]} \ar[rd, hook] & \\
			\stab \mMor_{t-s}(\F) \ar[ru, "\Theta^{-1}", "\simeq"'] \ar[rr, "\stab \delta^{[s, t]^c}", "\simeq"'] 
			&& \Gamma^{[s, t]} \ar[ll, "\stab \gamma^{[s, t]^c}" near start, "\simeq"' near start, bend left=7.5mm] \ar[ru, "R_{\Delta^{[s-1, t]}}", "\simeq"' near end] \ar[rr, hook'] \ar[rr, bend left=10mm]
			&& \stab \mMor_l(\F), \ar[ll, bend right=5mm] \ar[llll, bend left=10mm, "\stab {\hat \gamma}^{[s,t]^c}"] 
		\end{tikzcd}
	\]
	\[
		\begin{tikzcd}[sep={2cm,between origins}]
			\stab \mMor_{t-s}(\F) \ar[rr, "\stab \delta^{[s, t]^c}"', "\simeq"] \ar[rd, "\Theta"', "\simeq"]
			&& \Gamma^{[s, t]} \ar[ll, "\stab \gamma^{[s, t]^c}"' near start, "\simeq" near start, bend right=7.5mm] \ar[rd, "L_{\Delta^{[s, t+1]}}"', "\simeq" near end] \ar[rr, hook] \ar[rr, bend right=10mm]
			&& \stab \mMor_l(\F). \ar[ll,bend left=5mm] \ar[llll, bend right=10mm, "\stab {\check \gamma}^{[s, t]^c}"'] \\
			&\stab \mMor_{t-s}(\F) \ar[rr, "\simeq"']\ar[rrru, bend right=25mm, "\stab \delta^{[s+1,t+1]^c}"'] && \Gamma^{[s+1, t+1]} \ar[ru, hook]
		\end{tikzcd}
	\]
\end{proof}

Combining \Cref{lem: adj,lem: adj-compl,lem: adj-compl-mut} we obtain

\begin{thm} \label{thm: adj-chain}
	 Let $\F$ be a Frobenius category $\F$, $l \in \NN$, and $s, t \in \{0, \dots, l\}$ with $s \leq t$. There is the infinite adjoint chain
	 {\setlength{\jot}{10pt}
	 	\begin{gather*}
	 		\cdots \dashv \Theta^{t-s+1} \stab {\hat \gamma}^{[0, l-t+s-1]^c}\\
	 		\dashv \stab \delta^{[0, l-t+s-1]^c} \Theta^{-t+s-1} \dashv \cdots \dashv \Theta \stab {\hat \gamma}^{[t-s, l-1]^c} \dashv \stab \delta^{[t-s, l-1]^c} \Theta^{-1} \dashv \stab {\hat \gamma}^{[t-s+1, l]^c} \dashv \stab \delta^{[t-s+1, l]^c} \\
	 		\dashv \stab \gamma^{[0, t-s]} \dashv \cdots \dashv \stab \gamma^{[s-1,t-1]} \dashv \stab \delta^{[s-1, t]} \dashv \stab \gamma^{[s, t]} \dashv \stab \delta^{[s, t+1]} \dashv \stab \gamma^{[s+1,t+1]} \dashv \cdots \dashv \stab \gamma^{[l-t+s, l]} \\
	 		\dashv \stab \delta^{[0, l-t+s-1]^c} \dashv \stab{\check \gamma}^{[0, l-t+s-1]^c} \dashv \stab \delta^{[1, l-t+s]^c} \Theta \dashv \Theta^{-1} \stab{\check \gamma}^{[1, l-t+s]^c} \dashv \cdots \dashv \stab \delta^{[t-s+1, l]^c} \Theta^{t-s+1} \\
	 		\dashv \Theta^{-t+s-1} \stab{\check \gamma}^{[t-s+1, l]^c} \dashv \cdots,
	 	\end{gather*}
	 }
	 where $\stab {\hat \gamma}^{[0, l-t+s-1]^c}=\stab { \gamma}^{[l-t+s,l]}$ and $\stab{\check \gamma}^{[t-s+1, l]^c}=\stab \gamma^{[0,t-s]}$, see \Cref{con: gamma-compl}. \qed
\end{thm}

Using \Cref{cor: append-I} and the functors from \Cref{con: delta-compl,con: gamma-compl}, we can interpret the distinguished triangles in \Cref{prp: SOD-Gamma-Gamma,prp: SODs} in the language of \Cref{prp: SOD-adjoints}:

\begin{cor} 
	Let $\F$ be a Frobenius category, $l \in \NN$, and $s, t \in \{0, \dots, l\}$ with $s \leq t$. Any $X \in \stab \mMor_{l}(\F)$ fits into distinguished triangles as follows:
	\begin{enumerate}[wide]
		\item \label{cor: sos-triangles-Gamma-Gamma} For $s < l$, we have
		\[
			\begin{tikzcd}[sep={15mm,between origins}]
				\stab \delta^{[s+1, l]^c} \stab \gamma^{[0,s]}(X) \ar[d] & 0 \ar[r, equal] \ar[d] & \cdots \ar[r, equal] & 0 \ar[r, tail] \ar[d] & X^{s+1} \ar[r, tail, "\alpha^{s+1}"] \ar[d, equal] & X^{s+2} \ar[r, tail, "\alpha^{s+2}"] \ar[d, equal] & \cdots \ar[r, tail, "\alpha^{l-1}"] & X^l \ar[d, equal] \\
				X \ar[d] & X^0 \ar[r, tail, "\alpha^0"] \ar[d, equal] & \cdots \ar[r, tail, "\alpha^{s-1}"] & X^{s} \ar[r, tail, "\alpha^{s}"] \ar[d, equal] & X^{s+1} \ar[r, tail, "\alpha^{s+1}"] \ar[d, ] & X^{s+2} \ar[r, tail, "\alpha^{s+2}"] \ar[d] & \cdots \ar[r, tail, "\alpha^{l-1}"] & X^l \ar[d] \\
				\stab \delta^{[0, s]^c} \stab \gamma^{[s+1,l]}(X) &X^0 \ar[r, tail, "\alpha^0"] & \cdots \ar[r, tail, "\alpha^{s-1}"] & X^{s} \ar[r, tail, "i_{X^s}"] & I(X^{s}) \ar[r, equal] & I(X^{s}) \ar[r, equal] & \cdots \ar[r, equal] & I(X^{s}).
			\end{tikzcd}
		\]
		
		\item \label{cor: sos-triangles-Gamma-Delta} For $t < l$, we have
		\[
			\begin{tikzcd}[sep={15mm,between origins}]
				\stab \delta^{[s, t]^c} \stab {\check \gamma}^{[s,t]^c}(X) \ar[d] & 0 \ar[r, equal] \ar[d, tail] & \cdots \ar[r, equal] & 0 \ar[r, tail] \ar[d, tail] & Y^s \ar[r, tail] \ar[d, two heads] & \cdots \ar[r, tail] \ar[rd, phantom, "\square"] & Y^t \ar[r, tail, "i_{Y^t}"] \ar[d, two heads] & I(Y^t) \ar[r, equal] \ar[d] & \cdots \ar[r, equal] & I(Y^t) \ar[d]\\
				X \ar[d] & X^0 \ar[r, tail, "\alpha^0"] \ar[d, equal] & \cdots \ar[r, tail, "\alpha^{s-2}"] & X^{s-1} \ar[r, tail, "\alpha^{s-1}"] \ar[d, equal] & X^s \ar[r, tail, "\alpha^s"] \ar[d, tail, "\alpha^t \cdots \alpha^s"] \ar[ru, phantom, "\square"] & \cdots \ar[r, tail, "\alpha^{t-1}"] & X^t \ar[r, tail, "\alpha^t"] \ar[d, tail, "\alpha^t"] & X^{t+1} \ar[r, tail, "\alpha^{t+1}"] \ar[d, equal] & \cdots \ar[r, tail, "\alpha^{l-1}"] & X^l \ar[d, equal]\\
				\stab \delta^{[s, t+1]} \stab \gamma^{[s,t]}(X) & X^0 \ar[r, tail, "\alpha^0"] & \cdots \ar[r, tail, "\alpha^{s-2}"] & X^{s-1} \ar[r, tail] & X^{t+1} \ar[r, equal] & \cdots \ar[r, equal] & X^{t+1} \ar[r, equal] & X^{t+1} \ar[r, tail, "\alpha^{t+1}"] & \cdots \ar[r, tail, "\alpha^{l-1}"] & X^l,
			\end{tikzcd}
		\]
		where $Y^t$ is defined by
		\[
			\begin{tikzcd}[sep={15mm,between origins}]
				Y^t \ar[r, tail] \ar[d, two heads] \ar[rd, "\square", phantom] & P(X^{t+1}) \ar[d, two heads, "p_{X^{t+1}}"]\\
				X^t \ar[r, tail, "\alpha^t"] & X^{t+1}.
			\end{tikzcd}
		\]
		\item \label{cor: sos-triangles-Delta-Gamma} For $0 < s$, we have
		\[
			\begin{tikzcd}[sep={15mm,between origins}]
				\stab \delta^{[s-1, t]} \stab \gamma^{[s,t]}(X) \ar[d] & X^0 \ar[r, tail, "\alpha^0"] \ar[d, equal] & \cdots \ar[r, tail, "\alpha^{s-2}"] & X^{s-1} \ar[r, equal] \ar[d, equal] & X^{s-1} \ar[r, equal] \ar[d, tail, "\alpha^{s-1}"] & \cdots \ar[r, equal] & X^{s-1} \ar[r, tail] \ar[d, tail, "\alpha^{t-1} \cdots \alpha^{s-1}"] & X^{t+1} \ar[r, tail, "\alpha^{t+1}"] \ar[d, equal] & \cdots \ar[r, tail, "\alpha^{l-1}"] & X^l \ar[d, equal] \\
				X \ar[d] & X^0 \ar[r, tail, "\alpha^0"] \ar[d, two heads] & \cdots \ar[r, tail, "\alpha^{s-2}"] & X^{s-1} \ar[r, tail, "\alpha^{s-1}"] \ar[d, two heads] \ar[rd, phantom, "\square"] & X^s \ar[r, tail, "\alpha^s"] \ar[d, two heads] & \cdots \ar[r, tail, "\alpha^{t-1}"] \ar[rd, phantom, "\square"] & X^{t} \ar[r, tail, "\alpha^t"] \ar[d, two heads] & X^{t+1} \ar[r, tail, "\alpha^{t+1}"] \ar[d] & \cdots \ar[r, tail, "\alpha^{l-1}"] & X^l \ar[d] \\
				\stab \delta^{[s, t]^c} \stab {\hat \gamma}^{[s,t]^c}(X) & 0 \ar[r, equal] & \cdots \ar[r, equal] & 0 \ar[r, tail] & Y^s \ar[r, tail] \ar[ru, phantom, "\square"] & \cdots \ar[r, tail] & Y^t \ar[r, tail, "i_{Y^t}"] & I(Y^t) \ar[r, equal] & \cdots \ar[r, equal] & I(Y^t).
			\end{tikzcd}
		\]
		 \qed
	\end{enumerate}
\end{cor}

\section{Dualized hom-functors}

In this section, we review the construction of Bondal and Kapranov to lift representations of cohomological functors from semiorthogonal decompositions. Based on the decompositions from \Cref{sec: SOD}, we lift representations of dualized hom-functors from an algebraic triangulated category to the larger stable monomorphism categories.

\begin{ntn} \label{ntn: trans} Let $\T$ be a category, linear over a field $K$, and denote the category of $K$-vector spaces by $\Vect$.
	\begin{enumerate}
		\item \label{ntn: trans-h} For objects $X, Y \in \T$, we abbreviate 
		\[h_X := h_X^\T := \Hom_\T(-, X), \hspace{2.5mm} h^X := h^X_\T := \Hom_\T(X, -), \hspace{2.5mm} \text{and} \hspace{2.5mm} h^X_Y := \Hom_{\T}(X,Y). \]
		
		\item \label{ntn: trans-u} Consider a contravariant linear functor $h\colon \T \to \mathrm{Vect}$, an object $X \in \T$, and the set $\mathrm{Nat}(h_X, h)$ of natural transformations. The Yoneda lemma yields a bijection
		\[
			\begin{tikzcd}[row sep={7.5mm,between origins}, column sep={15mm,between origins}]
				\mathllap{\mathrm{Nat}(h_X, h)} \ar[r, <->] & \mathrlap{h(X),} \\
				\mathllap{\eta} \ar[r, |->] & \mathrlap{\eta_X(\id_X) =: e_\eta,} \\
				\mathllap{\left( h^Y_X \ni f \mapsto h(f)(e) \in h(Y) \right)_{Y \in \T}} & \mathrlap{e.} \ar[l, |->]
			\end{tikzcd}
		\]
		Note that any such natural transformation is automatically (component-wise) linear.
		\item \label{ntn: trans-u_X} We abbreviate $(-)^\ast := \Hom_K(-, K)$. Given a natural transformation $\eta_X \colon h^\T_{\tilde X} \to \left(h^X_\T\right)^\ast$ of contravariant functors $\T \to \Vect$, where $X, {\tilde X} \in \T$, we denote $e_X := e_{\eta_X} \in \left(h^X_{\tilde X}\right)^\ast$, see \ref{ntn: trans-u}.
	\end{enumerate}
\end{ntn}

We include the following statement for lack of reference:

\begin{lem} \label{lem: rep-comp}
	Let $\T$ be a triangulated category, linear over a field. Suppose that there are isomorphisms of functors
	\[\eta_X \colon h^\T_{\tilde X} \cong \left(h^X_\T\right)^\ast \hspace{2.5mm} \text{and} \hspace{2.5mm} \eta_Y \colon h^\T_{\tilde Y} \cong \left(h^Y_\T\right)^\ast,\]
	where $X, Y, \tilde X, \tilde Y \in \T$. Then each $f \in \Hom_\T(X, Y)$ defines a unique $\tilde f \in \Hom_\T(\tilde X, \tilde Y)$ such that $e_X(- \circ f) = e_Y(\tilde f \circ -)$ on $\Hom_\T(Y, \tilde X)$, see \Cref{ntn: trans}.\ref{ntn: trans-u_X}. This assignment is compatible with compositions.
\end{lem}

\begin{proof}
	By the Yoneda lemma, any morphism $f \in \Hom_\T(X, Y)$ corresponds to a unique morphism $\tilde f \in \Hom_\T(\tilde X, \tilde Y)$ via the following commutative diagram of natural transformations:
	\[
		\begin{tikzcd}[row sep={17.5mm,between origins}, column sep={32.5mm,between origins}]
			h^\T_{\tilde X} \ar[r, "\eta_X", "\cong"'] \ar[d, dashed, "h^\T_{\tilde f}"] & \left(h^X_\T\right)^\ast \ar[d, "\left(h^f_\T\right)^\ast"] \\
			h^\T_{\tilde Y} \ar[r, "\eta_Y", "\cong"'] & \left(h^Y_\T\right)^\ast,
		\end{tikzcd}
		\begin{tikzcd}[row sep={17.5mm,between origins}, column sep={32.5mm,between origins}]
			\id_{\tilde X} \ar[r, <->] \ar[d, |->] & e_X \ar[d, |->] \\
			\tilde f \ar[r, <->] & e_Y(\tilde f \circ -) = e_X(- \circ f).
		\end{tikzcd}
	\]
	This yields the first claim. For the second, consider another natural transformation $\eta_Z \colon h^\T_{\tilde Z} \cong \left(h^Z_\T\right)^\ast$, where $Z, \tilde Z \in \T$, $g \in \Hom_\T(Y, Z)$, and $\tilde g \in \Hom_\T(\tilde Y, \tilde Z)$ such that $e_Y(- \circ g) = e_Z(\tilde g \circ -)$. Then $e_X(- \circ gf) = e_Y(\tilde f \circ - \circ g) = e_Z(\tilde g \tilde f \circ -)$, and hence~$\tilde g \tilde f = \widetilde{gf}$ by uniqueness.
\end{proof}

We review the construction of Bondal and Kapranov from the proofs of {\cite[Thm.~2.10, Prop.~3.8]{BK89}}, which is explicitly used for the main result of this section.

\begin{thm} \label{thm: BK-2.10}
	Let $\T$ be a triangulated category, linear over a field, with a semiorthogonal decomposition $(\U, \V)$. If all contravariant linear cohomological functors $\U \to \mathrm{Vect}$ and $\V \to \mathrm{Vect}$ are representable, then so are all such functors $\T \to \mathrm{Vect}$.
\end{thm}

\begin{proof} Let $h \colon \T \to \mathrm{Vect}$ be a contravariant linear cohomological functor. We proceed in several steps.
	\begin{enumerate}[wide, label=(\arabic*)]
		\item \label{thm: BK-2.10-start} By assumption, there are objects $\tilde U \in \U$, $\tilde V, V' \in \V$, and isomorphisms of functors
		\[\eta^\U\colon h_{\tilde U}^\U \cong h\vert_\U, \hspace{5mm} \eta^\V\colon h_{\tilde V}^\V \cong h\vert_\V, \hspace{5mm} \text{ and } \hspace{5mm} \theta\colon h_{V'}^\V \cong h_{\tilde U}^\V.\]
		Set $e_{\U} := e_{\eta^\U}$ and $e_{\V} := e_{\eta^\V}$, see \Cref{ntn: trans}.\ref{ntn: trans-u}.
		
		\item \label{thm: BK-2.10-v} Let $u := e_\theta \in h^{V'}_{\tilde U}$, and define $v \in h^{V'}_{\tilde V}$ by the commutative square
		\begin{align} \label{diag: def-v}
			\begin{tikzcd}[column sep={30mm,between origins}, row sep={17.5mm,between origins}, ampersand replacement=\&]
				h^{\tilde U}_{\tilde U} \ar[d, "\eta^\U_{\tilde U}", "\cong"'] \ar[r, dashed] \& h^{V'}_{\tilde V} \ar[d, "\eta^\V_{V'}", "\cong"'] \& \id_{\tilde U} \ar[d, <->] \ar[r, |->, dashed] \& v \ar[d, <->] \\
				h(\tilde U) \ar[r, "h(u)"]\& h(V'), \&	e_{\U} \ar[r, |->] \& h(u)(e_{\U}) = h(v)(e_{\V}).
			\end{tikzcd}
		\end{align}
		
		\item \label{thm: BK-2.10-hom-PO} Complete the morphism defined by $u$ and $-v$ to a distinguished triangle in $\T$ as follows:
		\begin{align} \label{diag: total-triangle}
			\begin{tikzcd}[ampersand replacement=\&]
				V' \ar{r}{\begin{pmatrix} u \\ -v \end{pmatrix}} \& \tilde U \oplus \tilde V \ar{r}{\begin{pmatrix} p & q \end{pmatrix}} \& \tilde X
			\end{tikzcd}
		\end{align}
		It defines a homotopy cartesian square, which fits into a morphism
		\begin{align} \label{diag: hom-cart}
			\begin{tikzcd}[sep={17.5mm,between origins}, ampersand replacement=\&]
				V'' \ar[r] \ar[d, equal] \& V' \ar[r, "v"] \ar[d, "u"] \ar[rd, phantom, "\square"] \& \tilde V \ar[d, "q"] \\
				V'' \ar[r, dashed] \& \tilde U \ar[r, "p"] \& \tilde X
			\end{tikzcd}
		\end{align}
		of distinguished triangles in $\T$, see {\cite[Lem.~1.4.4]{Nee01}}. In the sequel, we establish an isomorphism of functors $h_{\tilde X}^\T \cong h$.
		
		\item \label{thm: BK-2.10-prep-isos} We have isomorphisms of functors 
		\begin{align}\label{diag: Hom(-,t)}
			h^\U_p \colon h^\U_{\tilde U} \cong h^\U_{\tilde X} \hspace{2.5mm} \textup{ and } \hspace{2.5mm}
			h^\V_q\colon h^\V_{\tilde V} \cong h^\V_{\tilde X}.
		\end{align}
		For $h^\U_p$ this is immediate from \eqref{diag: hom-cart} since $h^U$ is homological for all $U \in \U$ and $h^\U_\V=0$. For each $V \in \V$ and $f \in h^V_{V'}$ the naturality of $\theta$ and the bifunctoriality of $\Hom_\T(-,-)$ give rise to a commutative diagram
		\[
			\begin{tikzcd}[sep={17.5mm,between origins}]
				h^{V'}_{V'} \ar[r, "\theta_{V'}", "\cong"'] \ar[d, "h^f_{V'}"] & h^{V'}_{\tilde U} \ar[d, "h^f_{\tilde U}"] & h^{V'}_{V'} \ar[l, "h^{V'}_u"'] \ar[d, "h^f_{V'}"]& \id_{V'} \ar[r, <->] \ar[d, |->] & u \ar[d, |->] \ar[r, <-|] & \id_{V'} \ar[d, |->] \\
				h^V_{V'} \ar[r, "\theta_V", "\cong"'] & h^V_{\tilde U} & h^V_{V'}, \ar[l, "h^{V}_u"'] & f \ar[r, <->] & u \circ f \ar[r, <-|] & f,
			\end{tikzcd}
		\]
		by which $h^\V_u = \theta$ is an isomorphism. Applying the homological functor $h^V$ to \eqref{diag: hom-cart} yields a morphism 
		\[
			\begin{tikzcd}[sep={17.5mm,between origins}]
				\cdots \ar[r] & h^V_{V''} \ar[r] \ar[d, equal] & h^V_{V'} \ar[r, "h^V_v"] \ar[d, "h^V_u", "\cong"'] & h^V_{\tilde V} \ar[r] \ar[d, "h^V_q"] & \cdots \\
				\cdots \ar[r] & h^V_{V''} \ar[r] & h^V_{\tilde U} \ar[r, "h^V_p"] & h^V_{\tilde X} \ar[r] & \cdots
			\end{tikzcd}
		\]
		of long exact sequences of vector spaces. By the five lemma, $h^\V_q$ is an isomorphism as well.
		
		\item \label{thm: BK-2.10-res} Combining \ref{thm: BK-2.10-start} and \ref{thm: BK-2.10-prep-isos}, we obtain the following isomorphisms of functors:
		\begin{align} \label{diag: rho-U}
			\begin{tikzcd}[sep={17.5mm,between origins}, ampersand replacement=\&]
				h_{\tilde X}^\U \ar[rr, dashed, "\rho^\U", "\cong"'] \&\& h\vert_\U \& p \ar[rd, <->] \ar[rr, <->] \&\& e_{\U} \\
				 \& h_{\tilde U}^\U, \ar[ru, "\eta^\U"', "\cong"] \ar[lu, "h^\U_p", "\cong"'] \&\&\& \id_{\tilde U}, \ar[ru, <->]
			\end{tikzcd}
		\end{align}
		\begin{align} \label{diag: rho-V}
			\begin{tikzcd}[sep={17.5mm,between origins}, ampersand replacement=\&]
				h_{\tilde X}^\V \ar[rr, dashed, "\rho^\V", "\cong"'] \& \& h\vert_\V \& q \ar[rd, <->] \ar[rr, <->] \&\& e_{\V} \\
				\& h_{\tilde V}^\V, \ar[ru, "\eta^\V"', "\cong"] \ar[lu, "h^\V_q", "\cong"'] \&\&\& \id_{\tilde V}. \ar[ru, <->] 
			\end{tikzcd}
		\end{align}
		
		\item \label{thm: BK-2.10-prep} To prepare for the next step, we show that
		\begin{align} \label{equ: help-commute}
			h(u) \circ \rho^\U_{\tilde U} = \rho^\V_{V'} \circ h^u_{\tilde X}.
		\end{align}
		Given $f \in h^{\tilde U}_{\tilde X}$, we use \eqref{diag: Hom(-,t)} to obtain a (unique) $r \in h^{\tilde U}_{\tilde U}$ such that $f= h^{\tilde U}_p(r)$ and a (unique) $s \in h^{V'}_{\tilde V}$ such that $pru=h^u_{\tilde X}(f)=h^{V'}_q(s) = qs\in h^{V'}_{\tilde X}$. Since homotopy cartesian squares are weak pullbacks, we obtain a $w \in h^{V'}_{V'}$ such that 
		\begin{align} \label{diag: hom-PB}
			\begin{tikzcd}[sep={17.5mm,between origins}, ampersand replacement=\&]
				V' \ar[r, dashed, "w"] \ar[d, "u"] \ar[rr, bend left=10mm, "s"] \& V' \ar[r, "v"] \ar[d, "u"] \ar[rd, phantom, "\square"] \& \tilde V \ar[d, "q"] \\
				\tilde U \ar[r, "r"] \ar[rr,, bend right=10mm, "f"'] \& \tilde U \ar[r, "p"] \& \tilde X
			\end{tikzcd}
		\end{align}
		commutes. This yields
		\begin{align*}
			h(u) \left(\rho^\U_{\tilde U}(f)\right) &= h(u)\left(\rho^{\U}_{\tilde U}\left(h_p^{\tilde U}(r)\right)\right) \overset{\eqref{diag: rho-U}}= h(u) \left(\eta^\U_{\tilde U}(r)\right) \overset{\textup{\ref{ntn: trans}.\ref{ntn: trans-u}}}{\underset{\textup{\ref{thm: BK-2.10-start}}}=} h(u)\left(h(r)(e_\U)\right) = h(ru)(e_\U)\\
			& \overset{\eqref{diag: hom-PB}}= h(uw)(e_\U) = h(w)\left(h(u)(e_\U)\right) \overset{\eqref{diag: def-v}}= h(w)\left(h(v)(e_\V)\right) = h(vw)(e_\V) \\
			& \overset{\eqref{diag: hom-PB}}= h(s)(e_\V) \overset{\textup{\ref{ntn: trans}.\ref{ntn: trans-u}}}{\underset{\textup{\ref{thm: BK-2.10-start}}}=} \eta^\V_{V'}(s) \overset{\eqref{diag: rho-V}}= \rho^\V_{V'}\left(h_q^{V'}(s)\right) = \rho^\V_{V'}\left(h^u_{\tilde X}(f)\right).
		\end{align*}
		
		\item \label{thm: BK-2.10-nat-trans}To obtain a natural transformation $h_{\tilde X} \to h$ on all of $\T$, we apply the cohomological functors $h_{\tilde X}$ and $h$ to the distinguished triangle \eqref{diag: total-triangle}. By \ref{thm: BK-2.10-prep} and the naturality of $\rho^\V \colon h^\V_{\tilde X} \to h \vert_\V$ applied to $-v\colon V' \to \tilde V$, we have a commutative square
		\[
			\begin{tikzcd}[row sep={17.5mm,between origins}, column sep={35mm,between origins}, ampersand replacement=\&]
				h^{\tilde U}_{\tilde X} \oplus h^{\tilde V}_{\tilde X} \ar{r}{\begin{pmatrix}
						h^u_{\tilde X} & h^{-v}_{\tilde X}
				\end{pmatrix}} \ar[d, color=white, "\textcolor{black}{\cong}"] \& h^{V'}_{\tilde X} \ar[d, "\rho^\V_{V'}"', "\cong"] \\
				h(\tilde U) \oplus h(\tilde V) \ar{r}{\begin{pmatrix}
						h(u) & h(-v)
				\end{pmatrix}} \ar[<-]{u}{\begin{pmatrix}
						\rho^\U_{\tilde U} & 0 \\ 0 & \rho^\V_{\tilde V}
				\end{pmatrix}} \& h(V').
			\end{tikzcd}
		\]
		By \Cref{lem: middle-non-unique} for $\E=\Vect$ and the five lemma, we then obtain an isomorphism of long exact sequences of vector spaces
		\begin{align} \label{diag: total-long-ex}
			\begin{tikzcd}[row sep={17.5mm,between origins}, column sep={35mm,between origins}, ampersand replacement=\&]
				\cdots \ar[r] \& h^{\tilde X}_{\tilde X} \ar[d, dashed, "\cong", "\eta_{\tilde X}"'] \ar{r}{\begin{pmatrix}
						h^p_{\tilde X} \\ h^q_{\tilde X}
				\end{pmatrix}} \& h^{\tilde U}_{\tilde X} \oplus h^{\tilde V}_{\tilde X} \ar{r}{\begin{pmatrix}
						h^u_{\tilde X} & -h^v_{\tilde X}
				\end{pmatrix}} \ar[d, color=white, "\textcolor{black}{\cong}"] \& h^{V'}_{\tilde X} \ar[r] \ar[d, "\rho^\V_{V'}"', "\cong"] \& \cdots \\
				\cdots \ar[r] \& h(\tilde X) \ar[pos=0.4]{r}{\begin{pmatrix}
						h(p) \\ h(q)
				\end{pmatrix}} \& h(\tilde U) \oplus h(\tilde V) \ar{r}{\begin{pmatrix}
						h(u) & -h(v)
				\end{pmatrix}} \ar[<-]{u}{\begin{pmatrix}
						\rho^\U_{\tilde U} & 0 \\ 0 & \rho^\V_{\tilde V}
				\end{pmatrix}} \& h(V') \ar[r] \& \cdots,\\
				\& \id_{\tilde X} \ar[r, |->] \ar[d, <->] \& \begin{pmatrix}
					p \\ q
				\end{pmatrix} \ar[d, <->] \\
				\& e \ar[r, |->] \& \begin{pmatrix}
					h(p)(e) \\ h(q)(e)
				\end{pmatrix} = \begin{pmatrix}
					e_{\U} \\ e_{\V}
				\end{pmatrix},
			\end{tikzcd}
		\end{align}
		which involves a (non-unique) dashed morphism defining an element $e := \eta_{\tilde X}(\id_{\tilde X}) \in h(\tilde X)$. This extends to a natural transformation $\eta\colon h_{\tilde X} \to h$, see \Cref{ntn: trans}.\ref{ntn: trans-u}.
		
		\item \label{thm: BK-2.10-res-iso} The restrictions $\eta\vert_\U$ and $\eta\vert_\V$ are isomorphisms: Indeed, given $U \in \U$ and $f \in h^U_{\tilde X}$, we use \eqref{diag: Hom(-,t)} to obtain a (unique) $\tilde f \in h^U_{\tilde U}$ such that $f=h^U_p(\tilde f) = p \tilde f$. This yields 
		\[
		\eta_U(f) \overset{\textup{\ref{ntn: trans}.\ref{ntn: trans-u}}}= h(f)(e) = h(\tilde f)\left(h(p)(e)\right) \overset{\eqref{diag: total-long-ex}}= h(\tilde f)(e_{\U}) \overset{\textup{\ref{ntn: trans}.\ref{ntn: trans-u}}}= \eta^\U_U(\tilde f) \overset{\eqref{diag: rho-U}}= \rho^\U_U\left(h^U_p(\tilde f)\right) = \rho^\U_U(f).
		\]
		Hence, $\eta \vert_\U = \rho^\U$ is an isomorphism, see \ref{thm: BK-2.10-res}. The proof of $\eta \vert_\V = \rho^\V$ is analogous.
		
		\item It remains to be seen that $\eta$ is an isomorphism. To this end, place an arbitrary object $X \in \T$ into a distinguished triangle \begin{tikzcd}[cramped, sep=small] U \ar[r] & X \ar[r] & V \end{tikzcd} with $U \in \U$ and $V \in \V$, see \Cref{dfn: sod}.\ref{dfn: sod-2}. Applying the natural transformation $\eta$ yields a morphism of long exact sequences of vector spaces
		\[
			\begin{tikzcd}[sep={17.5mm,between origins}]
				\cdots \ar[r] & h_{\tilde X}(V) \ar[r] \ar[d, "\eta_V", "\cong"'] & h_{\tilde X}(X) \ar[r] \ar[d, "\eta_X"] & h_{\tilde X}(U) \ar[r] \ar[d, "\eta_U", "\cong"'] & \cdots \\
				\cdots \ar[r] & h(V) \ar[r] & h(X) \ar[r] & h(U) \ar[r] & \cdots,
			\end{tikzcd}
		\]
		with isomorphisms as indicated due to \ref{thm: BK-2.10-res-iso}. By the five lemma, $\eta_X$ is an isomorphism. \qedhere
	\end{enumerate}
\end{proof}

\begin{lem} \label{lem: middle-non-unique}
	Let $\E$ be an exact category with $\Proj(\E)=\E$, or, equivalently, $\Inj(\E)=\E$. Then any solid commutative diagram over $\E$ with exact rows as below can be completed by a (non-unique) dashed morphism:
	\[
		\begin{tikzcd}[sep={15mm,between origins}]
			A \ar[r] \ar[d] & B \ar[r] \ar[d] & C \ar[r] \ar[d, dashed] & D \ar[r] \ar[d] & E \ar[d] \\
			A' \ar[r] & B' \ar[r] & C' \ar[r] & D' \ar[r] & E'
		\end{tikzcd}
	\]
	\qed
\end{lem}

With an additional hypothesis, a modification of the proof of \Cref{thm: BK-2.10} yields a refined result:

\begin{thm} \label{thm: BK-2.10-alt-hyp}
	Let $\T$ be a triangulated category, linear over a field, with semiorthogonal decompositions $(\U, \V)$ and $(\V, \ro \V)$. A contravariant linear cohomological functor $h \colon \T \to \mathrm{Vect}$ is representable if $h\vert_\U$ and $h\vert_\V$ are so.
\end{thm}

\begin{proof} We only describe the changes to the proof of \Cref{thm: BK-2.10} to eliminate $\theta$ in step \ref{thm: BK-2.10-start}. We pick a distinguished triangle $\begin{tikzcd}[cramped, sep=small] V' \ar[r, "u"] & \tilde U \ar[r] & \ro V \end{tikzcd}$ with $V' \in \V$ and $\ro V \in \ro \V$. In step \ref{thm: BK-2.10-v}, this $u$ is used instead of $e_\theta$ to define $v$. In step \ref{thm: BK-2.10-hom-PO}, we extend the diagram \eqref{diag: hom-cart} to
	\begin{align}
		\begin{tikzcd}[sep={17.5mm,between origins}, ampersand replacement=\&]
			V'' \ar[r] \ar[d, equal] \& V' \ar[r, "v"] \ar[d, "u"] \ar[rd, phantom, "\square"] \& \tilde V \ar[d, "q"] \\
			V'' \ar[r] \& \tilde U \ar[r, "p"] \ar[d] \& \tilde X \ar[d, dashed] \\
			\& \ro V \ar[r, equal] \& \ro V
		\end{tikzcd}
	\end{align}
	with columns distinguished triangles, see {\cite[Lem.~1.4.4]{Nee01}}. In step \ref{thm: BK-2.10-prep-isos}, the argument for $h^\U_p$ then also applies to $h^\V_q$. From step \ref{thm: BK-2.10-res} onward, the proof remains unchanged.
\end{proof}

The previous construction is now applied to dualized hom-functors, see {\cite[Prop.~3.8]{BK89}}:

\begin{cor} \label{prp: BK-3.8}
	Let $\T$ be a triangulated category, linear over a field, with semiorthogonal decompositions $(\lo \U, \U)$, $(\U, \V)$, and $(\V, \ro \V)$. Given an object $X \in \T$, consider distinguished triangles $\lo U \to X \xrightarrow{t_U} U$ and $U' \to X \xrightarrow{t_V} V$ with $\lo U \in \lo \U$, $U, U' \in \U$, and $V \in \V$. Then the functor $\left(h^X_\T\right)^\ast$ is representable if the functors $\left(h^U_\U\right)^\ast$ and $\left(h^V_\V\right)^\ast$ are so.
\end{cor}

\begin{proof}
	We apply \Cref{thm: BK-2.10-alt-hyp} to $h := \left(h^X_\T\right)^\ast$. Note that $h^{t_U}_{\U} \colon h^U_\U \cong h^X_\U$ and $h^{t_V}_{\V}\colon h^V_\V \cong h^X_\V$ are isomorphisms of functors. By hypothesis, there are $\tilde U \in \U$, $\tilde V \in \V$, and isomorphisms of functors $\eta_U\colon h^\U_{\tilde U} \cong \left(h^U_\U\right)^\ast$ and $\eta_V\colon h^\V_{\tilde V} \cong \left(h^V_\V\right)^\ast$. By composition, we obtain isomorphisms of functors
	
	\begin{align} \label{diag: eta-U-V}
		\begin{tikzcd}[row sep={17.5mm,between origins}, column sep={25mm,between origins}, ampersand replacement=\&]
			\& \left(h^X_\U\right)^\ast \mathrlap{\; = h \vert_\U} \ar[d, "\left(h^{t_U}_{\U}\right)^\ast", "\cong"'] \\
			h^\U_{\tilde U} \ar[r, "\eta_U"', "\cong"] \ar[ru, dashed, "\eta^\U", "\cong"'] \& \left( h^U_\U \right)^\ast, \\
			\& e_\U \ar[d, <->] \\
			\id_{\tilde U} \ar[r, <->] \ar[ru, dashed, <->] \& e_U \mathrlap{\; = e_\U\left(- \circ t_U\right),}
		\end{tikzcd}
		\hspace{4cm}
		\begin{tikzcd}[row sep={17.5mm,between origins}, column sep={25mm,between origins}, ampersand replacement=\&]
			\mathllap{h \vert_\V = \; } \left( h^X_\V \right)^\ast \ar[d, "\left(h^{t_V}_\V\right)^\ast"', "\cong"] \& \\
			\left(h^V_\V\right)^\ast \& h^\V_{\tilde V}, \ar[l, "\eta_V", "\cong"'] \ar[lu, dashed, "\eta^\V"', "\cong"] \\
			e_\V \ar[d, <->] \\
			\mathllap{e_\V\left(- \circ t_V\right) = \;} e_V \ar[r, <->] \& \id_{\tilde V}. \ar[lu, dashed, <->]
		\end{tikzcd}
	\end{align}
	This shows that $h \vert_\U$ and $h\vert_\V$ are representable and the claim follows from \Cref{thm: BK-2.10-alt-hyp}.
	\qedhere
\end{proof}

We apply the preceeding construction to prove \Cref{thmA: rep} from the introduction:

\begin{thm} \label{thm: mMor-rep}
	Let $\F$ be a Frobenius category, linear over field, $l \in \NN$, and $(X, \alpha) \in \mMor_{l}(\F)$. Consider two diagrams
	\begin{align} \label{diag: start-pushout}
		\begin{tikzcd}[sep={20mm,between origins}, ampersand replacement=\&]
			X^{0,0} \ar[r, tail, "\alpha^{0,0}"] \ar[d, two heads] \ar[rd, phantom, "\square"] \& X^{0,1} \ar[r, tail, "\alpha^{0,1}"] \ar[d, two heads, "\beta^{0,1}"] \ar[rd, phantom, "\square"] \& X^{0,2} \ar[r, tail, "\alpha^{0,2}"] \ar[d, two heads, "\beta^{0,2}"] \& \cdots \ar[r, tail, "\alpha^{0,l-3}"] \ar[rd, phantom, "\square"] \& X^{0,l-2} \ar[r, tail, "\alpha^{0,l-2}"] \ar[d, two heads, "\beta^{0,l-2}"] \ar[rd, phantom, "\square"] \& X^{0,l-1} \ar[r, tail, "\alpha^{0,l-1}"] \ar[d, two heads, "\beta^{0,l-1}"] \ar[rd, phantom, "\square"] \& X^{0,l} \ar[d, two heads, "\beta^{0,l}"] \\
			0 \ar[r, tail] \& X^{1,1} \ar[r, tail, "\alpha^{1,1}"] \ar[d, two heads] \ar[rd, phantom, "\square"] \& X^{1,2} \ar[r, tail, "\alpha^{1,2}"] \ar[d, two heads, "\beta^{1,2}"] \ar[ru, phantom, "\square"] \& \cdots \ar[r, tail, "\alpha^{1,l-3}"] \ar[rd, phantom, "\square"] \& X^{1,l-2} \ar[r, tail, "\alpha^{1,l-2}"] \ar[d, two heads, "\beta^{1,l-2}"] \ar[rd, phantom, "\square"] \& X^{1,l-1} \ar[r, tail, "\alpha^{1,l-1}"] \ar[d, two heads, "\beta^{1,l-1}"] \ar[rd, phantom, "\square"] \& X^{1,l} \ar[d, two heads, "\beta^{1,l}"] \\
			\& 0 \ar[r, tail] \& X^{2,2} \ar[r, tail, "\alpha^{2,2}"] \ar[d, two heads] \ar[ru, phantom, "\square"] \& \cdots \ar[r, tail, "\alpha^{2,l-3}"] \& X^{2,l-2} \ar[r, tail, "\alpha^{2,l-2}"] \ar[d, two heads, "\beta^{2,l-2}"] \& X^{2,l-1} \ar[r, tail, "\alpha^{2,l-1}"] \ar[d, two heads, "\beta^{2,l-1}"] \& X^{2,l} \ar[d, two heads, "\beta^{2,l}"] \\
			\&\& 0 \ar[r, tail] \ar[ru, phantom, "\square"] \& \cdots \ar[ru, phantom, "\square"] \& \ar[d, two heads] \ar[rd, phantom, "\square"] \ar[ru, phantom, "\square"] \myvdots \& \myvdots \ar[d, two heads, "\beta^{l-2, l-1}" near start] \ar[rd, phantom, "\square"] \ar[ru, phantom, "\square"] \& \myvdots \ar[d, two heads, "\beta^{l-2, l}"] \\
			\&\&\&\& 0 \ar[r, tail] \& X^{l-1,l-1} \ar[r, tail, "\alpha^{l-1,l-1}"] \ar[d, two heads] \ar[rd, phantom, "\square"] \& X^{l-1,l} \ar[d, two heads, "\beta^{l-1,l}"] \\
			\&\&\&\&\& 0 \ar[r, tail] \& X^{l,l},
		\end{tikzcd}
	\end{align}
	\begin{align} \label{diag: Serre-pushout}
		\begin{tikzcd}[sep={20mm,between origins}, ampersand replacement=\&]
			\tilde X^{0,0} \ar[r, "\tilde \beta^{0,0}"] \ar[d, tail] \ar[rd, phantom, "\square"] \& \tilde X^{0,1} \ar[r, "\tilde \beta^{0,1}"] \ar[d, tail, "\tilde \alpha^{0,1}"] \ar[rd, phantom, "\square"] \& \tilde X^{0,2} \ar[r, "\tilde \beta^{0,2}"] \ar[d, tail, "\tilde \alpha^{0,2}"] \& \cdots \ar[r, "\tilde \beta^{0,l-3}"] \ar[rd, phantom, "\square"] \& \tilde X^{0,l-2} \ar[r, "\tilde \beta^{0,l-2}"] \ar[d, tail, "\tilde \alpha^{0,l-2}"] \ar[rd, phantom, "\square"] \& \tilde X^{0,l-1} \ar[r, "\tilde \beta^{0,l-1}"] \ar[d, tail, "\tilde \alpha^{0,l-1}"] \ar[rd, phantom, "\square"] \& \tilde X^{0,l} \ar[d, tail, "\tilde \alpha^{0,l}"] \\
			I^0 \ar[r] \& \tilde X^{1,1} \ar[r, "\tilde \beta^{1,1}"] \ar[d, tail] \ar[rd, phantom, "\square"] \& \tilde X^{1,2} \ar[r, "\tilde \beta^{1,2}"] \ar[d, tail, "\tilde \alpha^{1,2}"] \ar[ru, phantom, "\square"] \& \cdots \ar[r, "\tilde \beta^{1,l-3}"] \ar[rd, phantom, "\square"] \& \tilde X^{1,l-2} \ar[r, "\tilde \beta^{1,l-2}"] \ar[d, tail, "\tilde \alpha^{1,l-2}"] \ar[rd, phantom, "\square"] \& \tilde X^{1,l-1} \ar[r, "\tilde \beta^{1,l-1}"] \ar[d, tail, "\tilde \alpha^{1,l-1}"] \ar[rd, phantom, "\square"] \& \tilde X^{1,l} \ar[d, tail, "\tilde \alpha^{1,l}"] \\
			\& I^1 \ar[r] \& \tilde X^{2,2} \ar[r, "\tilde \beta^{2,2}"] \ar[d, tail] \ar[ru, phantom, "\square"] \& \cdots \ar[r, "\tilde \beta^{2,l-3}"] \& \tilde X^{2,l-2} \ar[r, "\tilde \beta^{2,l-2}"] \ar[d, tail, "\tilde \alpha^{2,l-2}"] \& \tilde X^{2,l-1} \ar[r, "\tilde \beta^{2,l-1}"] \ar[d, tail, "\tilde \alpha^{2,l-1}"] \& \tilde X^{2,l} \ar[d, tail, "\tilde \alpha^{2,l}"] \\
			\&\& I^2 \ar[r] \ar[ru, phantom, "\square"] \& \cdots \& \myvdots \ar[d, tail] \ar[rd, phantom, "\square"] \ar[ru, phantom, "\square"] \& \myvdots \ar[d, tail, "\tilde \alpha^{l-2, l-1}" near start] \ar[rd, phantom, "\square"] \ar[ru, phantom, "\square"] \& \myvdots \ar[d, tail, "\tilde \alpha^{l-2, l}"] \\
			\&\&\&\& I^{l-2} \ar[r] \& \tilde X^{l-1,l-1} \ar[r, "\tilde \beta^{l-1,l-1}"] \ar[d, tail] \ar[rd, phantom, "\square"] \& \tilde X^{l-1,l} \ar[d, tail, "\tilde \alpha^{l-1,l}"] \\
			\&\&\&\&\& I^{l-1} \ar[r] \& \tilde X^{l,l}
		\end{tikzcd}
	\end{align}
	of bicartesian squares in $\F$, where $(X^{0, \bullet}, \alpha^{0, \bullet}) = (X, \alpha)$ and $I^k \in \Inj(\F)$ for $k \in \{0, \dots, l-1\}$. Set $(\tilde X, \tilde \alpha) := (\tilde X^{\bullet, l}, \tilde \alpha^{\bullet, l})$. Suppose that there are isomorphisms of functors \[\eta^{i,j} := \eta_{X^{i,j}}\colon h^{\M_0}_{\tilde X^{i,j}} \cong \left(h^{X^{i,j}}_{\M_0}\right)^\ast,\]
	where $0 \leq i \leq j \leq l$, such that $e_{X^{i,j}}\left(- \circ \ol{\beta^{i,j}}\right) = e_{X^{i+1,j}}\left(\ol{\tilde \alpha^{i,j}} \circ -\right)$ on $\Hom_{\M_0}(X^{i+1, j}, \tilde X^{i,j})$ for all $0\leq i < j \leq l$, see \Cref{ntn: trans}.\ref{ntn: trans-u_X}. Then there exists an isomorphism of functors
	\[\eta_X\colon h^{\M_l}_{\tilde X} \cong \left(h^X_{\M_l}\right)^\ast\]
	such that $e_X\left(\ol{(0, \dots, 0, \psi \, \beta^{l-1,l} \cdots \beta^{0, l})}\right) = e_{X^{l,l}}\left(\ol \psi\right)$ for all $\psi \in \Hom_\F(X^{l,l}, \tilde X^l)$.
\end{thm}

\begin{proof}
	The claim is trivial for $l=0$. For $l > 0$, we proceed as in \Cref{thm: BK-2.10-alt-hyp} and \Cref{prp: BK-3.8} with the notation used there. To this end, we set
	\begin{align} \label{eqn: input}
		h := \left(h^X_{\M_l}\right)^\ast, \; \lo \U := \Gamma^{l-1}, \; \U := \Delta^{[l-1, l]}, \; \V := \Gamma^l, \; \textup{and} \; \ro \V := \Gamma^{[0, l-1]},
	\end{align}
	see \Cref{prp: SOD-Gamma-Gamma,prp: SODs}. Contracting the $(l-1)$st column and the $l$th row of both \eqref{diag: start-pushout} and \eqref{diag: Serre-pushout} establishes the given setup for $\gamma^{l-1}(X) \in \M_{l-1}$ and $\widetilde{\gamma^{l-1}(X)}=\gamma^l(\tilde X)$, see \Cref{lem: rep-comp}. By induction, we obtain an isomorphism of functors $\eta_{\gamma^{l-1}(X)}\colon h^{\M_{l-1}}_{\gamma^{l}(\tilde X)} \xrightarrow{\cong} \left(h^{\gamma^{l-1}(X)}_{\M_{l-1}}\right)^\ast$ such that
	\begin{align} \label{diag: eta-induction}
		e_{\gamma^{l-1}(X)}\left(\ol{(0, \dots, 0, \psi \, \beta^{l-2,l} \cdots \beta^{0, l})}\right) = e_{X^{l-1,l}}\left(\ol \psi\right)
	\end{align}
	for all $\psi \in \Hom_\F(X^{l-1,l}, \tilde X^{l-1})$. To process \Cref{prp: BK-3.8}, consider the distinguished triangles
	
	\begin{align} \label{diag: 1st-triangle}
		\begin{tikzcd}[sep={17.5mm,between origins}, ampersand replacement=\&]
			\lo U \ar[d] \& 0 \ar[r, equal] \ar[d] \& \cdots \ar[r, equal] \& 0 \ar[r, tail] \ar[d] \& Y^{l-1} \ar[r, tail] \ar[rd, phantom, "\square"] \ar[d, two heads] \& P(X^l) \ar[d, two heads, "p_{X^l}"] \\
			X \ar[d, "t_U"] \& X^0 \ar[r, tail, "\alpha^{0}"] \ar[d, equal] \& \cdots \ar[r, tail, "\alpha^{l-3}"] \& X^{l-2} \ar[r, tail, "\alpha^{l-2}"] \ar[d, equal] \& X^{l-1} \ar[r, tail, "\alpha^{l-1}"] \ar[d, tail, "\alpha^{l-1}"] \& X^l \ar[d, equal, "\phantom{\beta^{l-1,l} \cdots \beta^{0,l}}"] \\
			U \& X^0 \ar[r, tail, "\alpha^{0}"] \& \cdots \ar[r, tail, "\alpha^{l-3}"] \& X^{l-2} \ar[r, tail, "\alpha^{l-1}\alpha^{l-2}"] \& X^l \ar[r, equal] \& X^l,
		\end{tikzcd}
	\end{align}
	\begin{align} \label{diag: 2nd-triangle}
		\begin{tikzcd}[sep={17.5mm,between origins}, ampersand replacement=\&]
			U' \ar[d] \& X^0 \ar[r, tail, "\alpha^0"] \ar[d, equal] \& \cdots \ar[r, tail, "\alpha^{l-3}"] \& X^{l-2} \ar[r, tail, "\alpha^{l-2}"] \ar[d, equal] \& X^{l-1} \ar[r, equal] \ar[d, equal] \& X^{l-1} \ar[d, tail, "\alpha^{l-1}"] \\
			X \ar[d, "t_V"] \& X^0 \ar[r, tail, "\alpha^{0}"] \ar[d] \& \cdots \ar[r, tail, "\alpha^{l-3}"] \& X^{l-2} \ar[r, tail, "\alpha^{l-2}"] \ar[d] \& X^{l-1} \ar[r, tail, "\alpha^{l-1}"] \ar[d, two heads] \ar[rd, phantom, "\square"] \& X^l \ar[d, two heads, "\beta^{l-1,l} \cdots \beta^{0,l}"] \\
			V \& 0 \ar[r, equal] \& \cdots \ar[r, equal] \& 0 \ar[r, equal] \& 0 \ar[r, tail] \& X^{l,l}
		\end{tikzcd}
	\end{align}
	in $\M_l$, see Propositions \ref{prp: SODs}.\ref{prp: SODs-Gamma-Delta}.\ref{prp: SODs-Gamma-Delta-diag} and \ref{prp: SODs}.\ref{prp: SODs-Delta-Gamma}.\ref{prp: SODs-Delta-Gamma-diag}.\\
	To represent $\left(h^U_\U\right)^\ast$ and hence $h \vert_\U$, we use the equivalence $\U \xrightarrow{\simeq} \M_{l-1}$, given by the restriction of $\stab \gamma^{l-1}$, see \Cref{lem: Delta-mMor}, and \eqref{diag: eta-U-V}: Set
	\[
		\begin{tikzcd}[row sep={17.5mm,between origins}]
			\tilde U \colon \; \tilde X^0 \ar[r, tail, "\tilde \alpha^0"] & \tilde X^1 \ar[r, tail, "\tilde \alpha^1"] & \dots \ar[r, tail, "\tilde \alpha^{l-2}"] & \tilde X^{l-1} \ar[r, equal] & \tilde X^{l-1}.
		\end{tikzcd}
	\]
	Since $\gamma^{l-1}(U) = \gamma^{l-1}(X)$ and $\gamma^{l-1}(\tilde U)=\gamma^l(\tilde X)$, we obtain a commutative diagram
	\begin{equation} \label{diag: left-wing}
		\begin{tikzcd}[row sep={17.5mm,between origins}, column sep=large, ampersand replacement=\&]
			\& h\vert_\U \ar[d, "\left(h^{t_U}_\U\right)^\ast", "\cong"']\\
			h^\U_{\tilde U} \ar[r, dashed, "\cong", "\eta_U"'] \ar[d, "\stab\gamma^{l-1}"', "\cong"] \ar[ru, dashed, "\eta^\U", "\cong"'] \& \left(h^U_\U\right)^\ast \\
			h^{\stab \gamma^{l-1}(\U)}_{\gamma^{l-1}(\tilde U)} \ar[d, equal] \ar[r, dashed, "\cong"'] \& \left(h^{\gamma^{l-1}(X)}_{\stab \gamma^{l-1}(\U)}\right)^\ast \ar[u, "\left(\stab\gamma^{l-1}\right)^\ast"', "\cong"] \ar[d, equal]\\
			 h^{\M_{l-1}}_{\gamma^l(\tilde X)} \circ \stab \gamma^{l-1} \ar[r, "\eta_{\gamma^{l-1}(X)} \, \circ \, \stab \gamma^{l-1}", "\cong"'] \& \left(h^{\gamma^{l-1}(X)}_{\M_{l-1}} \right)^\ast \circ \stab \gamma^{l-1},
		\end{tikzcd}\\
		\begin{tikzcd}[row sep={17.5mm,between origins}, column sep=large, ampersand replacement=\&]
			\& e_\U \ar[ld,<->] \ar[d, <->] \\
			\id_{\tilde U} \ar[dd, <->] \ar[r, <->] \& e_\U(- \circ t_U) = e_{\gamma^{l-1}(X)} \circ \stab \gamma^{l-1} \ar[dd, <->] \\\\
			\id_{\gamma^l(\tilde X)} \ar[r, <->] \& e_{\gamma^{l-1}(X)},
		\end{tikzcd}
	\end{equation}
	of isomorphisms of functors. To represent $\left(h^V_\V\right)^\ast$ and hence $h \vert_\V$, we use the equivalence $\V \xrightarrow{\simeq} \M_0$, given by the restriction of $\stab \gamma^{l^c}$, see \Cref{con: delta-compl}, and \eqref{diag: eta-U-V}: Set
	\[
		\begin{tikzcd}[row sep={17.5mm,between origins}]
			\tilde V \colon \; 0 \ar[r, equal] & 0 \ar[r, equal] & \dots \ar[r, equal] & 0 \ar[r, tail] & \tilde X^{l}.
		\end{tikzcd}
	\]
	Since $\gamma^{l^c}(V)=X^{l,l}$ and $\gamma^{l^c}(\tilde V)=\tilde X^l$, we obtain a commutative diagram
	\begin{align} \label{diag: right-wing}
		\begin{tikzcd}[row sep={17.5mm,between origins}, column sep=large, ampersand replacement=\&]
			h\vert_\V \ar[d, "\left(h^{t_V}_\U\right)^\ast"', "\cong"] \\
			\left(h^V_\V\right)^\ast \& h^\V_{\tilde V} \ar[lu, dashed, "\eta^\V"', "\cong"] \ar[l, dashed, "\cong"', "\eta_V"] \ar[d, "\stab \gamma^{l^c}", "\cong"'] \\
			\left(h^{\gamma^{l^c}(V)}_{\stab \gamma^{l^c}(\V)} \right)^\ast \ar[d, equal] \ar[u, "\left(\stab \gamma^{l^c}\right)^\ast", "\cong"'] \&  h^{\stab\gamma^{l^c}(\V)}_{\gamma^{l^c}(\tilde V)} \ar[d, equal] \ar[l, dashed, "\cong"] \\
			\left(h^{X^{l,l}}_{\M_0}\right)^\ast \circ \stab \gamma^{l^c} \& h^{\M_0}_{\tilde X^l} \circ \stab \gamma^{l^c}, \ar[l, "\eta^{l,l} \, \circ \, \stab \gamma^{l^c}"', "\cong"]
		\end{tikzcd}
		\begin{tikzcd}[row sep={17.5mm,between origins}, column sep=large, ampersand replacement=\&]
			e_\V \ar[rd, <->] \ar[d, <->] \\
			e_{X^{l,l}} \circ \stab \gamma^{l^c} = e_\V(- \circ t_V) \ar[r, <->] \ar[dd, <->] \& \id_{\tilde V} \ar[dd, <->] \\\\
			e_{X^{l,l}} \ar[r, <->] \& \id_{\tilde X^l},
		\end{tikzcd}
	\end{align}
	of isomorphisms of functors. As in the proof of \Cref{thm: BK-2.10-alt-hyp}, we consider the distinguished triangle
	\begin{align} \label{diag: 3rd-triangle}
		\begin{tikzcd}[sep={17.5mm,between origins}, ampersand replacement=\&]
			V' \ar[d, "u"] \& 0 \ar[r, equal] \ar[d] \& \cdots \ar[r, equal] \& 0 \ar[r, equal] \ar[d] \& 0 \ar[r, tail] \ar[d] \& \tilde X^{l-1} \ar[d, equal] \\
			\tilde U \ar[d] \& \tilde X^0 \ar[r, tail, "\tilde \alpha^{0}"] \ar[d, equal] \& \cdots \ar[r, tail, "\tilde \alpha^{l-3}"] \& \tilde X^{l-2} \ar[r, tail, "\tilde \alpha^{l-2}"] \ar[d, equal] \& \tilde X^{l-1} \ar[r, equal] \ar[d, equal] \& \tilde X^{l-1} \ar[d, tail, "i_{\tilde X^{l-1}}"] \\
			\ro V \& \tilde X^0 \ar[r, tail, "\tilde \alpha^{0}"] \& \cdots \ar[r, tail, "\tilde \alpha^{l-3}"] \& \tilde X^{l-2} \ar[r, tail, "\tilde \alpha^{l-2}"] \& \tilde X^{l-1} \ar[r, tail] \& I(\tilde X^{l-1}),
		\end{tikzcd}
	\end{align}
	in $\M_l$, see \Cref{prp: SOD-Gamma-Gamma}.\ref{prp: SOD-Gamma-Gamma-diag}.\\
	Finally, we show that $\tilde X$ results from the construction in \Cref{thm: BK-2.10}: To this end, we compute $v \in \Hom_{\M_l}(V', \tilde V)$ by combining \eqref{diag: left-wing} and \eqref{diag: right-wing}, see \eqref{diag: def-v}:
	\begin{equation*}
		\begin{tikzcd}[row sep={17.5mm,between origins}, column sep={42.5mm,between origins}, ampersand replacement = \&]
					\& h(\tilde U) \ar[r, "h(u)"] \ar[d, "\left(h^{t_U}_{\tilde U}\right)^\ast", "\cong"'] \& h(V') \ar[d, "\left(h^{t_V}_{V'}\right)^\ast"', "\cong"] \& \\
					h^{\tilde U}_{\tilde U} \ar[r, "\cong"] \ar[ru, "\eta^\U_{\tilde U}", "\cong"'] \ar[d, "\stab \gamma^{l-1}"', "\cong"] \& \left(h^U_{\tilde U}\right)^\ast \ar[r, dashed] \& \left(h^V_{V'}\right)^\ast \& h^{V'}_{\tilde V} \ar[l, "\cong"'] \ar[lu, "\cong", "\eta^\V_{V'}"'] \ar[d, "\stab \gamma^{l^c}", "\cong"'] \\
					h^{\gamma^l(\tilde X)}_{\gamma^l(\tilde X)} \ar[r, "\left(\eta_{\gamma^{l-1}(X)} \, \circ \, \stab \gamma^{l-1}\right)_{\gamma^l(\tilde X)}", "\cong"'] \& \left(h^{\gamma^{l-1}(X)}_{\gamma^l(\tilde X)}\right)^\ast \ar[u, "\left(\stab \gamma^{l-1}\right)^\ast"', "\cong"] \ar[r, dashed] \& \left(h^{X^{l,l}}_{\tilde X^{l-1}}\right)^\ast \ar[u, "\left(\stab \gamma^{l^c}\right)^\ast", "\cong"'] \& h^{\tilde X^{l-1}}_{\tilde X^l} \ar[l, "\left(\eta^{l,l} \, \circ \, \stab \gamma^{l^c}\right)_{\tilde X^{l-1}}"', "\cong"]
				\end{tikzcd}
	\end{equation*}
	\begin{equation*}
				\begin{tikzcd}[row sep={17.5mm,between origins}, column sep={42.5mm,between origins}, ampersand replacement = \&]
					\& e_{\gamma^{l-1}(X)} \circ \stab \gamma^{l-1} \ar[r, |->] \ar[d, "(1)", <->] \& \left(e_{\gamma^{l-1}(X)} \circ \stab \gamma^{l-1}\right)(u \circ -) \ar[d, <->] \& \\
					\id_{\tilde U} \ar[r, dashed, <->] \ar[ru, <->, dashed] \ar[d, <->] \& e_{\gamma^{l-1}(X)} \circ \stab \gamma^{l-1} \ar[r, dashed, |->] \& \left(e_{\gamma^{l-1}(X)} \circ \stab \gamma^{l-1}\right)(u \circ - \circ t_V) \& \stab \delta^{l^c} \ol{\tilde \alpha^{l-1}} = v \ar[l, <->, dashed] \ar[lu, <->, dashed] \ar[d, <->] \\
					\id_{\gamma^l(\tilde X)} \ar[r, <->] \& e_{\gamma^{l-1}(X)} \ar[r, |->, dashed] \ar[u, <->] \& e_{X^{l,l}}\left(\ol{\tilde \alpha^{l-1}} \circ -\right) \ar[u, "(2)", <->] \& \ol{\tilde \alpha^{l-1}} \ar[l, <->]
				\end{tikzcd}
		\end{equation*}
		
	The correspondence (1) holds since $\stab \gamma^{l-1}(t_U) = \id_{\gamma^{l-1}(X)}$, see \eqref{diag: 1st-triangle}. To see (2), note that a general element of $h^V_{V'}$ is of the form $\ol{(0, \dots, 0, \psi)}$, where $\psi \in \Hom_\F(X^{l,l}, \tilde X^{l-1})$. Using the hypothesis $e_{X^{l-1,l}}\left(- \circ \ol{\beta^{l-1,l}}\right) = e_{X^{l,l}}\left(\ol{\tilde \alpha^{l-1,l}} \circ -\right)$, we compute
	\begin{align*}
		e_{X^{l,l}} \left( \ol{\tilde \alpha^{l-1}} \circ \stab \gamma^{l^c} \left( \ol{(0, \dots, 0, \psi)} \right) \right) &= e_{X^{l,l}} \left( \ol{\tilde \alpha^{l-1}} \ol{\psi}\right) = e_{X^{l-1, l}}\left( \ol{\psi} \ol{\beta^{l-1,l}} \right)\\
		&\overset{\textup{\eqref{diag: eta-induction}}}= e_{\gamma^{l-1}(X)} \left(\ol{(0, \dots, 0, \psi \beta^{l-1,l} \beta^{l-2,l} \cdots \beta^{0,l})}\right) \\
		&\underset{\textup{\eqref{diag: 3rd-triangle}}}{\overset{\textup{\eqref{diag: 2nd-triangle}}}=} \left(e_{\gamma^{l-1}(X)}\circ \stab \gamma^{l-1} \right) \left(u \circ \ol{(0, \dots, 0, \psi)} \circ t_V \right).
	\end{align*}
	All other correspondences are obvious or due to the diagram's commutativity.\\
	We put $u$, $v$, and $\tilde X$ in a commutative square
	\[
		\begin{tikzcd}[sep={12.5mm,between origins}]
			& \tilde V \ar[dd, "q"] && 0 \ar[rr, equal] \ar[dd] && \cdots \ar[rr, equal] && 0 \ar[rr, tail] \ar[dd] && \tilde X^l \ar[dd, equal] \\
			V' \ar[dd, "u"] \ar[ru, tail, "v"] && 0 \ar[rr, equal, crossing over] \ar[ru, equal] && \cdots \ar[rr, equal] && 0 \ar[rr, tail, crossing over] \ar[ru, equal] && \tilde X^{l-1} \ar[ru, tail, "\tilde \alpha^{l-1}"] \\
			& \tilde X && \tilde X^0 \ar[rr, tail, "\tilde \alpha^0"] && \cdots \ar[rr, tail, "\tilde \alpha^{l-2}" near start] && \tilde X^{l-1} \ar[rr, tail, "\tilde \alpha^{l-1}" near start] && \tilde X^l \\
			\tilde U \ar[ru, tail, "p"'] && \tilde X^0 \ar[rr, tail, "\tilde \alpha^0"] \ar[ru, equal] \ar[uu, <-, crossing over] && \cdots \ar[rr, tail, "\tilde \alpha^{l-2}"] && \tilde X^{l-1} \ar[rr, equal] \ar[ru, equal] \ar[uu, <-, crossing over] && \tilde X^{l-1} \ar[ru, tail, "\tilde \alpha^{l-1}"'] \ar[uu, equal, crossing over]
		\end{tikzcd}
	\]
	in $\mMor_l(\F)$. It yields a (termwise) short exact sequence \begin{tikzcd}[cramped, ampersand replacement=\&]
		V' \ar[>->]{r}{\begin{pmatrix} u \\ -v \end{pmatrix}} \& \tilde U \oplus \tilde V \ar[->>]{r}{\begin{pmatrix} p & q \end{pmatrix}} \& \tilde X
	\end{tikzcd} in $\mMor_l(\F)$, where $p = (\id_{\tilde X^0}, \dots, \id_{\tilde X^{l-1}}, \tilde \alpha^{l-1})$ and $q = (0, \dots, 0, \, \id_{\tilde X^l})$. The corresponding distinguished triangle, see \Cref{lem: ses-triangle}, matches \eqref{diag: total-triangle}. Therefore, \Cref{thm: BK-2.10} yields an isomorphism of functors $\eta_X := \eta \colon h_{\tilde X}^{\M_l} \cong h$ defined by $e_X := e \in h(\tilde X)$, see step \ref{thm: BK-2.10-nat-trans} of the proof. \\
	It remains to verify that $e_X\left(\ol{(0, \dots, 0, \psi \, \beta^{l-1,l} \cdots \beta^{0, l})}\right) = e_{X^{l,l}}\left(\ol \psi\right)$ for all $\psi \in \Hom_\F(X^{l,l}, \tilde X^l)$:
	\begin{align*}
		e_X\left(\ol{(0, \dots, 0, \psi \, \beta^{l-1,l} \cdots \beta^{0, l})}\right) & = e\left(q \circ \ol{(0, \dots, 0, \psi \, \beta^{l-1,l} \cdots \beta^{0, l})}\right)\\
		&\overset{\eqref{eqn: input}}= h(q)(e)\left(\ol{(0, \dots, 0, \psi \, \beta^{l-1,l} \cdots \beta^{0, l})}\right) \\
		&\overset{\eqref{diag: total-long-ex}}= e_{\V}\left(\ol{(0, \dots, 0, \psi \, \beta^{l-1,l} \cdots \beta^{0, l})}\right) \overset{\eqref{diag: 2nd-triangle}}= e_{\V}\left(\ol{(0, \dots, 0, \psi)} \circ t_V \right)\\
		& \overset{\eqref{diag: right-wing}}= \left( e_{X^{l,l}} \circ \stab \gamma^{l^c} \right)\left(\ol{(0, \dots, 0, \psi)}\right) = e_{X^{l,l}}\left(\ol \psi\right) \qedhere
	\end{align*}
\end{proof}

It remains to construct a diagram of the form \eqref{diag: Serre-pushout}, suitable for \Cref{thm: mMor-rep}:

\begin{prp} \label{lem: mMor-hyp}
	Let $\F$ be a Frobenius category, linear over a field, such that $\stab \F$ is hom-finite. Given any diagram in $\F$ of the form \eqref{diag: start-pushout}, suppose that any object $X \in \F$ in the diagram admits an isomorphism $\eta_X \colon h_{\tilde X} \cong \left(h^X\right)^\ast$ of contravariant functors $\stab \F \to \Vect$, where $\tilde X \in \F$. Then any diagram in $\F$ of the form \eqref{diag: Serre-pushout} with $e_{X^{0, j}}\left(- \circ \ol{\alpha^{0,j}}\right) = e_{X^{0,j+1}}\left(\ol{\tilde \beta^{0,j}} \circ -\right)$ on $\Hom_{\stab \F}(X^{0, j+1}, \tilde X^{0,j})$ for all $j \in \{0, \dots, l-1\}$ meets the requirements stated in \Cref{thm: mMor-rep}. 
\end{prp}

\begin{proof}
	Using the notation from \eqref{diag: start-pushout}, for each $i,j \in \{0,\dots,l\}$ with $i \leq j$, there is an isomorphism of functors $\hat \eta_{X^{i,j}} \colon h_{\hat X^{i,j}} \cong \left(h^{X^{i,j}}\right)^\ast$ defining $\hat e_{X^{i,j}} := e_{\hat \eta_{X^{i,j}}} \in \left(h^{X^{i,j}}_{\hat X^{i,j}}\right)^\ast$ , where $\hat X^{i,j} \in \F$, and morphisms $\hat \alpha^{i,j} \colon \hat X^{i,j} \to \hat X^{i+1,j}$ and $\hat \beta^{i,j} \colon \hat X^{i,j} \to \hat X^{i,j+1}$ in $\F$ such that
	\begin{itemize}
		\item $\hat e_{X^{i,j}}\left(- \circ \ol{\alpha^{i,j}}\right) = \hat e_{X^{i,j+1}}\left(\ol{\hat \beta^{i,j}} \circ -\right)$ on $\Hom_{\stab \F}(X^{i, j+1}, \hat X^{i,j})$ for all $0\leq i\leq j < l$, 
		\item $\hat e_{X^{i,j}}\left(- \circ \ol{\beta^{i,j}}\right) = \hat e_{X^{i+1,j}}\left(\ol{\hat \alpha^{i,j}} \circ -\right)$ on $\Hom_{\stab \F}(X^{i+1, j}, \hat X^{i,j})$ for all $0\leq i < j \leq l$.
	\end{itemize}
	For $j \in \{0, \dots, l\}$, set $\tilde X^{0,j} := \hat X^{0,j}$ and $\tilde \beta^{0,j} := \hat \beta^{0, j}$ if $j<l$ to construct \eqref{diag: Serre-pushout} by pushouts and a choice of admissible monics $\monic{\tilde \alpha^{j,j} \colon \tilde X^{j,j}}{I^j =: \tilde X^{j+1,j} \in \Inj(\F)}$ for $j \in \{0, \dots, l-1\}$. It remains to establish isomorphisms of functors $\eta_{X^{i,j}} \colon h_{\tilde X^{i,j}} \cong \left( h^{X^{i,j}} \right)^\ast$ such that 
	\begin{itemize}
		\item $e_{X^{i,j}}\left(- \circ \ol{\alpha^{i,j}}\right) = e_{X^{i,j+1}}\left(\ol{\tilde \beta^{i,j}} \circ -\right)$ on $\Hom_{\stab \F}(X^{i, j+1}, \tilde X^{i,j})$ for all $0\leq i\leq j < l$,
		
		\item $e_{X^{i,j}}\left(- \circ \ol{\beta^{i,j}}\right) = e_{X^{i+1,j}}\left(\ol{\tilde \alpha^{i,j}} \circ -\right)$ on $\Hom_{\stab \F}(X^{i+1, j}, \tilde X^{i,j})$ for all $0\leq i < j \leq l$.
	\end{itemize}
	We proceed inductively: Set $\eta_{X^{0,j}} := \hat \eta_{X^{0,j}}$ for $j \in \{0,\dots,l\}$. Then $e_{X^{0,j}} = \hat e_{X^{0,j}}$ and the required properties hold. To construct $\eta_{X^{i+1,j+1}}$ for $i,j \in \{0, \dots, l-1\}$ with $i \leq j$, suppose that $\eta_{X^{i',j'}}$ with the required properties is defined whenever $i' \leq i$ or $i'=i+1$ and $j' \leq j$. We may then assume that $\hat X^{i,j} = \tilde X^{i,j}$, $\hat X^{i,j+1} = \tilde X^{i,j+1}$, $\hat X^{i+1,j} = \tilde X^{i+1,j}$, $\ol{\hat \alpha^{i,j}} = \ol{\tilde \alpha^{i,j}}$, and $\ol{\hat \beta^{i,j}} = \ol{\tilde \beta^{i,j}}$. To complete the induction, apply \Cref{lem: square-rep} to the homotopy cartesian squares
	\[
		\begin{tikzcd}[sep={22.5mm,between origins}]
			X^{i,j} \ar[r, "\ol{\alpha^{i,j}}"] \ar[d, "\ol{\beta^{i,j}}"] \ar[rd, phantom, "\square"] & X^{i,j+1} \ar[d, "\ol{\beta^{i,j+1}}"] & \tilde X^{i,j} \ar[r, "\ol{\tilde \beta^{i,j}}"] \ar[d, "\ol{\tilde \alpha^{i,j}}"] \ar[rd, phantom, "\square"] & \tilde X^{i,j+1} \ar[d, "\ol{\hat \alpha^{i,j+1}}"] \\
			X^{i+1, j} \ar[r, "\ol{\alpha^{i+1,j}}"] & X^{i+1,j+1}, & \tilde X^{i+1, j} \ar[r, "\ol{\hat \beta^{i+1,j}}"] & \tilde X^{i+1,j+1}
		\end{tikzcd}
	\]
	in $\stab \F$, see \Cref{prp: Buehler2.12} and \Cref{lem: ses-triangle}. For $i=j$, note that both $X^{j+1,j}$ and $\tilde X^{j+1,j}$, and the associated isomorphism of functors $h_{\tilde X^{j+1,j}} \cong \left(h^{X^{j+1,j}}\right)^\ast$ are zero.
\end{proof}

\begin{lem} \label{lem: square-rep}
	In a triangulated category $\T$, hom-finite over a field, consider two homotopy cartesian squares:
	\[
		\begin{tikzcd}[sep={17.5mm,between origins}]
			A \ar[r, "a"] \ar[d, "c"] \ar[rd, phantom, "\square"] & B \ar[d, "b"] & \tilde A \ar[r, "\tilde a"] \ar[d, "\tilde c"] \ar[rd, phantom, "\square"] & \tilde B \ar[d, "\tilde b"] \\
			C \ar[r, "d"] & D & \tilde C \ar[r, "\tilde d"] & \tilde D
		\end{tikzcd}
	\]
	Suppose that there are isomorphisms of functors $\eta_A \colon h_{\tilde A} \cong \left(h^A\right)^\ast$, $\eta_B \colon h_{\tilde B} \cong \left(h^B\right)^\ast$, $\eta_C \colon h_{\tilde C} \cong \left(h^C\right)^\ast$, and $\hat \eta_D \colon h_{\hat D} \cong \left(h^D\right)^\ast$, where $\hat D \in \T$, such that
	\begin{itemize}
		\item $e_A(-\circ a) = e_B(\tilde a \circ -)$ on $\Hom(B, \tilde A)$,
		\item $e_A(-\circ c) = e_C(\tilde c \circ -)$ on $\Hom(C, \tilde A)$.
	\end{itemize}
	Then there is an isomorphism of functors $\eta_D \colon h_{\tilde D} \cong \left(h^D\right)^\ast$ such that
	\begin{itemize}
		\item $e_B(-\circ b) = e_D(\tilde b \circ -)$ on $\Hom(D, \tilde B)$,
		\item $e_C(-\circ d) = e_D(\tilde d \circ -)$ on $\Hom(D, \tilde C)$.
	\end{itemize}
\end{lem}

\begin{proof}
	Set $\hat e_D := e_{\hat \eta_D} \in \left(h^D_{\hat D}\right)^\ast$ and consider the morphisms $\hat b \colon \tilde B \to \hat D$ and $\hat d \colon \tilde C \to \hat D$ such that $e_B(-\circ b) = \hat e_D(\hat b \circ -)$ on $\Hom(D, \tilde B)$ and $e_C(-\circ d) = \hat e_D(\hat d \circ -)$ on $\Hom(D, \tilde C)$, see \Cref{lem: rep-comp}. The second homotopy cartesian square and {\cite[Prop.~3.3]{BK89}} yield an isomorphism
	\[
		\begin{tikzcd}[sep={22.5mm,between origins}, ampersand replacement=\&]
			\tilde A \ar{r}{\begin{pmatrix} \tilde a \\ - \tilde c \end{pmatrix}} \ar[d, equal] \& \tilde B \oplus \tilde C \ar{r}{\begin{pmatrix} \tilde b & \tilde d \end{pmatrix}} \ar[d, equal] \& \tilde D \ar[d, "f", "\cong"', dashed] \\
			\tilde A \ar{r}{\begin{pmatrix} \tilde a \\ - \tilde c \end{pmatrix}} \& \tilde B \oplus \tilde C \ar{r}{\begin{pmatrix} \hat b & \hat d \end{pmatrix}} \& \hat D
		\end{tikzcd}
	\]
	of distinguished triangles in $\T$. Define $\eta_D$ as the composition
	\[
		\begin{tikzcd}[sep={17.5mm,between origins}]
			h_{\tilde D} \ar[rr, dashed, "\eta_D", "\cong"'] \ar[rd, "h_f"', "\cong"] && \left(h^D\right)^\ast & \id_{\tilde D} \ar[rr, <->] \ar[rd, <->] && \hat e_D(f \circ -) \mathrlap{\; = e_D} \\
			& h_{\hat D}, \ar[ru, "\hat \eta_ D"', "\cong"] &&& f. \ar[ru, <->]
		\end{tikzcd}
	\]
	We have
	\begin{itemize}
		\item $e_B(-\circ b) = \hat e_D(\hat b \circ -) = \hat e_D(f \tilde b \circ -) = e_D(\tilde b \circ -)$ on $\Hom(D, \tilde B)$,
		\item $e_C(-\circ d) = \hat e_D(\hat d \circ -) = \hat e_D(f \tilde d \circ -) = e_D(\tilde d \circ -)$ on $\Hom(D, \tilde C)$,
	\end{itemize}
	as desired.
\end{proof}


\printbibliography

\end{document}